\documentclass[11pt]{amsart}

\usepackage{epigamath}

\usepackage[english]{babel}

\numberwithin{equation}{section}

\usepackage[shortlabels]{enumitem}
\setlist[enumerate,1]{label={\rm(\roman*)}, ref={\rm\roman*}}

\usepackage{multirow}
\usepackage{wasysym}

\usepackage{graphicx}
\usepackage{verbatim}
\usepackage{tikz}

\usepackage{amsfonts}
\usepackage{amssymb,amsmath,amsxtra,latexsym,amscd,amsthm,systeme}

\usepackage{epstopdf}
\usepackage{float}
\usepackage{pgfplots}
\usepackage{listings}
\usepackage{longtable}
\usepackage{mathrsfs} % check compatibility with epiga

\usepackage{mathtools}
\usepackage{tikz-cd}
\usepackage{diagbox}

\theoremstyle{definition}
\newtheorem{definition}{Definition}[section]

\theoremstyle{plain}
\newtheorem{proposition}[definition]{Proposition}
\newtheorem{theorem}[definition]{Theorem}
\newtheorem{corollary}[definition]{Corollary}
\newtheorem{lemma}[definition]{Lemma}
\newtheorem{conjecture}{Conjecture}

\theoremstyle{remark}
\newtheorem{remark}[definition]{Remark}
\newtheorem{example}[definition]{Example}

\newtheorem{innercustomclaim}{Claim}
\newenvironment{claim}[1]{\renewcommand\theinnercustomclaim{#1}\innercustomclaim}{\endinnercustomclaim}

\newcommand{\R}{\mathbb{R}}
\newcommand{\Z}{\mathbb{Z}}
\newcommand{\C}{\mathbb{C}}

\newcommand{\cp}{\mathbb{C}P^2}
\newcommand{\ppp}{\mathbb{C}P^1\times \mathbb{C}P^1\times \mathbb{C}P^1}
\newcommand{\cpp}{\mathbb{C}P^3}

\newcommand{\pp}{\mathbb{C}P^1\times \mathbb{C}P^1}

\newcommand{\rprprp}{\mathbb{R}P^1\times \mathbb{R}P^1\times \mathbb{R}P^1}
\newcommand{\ssrp}{\mathbb{S}^2\times \mathbb{R}P^1}

\DeclareMathOperator{\Crit}{Crit}
\DeclareMathOperator{\GW}{GW}
\DeclareMathOperator{\ind}{ind}
\DeclareMathOperator{\Pic}{Pic}
\DeclareMathOperator{\Sing}{Sing}
\DeclareMathOperator{\SO}{SO}
\DeclareMathOperator{\Pin}{Pin}
\DeclareMathOperator{\Spin}{Spin}
\DeclareMathOperator{\rank}{rank}
\DeclareMathOperator{\ev}{ev}
\renewcommand{\Im}{\operatorname{Im}}
\renewcommand{\sp}{\operatorname{sp}}
\newcommand{\Sym}{\operatorname{Sym}}
\newcommand{\PD}{\operatorname{PD}}
\newcommand{\pt}{\mathrm{pt}}

\newcommand{\sing}{\mathrm{sing}}

\newcommand{\lra}{\longrightarrow}

\def\lowcong{\vbox to 0pt{\vss\hbox{$\scriptstyle\cong$}\vskip-1.5pt}}

\newcommand{\supth}[1]{\ensuremath{#1^{\mathrm{th}}}}

\EpigaVolumeYear{10}{2026} \EpigaArticleNr{4} \ReceivedOn{June 26, 2023}
\InFinalFormOn{September 8, 2025}
\AcceptedOn{October 3, 2025}

\title{Gromov--Witten and Welschinger invariants of del Pezzo varieties}
\titlemark{Gromov--Witten and Welschinger invariants of del Pezzo varieties}

\author{Thi Ngoc Anh Nguyen}
\address{Faculty of Fundamental Sciences, Phenikaa University, Nguyen Trac Street, Duong Noi Ward, Hanoi 12116, Viet Nam}
\email{anh.nguyenthingoc3@phenikaa-uni.edu.vn}

\authormark{T.\,N.\,A.~Nguyen}

\AbstractInEnglish{In this paper, we establish formulas for computing genus $0$ Gromov--Witten and Welschinger invariants for certain del Pezzo varieties of dimension $3$ by comparing them to formulas for dimension~$2$. These formulas generalize the computations  for $3$-dimensional projective space given by E.~Brugall\'e  and P.~Georgieva.}


\acknowledgement{The author would like to express her gratitude to the Henri Lebesgue Center and Phenikaa University for their financial support.}

\begin{document}

%%%%%%%%%%%%%%%%%%%%%%%%%%%%%%%
% Title page
%%%%%%%%%%%%%%%%%%%%%%%%%%%%%%%

%\removeabove{}
%\removebetween{}
%\removebelow{}

\maketitle

\begin{prelims}

\DisplayAbstractInEnglish

\bigskip

\DisplayKeyWords

\medskip

\DisplayMSCclass

\end{prelims}

%%%%%%%%%%%%%%%%%%%%%
% Table of Contents
%%%%%%%%%%%%%%%%%%%%%

\newpage

\setcounter{tocdepth}{1}

\tableofcontents

%%%%%%%%%%%%%%%%%%%%%
% Content begins here
%%%%%%%%%%%%%%%%%%%%%

\section{Introduction and statement of the main results} \label{intro}

In 2016, E.~Brugall\'e and P.~Georgieva achieved a complete calculation of the genus $0$ Gromov--Witten and Welschinger invariants for $\cpp$ (see \cite{brugalle2016pencils}),  establishing a correspondence between the invariants of $\mathbb{C}P^3$ and those of $\mathbb{C}P^1 \times \mathbb{C}P^1$.  The idea comes from the observation of J.~Koll\'ar, who demonstrated that pencils of quadrics in $\mathbb{C}P^3$ can be utilized to evaluate certain enumerative invariants of $\mathbb{C}P^3$ (see \cite{kollar2015examples}). In this paper, we extend the results established in \cite{brugalle2016pencils} to other high-degree del Pezzo varieties~$X$. In this context, the occurrence of monodromy phenomena necessitates a more sophisticated approach to the enumeration of curves on pencils of surfaces. Moreover, the enumeration of real curves requires the incorporation of additional structures, such as spinor states, in order to accurately determine and compare the signs associated with real curves on varieties of different dimensions. More precisely, we compare the genus $0$ Gromov--Witten (respectively, Welschinger) invariants of $3$-dimensional (respectively, real $3$-dimensional)
del Pezzo varieties with those of non-singular (respectively, real non-singular) 
surfaces that realize half of the first Chern class of $X$, that is,  non-singular (respectively, real non-singular) surfaces in the linear system $|-\frac{1}{2}K_X|$, where $K_X$ denotes the canonical divisor of $X$. 
According to Proposition~\ref{Prop-delPezzo.surfaces}, these surfaces are themselves del Pezzo surfaces. 
The method employed in this work  (see Proposition~\ref{Prop-why.degree.is.678}) confines our attention to $3$-dimensional del Pezzo varieties of degree $6$, $7$ or $8$. In 2018,  A.~Zahariuc  adopted a similar approach to study  genus $0$ Gromov--Witten invariants  of $3$-dimensional del Pezzo varieties of degree $2$, $3$, $4$ and $5$ (see \cite{zahariuc2018rational}). It would be interesting to investigate possible connections between his formula and the one obtained in Theorem~\ref{Theorem1-GW}  of the present work. More generally, we hope to further explore relations between the Gromov--Witten and Welschinger invariants of (real) del Pezzo varieties of dimension $2$ and $3$ for arbitrary degree.

We begin by introducing some notation and summarizing the main results of this paper.

\subsubsection*{Some notations}
Throughout this text, the lower dot ``$.$'' denotes the intersection pairing between two homology classes, while the centered dot ``$\cdot$'' denotes the cup product.

\subsection{Introduction and motivation}
A Fano variety of dimension $n$, first introduced by G.~Fano in the 1930s, is a non-singular complex projective variety $X_n$ of complex dimension $n$ whose anti-canonical divisor $-K_{X_n}$ is ample. The positive integer $r>0$ for which $-K_{X_n}=rH$, where $H$ is a primitive divisor in $\Pic(X_n)$, is called the \textit{index of\, $X_n$} and is denoted by $\ind(X_n)$. An $n$-dimensional Fano variety of index divisible by $n-1$ is called an \textit{$n$-dimensional del Pezzo variety} for every $n\geq 2$. All $2$-dimensional Fano varieties are del Pezzo surfaces.  Non-singular Fano threefolds were essentially classified in the 1980s by V.\,A.~Iskovskikh and later refined by S.~Mukai. For higher dimensions, however, only partial results of classification are known, with important contributions of T.~Fujita on polarized varieties (see 
 \cite{iskovskikh2010algebraic}). The degree of an $n$-dimensional del Pezzo variety $X_n$, denoted by $\deg(X_n)$, is the integer given by $(\frac{1}{n-1})^n(-K_{X_n})^n$. It should be noted that the degree of a $3$-dimensional del Pezzo variety does not exceed $8$  (see \cite[Chapters~2 and~3]{iskovskikh2010algebraic}).

We consider a rational curve realizing a homology class $\beta\in H_2(X_n;\Z)$ as the image of a morphism $f\colon \mathbb{C}P^1\rightarrow X_n$ such that $f_*([\mathbb{C}P^1])=\beta$. Suppose that $n-1$ divides $\left(c_1(X_n)\cdot\beta+n-3\right)$. We denote by $k_\beta$ the integer $\frac{1}{n-1} \left(c_1(X_n)\cdot\beta+n-3\right)$. In general, Gromov--Witten invariants are defined using the \emph{virtual fundamental class} on the moduli space of stable maps (see \cite{KontsevichManin}). For del Pezzo surfaces and threefolds, however, the genus $0$ Gromov--Witten invariants are enumerative (see \cite{Vakil,Coskun,iskovskikh2010algebraic}). We have the following definition.

\begin{definition} \label{definitionGW}
The \emph{genus $0$ Gromov--Witten invariant of an  $n$-dimensional del Pezzo variety $X_n$ $(n\in \{2,3\})$ of homology class $\beta\in H_2(X;\Z)$},  denoted by $\GW_{X_n}(\beta)$, is the number of irreducible rational curves representing the homology class $\beta$ and passing through a generic configuration $\underline{x}_\beta$ of $k_\beta$ points in $X_n$.
\end{definition}

We denote by $(X_n,\tau)$ a real $n$-dimensional del Pezzo variety, that is, an $n$-dimensional del Pezzo variety~$X_n$ equipped with an anti-holomorphic involution  $\tau$, referred to as a real structure. The fixed locus of this involution is called the real part of $X_n$ and is denoted by $\R X_n$. An algebraic curve in  $(X_n,\tau)$  is said to be  \textit{real} if it is preserved by the real structure $\tau$. A \textit{real} configuration of points in  $(X_n,\tau)$ is a configuration of points that are either real or exchanged by the real structure $\tau$. Denote  by
$$ H_2^{-\tau}(X_n;\Z)=\{\beta \in H_2(X_n;\Z): \tau_*(\beta)=-\beta \}$$
the group of $\tau$-anti-invariant classes.
Let $\beta\in  H_2^{-\tau}(X_n;\Z)$ and let $\underline{x}_\beta$ be a real  generic configuration of $k_\beta$ points in $X_n$ consisting of $r$ real points and $l$ pairs of complex conjugate points, where $k_\beta=r+2l$. 

For each real rational curve in  $(X_n,\tau)$ passing through a real configuration of points, where the number of real points is fixed, one can associate a sign $\pm 1$, known as the \textit{Welschinger's sign} (see \cite{welschinger2005fourfolds,welschinger2005spinor}).  When these curves are counted with their associated signs, the resulting total is an invariant that does not depend on the choice of real configurations, is called the\textit{ Welschinger invariant} and is defined in the following definition. 

\begin{definition} \label{definitionW}
Let $(X_n,\tau)$ denote a real  $n$-dimensional ($n\in \{2,3\}$) del Pezzo variety whose real part is $\R X_n$.
The \emph{genus $0$ Welschinger invariant of\, $(X_n,\tau)$ of homology class $\beta\in H_2^{-\tau}(X_n;\Z)$} is the number of irreducible rational curves, counted with the Welschinger's sign (see Definitions~\ref{def-sign.curve.w.node.dim2} and~\ref{def-spinor3dim}), representing the homology class $\beta$ and passing through a real generic configuration $\underline{x}_\beta$ of $k_\beta$ points in $X_n$ consisting of $l$ pairs of complex conjugate points.

In particular, if $\R X_3$ is orientable with orientation $\mathfrak{o}_{X_3}$ and is equipped with  a $\Spin_3$-structure $\mathfrak{s}_{X_3}$, we denote this invariant by $W_{\R X_3}^{\mathfrak{s}_{X_3}, \mathfrak{o}_{X_3}}(\beta,l)$.

When $n=2$, this invariant is simply denoted by  $W_{\R X_2}(\beta,l)$.
\end{definition}
 
Welschinger invariants, introduced by J.-Y.~Welschinger in the early 2000s, are regarded as the real counterpart of the aforementioned Gromov--Witten invariants. For real rational surfaces (such as real del Pezzo surfaces), Welschinger's original papers \cite{welschinger2005fourfolds,welschinger2005spinor} gave some explicit computations, which were later generalized and approached by tropical geometry techniques (notably by G.~Mikhalkin \cite{mikhalkin2005enumerative}, I.~Itenberg--V.~Kharlamov--E.~Shustin \cite{itenberg2009welschinger}, E. Brugall\'e \cite{brugalle2007enumeration}, E.~Brugall\'e--G.~Mikhalkin using floor di\-a\-grams \cite{brugalle2015floor}, X.~Chen--A.~Zinger using  WDVV-type relations \cite{chen2021wdvv} and others).  In contrast, computational techniques for determining Welschinger invariants of threefolds remain comparatively limited. Some specific cases of threefold invariants have been computed in \cite{brugalle2007enumeration,georgieva2017enumeration,chen2021wdvv}. In particular, the case of convex algebraic threefolds (including Fano threefolds like $\mathbb{C}P^3$) was first addressed by J.-Y.~Welschinger (in \cite{welschinger2005spinor}), who defined these invariants in higher dimensions by means of spinor states and provided computations for small degrees. Further developments of explicit computations  were investigated for $\mathbb{C}P^3$ by E.~Brugall\'e and P.~Georgieva \cite{brugalle2016pencils}, for a blow-up of $\mathbb{C}P^3$ by Y.~Ding \cite{ding2020remark}) both using the observations by J.~Koll\' ar in \cite{kollar2015examples}.  Generally, explicit computation in dimension $3$ requires careful attention to Spin- or $\Pin^\pm$-structures, and the area remains an active research field, with connections to tropical geometry, degeneration techniques and the study of the sharpness of lower bounds in real enumerative geometry.

\subsection{Main results}

Throughout the remainder of this paper, unless otherwise stated, we denote by $X$ a $3$-dimensional del Pezzo variety. Let $\Sigma\subset X$ be a non-singular surface  in the linear system $|-\frac{1}{2}K_X|$, and let $\psi\colon H_2(\Sigma;\mathbb{Z})\rightarrow H_2(X;\mathbb{Z})$ denote the homomorphism induced by the inclusion $\Sigma \hookrightarrow X$. We have the following theorem.

\begin{theorem}\label{Theorem1-GW} 
Let $X$ denote a $3$-dimensional del Pezzo variety such that $\ker(\psi) \cong \Z$, where $\Sigma$ and $\psi$ are as described above. Let $S$ be a generator of\, $\ker (\psi)$.
For every homology class $d\in H_2(X;\Z)$, one has
$$\GW_X(d)=\frac{1}{2}\sum_{D\in \psi^{-1}(d)}(D.S)^2 \GW_{\Sigma}(D).$$
\end{theorem}

Theorem~\ref{Theorem1-GW}  generalizes  \cite[Theorem 2]{brugalle2016pencils}. This result enables the computation of  genus $0$ Gromov--Witten invariants for almost  all $3$-dimensional  del Pezzo varieties of degree $6$, $7$ and $8$.

Now suppose that $(X,\tau)$ is a $3$-dimensional del Pezzo variety equipped with a real structure $\tau$. Assume further that $\R X$ is orientable with orientation $\mathfrak{o}_X$ and is equipped with a $\Spin_3$-structure $\mathfrak{s}_X$.  Let the pair $(\Sigma,\tau|_\Sigma) \in |-\frac{1}{2}K_X|$  be a non-singular real del Pezzo surface. Let $\psi\colon H_2(\Sigma;\mathbb{Z})\to H_2(X;\mathbb{Z})$ be as in the Gromov--Witten setting. Define $g(D)=\frac{1}{2}(K_\Sigma . D +D^2+2) \in \Z$, and let $\epsilon\colon H_2(\Sigma;\mathbb{Z})\to \mathbb{Z}_2$ be the map     given by
$$\epsilon(D)=
\begin{cases}
  c_1(\Sigma)\cdot D \mod 2 & \text{ if } \R f^* T \Sigma /T \C P^1 \text{ realizes the isotopy class } E^+,\\
  c_1(\Sigma)\cdot D-1 \mod 2 &\text{ if } \R f^* T \Sigma /T \C P^1\text{ realizes the isotopy class } E^-,
  \end{cases}$$
 where $f\colon \C P^1\to \Sigma\subset X$ is a  real immersion such that $f_*[\C P^1]=D\in H_2(\Sigma;\mathbb{Z})$, and  $E^\pm$ are two isotopy classes of the real part of a holomorphic line subbundle of degree $c_1(\Sigma)\cdot D-2$ of the rank $2$ holomorphic normal  bundle $f^*T X/T \C P^1$ (see Section~\ref{Sect-spinor.dim.3.with.normal.bundle} and Definition~\ref{def-epsilon.function}). Let  $s^{\mathfrak{s}_X|_\Sigma}\colon H_1(\R \Sigma;\Z_2)\to \Z_2$
be the quasi-quadratic enhancement in one-to-one correspondence with the $\Pin^-_2$-structure $\mathfrak{s}_{X|\Sigma}$ over $\R \Sigma$. Such a $\Pin^-_2$-structure  is restricted from the $\Spin_3$-structure $\mathfrak{s}_X$ over $\R X$ (see the construction of this correspondence in Section~\ref{subsect-construction_quasiquad_enhancement}, in particular Proposition~\ref{prop-quasi-quadratic.enhancement.def.prop} and Equation (\ref{eq-RelationW2W2})). Let $\rho\colon  H_2^{-\tau}(\Sigma;\Z) \to H_1(\R \Sigma;\Z_2)$ be the homomorphism given by the homology class realized by the real part of the $\tau$-anti-invariant representative of $D$. We have the following theorem.

\begin{theorem}\label{Theorem2-W}
Let $((X,\tau),\mathfrak{o}_X,\mathfrak{s}_X)$ denote a real $3$-dimensional del Pezzo variety, where  $\mathfrak{o}_X$ is an orientation and $\mathfrak{s}_X$ is a $\Spin_3$-structure on $\R X$. Suppose that $\ker(\psi) \cong \Z$, where $\psi$ is as described above. Let $S$ be a generator of\, $\ker(\psi)$.  Suppose $(\Sigma,\tau|_{\Sigma})\in |-\frac{1}{2}K_X|$ is  such that $S$ is $\tau|_{\Sigma}$-anti-invariant.  For every $d\in H_2^{-\tau}(X;\Z)$ and $0\leq l< \frac{1}{2}k_d$,  one has
$$W^{\mathfrak{s}_X,\mathfrak{o}_X}_{\R X}(d,l)=\frac{1}{2}\sum_{D\in \psi^{-1}(d)}(-1)^{\epsilon(D)+g(D)+s^{\mathfrak{s}_X|_\Sigma}(\rho D)} |D.S| W_{\R \Sigma} (D,l).$$
\end{theorem}

Theorem~\ref{Theorem2-W} generalizes \cite[Theorem 1]{brugalle2016pencils}. As a consequence, this theorem enables the computation of  genus $0$ Welschinger invariants for almost all real $3$-dimensional del Pezzo varieties of degree $6$, $7$ and $8$. Since the Welschinger invariants count real curves with an assigned sign  $\pm 1$,  their vanishing does not necessarily imply the non-existence of such real curves. Consequently, this gives rise to the sharpness and vanishing phenomena of the Welschinger invariants, which concerns whether the lower bounds provided by these invariants are actually attained. The first examples demonstrating the sharpness of vanishing genus $0$ Welschinger invariants
in projective threespace were explored by G.~Mikhalkin for rational curves of degree~$4$, and subsequently by J.~Koll\'ar for every even degree; see \cite{kollar2015examples}. We extend J.~Koll\'ar's results to certain classes of $3$-dimensional del Pezzo varieties, as established in the following theorem.

\begin{theorem}\label{Theorem3-W}
Let $((X,\tau),\mathfrak{o}_X,\mathfrak{s}_X)$ denote a real $3$-dimensional del Pezzo variety as in Theorem~\ref{Theorem2-W}. Let the class $d\in H_2^{-\tau}(X;\Z)$ be  such that $D.S$ is even for every $D\in \psi^{-1}(d)$. Let $k_d$ be the associated integer of $d$. Then, there exists a real configuration of\, $k_d$ points, containing at least one real point, through which no real irreducible rational curve representing the homology class $d$ passes.
\end{theorem}

 The structure of this paper is as follows. In Section~\ref{Section-Pencils}, we recall certain properties of complex and real pencils of surfaces in the linear system $|-\frac{1}{2}K_X|$, and we describe the monodromy action to such pencils. We refer to  the \textit{$1$-dimensional problem} (respectively, the \textit{$1$-dimensional real problem})
as the problem of enumeration on pencils of surfaces (respectively, real pencils of surfaces). Solving this problem contributes to the connecting terms in Theorems~\ref{Theorem1-GW} and~\ref{Theorem2-W}. In Section~\ref{Section-ratcurve}, we investigate the properties of rational curves both on $3$-dimensional del Pezzo varieties $X$ and on non-singular del Pezzo surfaces $\Sigma\in |-\frac{1}{2}K_X|$. As a consequence, we establish a comparison between the genus $0$ Gromov--Witten invariants of $X$ and those of $\Sigma$. Section~\ref{Sect-Welschinger} is devoted to the problem of  Welschinger's signs $\{\pm 1\}$ arising in the enumeration of real curves on real surfaces and real $3$-dimensional varieties. We then propose a comparison of these signs via the quasi-quadratic enhancement corresponding to a particular $\Pin_2^-$-structure on the real part of the algebraic surfaces.
In Section~\ref{Sect-proofs}, we present the proofs of Theorems~\ref{Theorem1-GW}, \ref{Theorem2-W} and~\ref{Theorem3-W}.
In the final section, we first apply Theorem~\ref{Theorem1-GW} in several specific cases (Theorems~\ref{Theorem-GW.in.deg.8},~\ref{Theorem-GW.in.deg.7},~\ref{Theorem-GW.in.deg.6}), providing explicit calculation tables of  genus~$0$ Gromov--Witten invariants for $\C P^1\times\C P^1\times \C P^1$, generated using a Maple program. Subsequently, we apply Theorems~\ref{Theorem2-W} and~\ref{Theorem3-W} to several particular cases (Theorems~\ref{Theorem-W.in.deg.8},~\ref{Theorem-W.in.deg.7},~\ref{Theorem-W.in.deg.6.1},~\ref{Theorem-W.in.deg.6.2}, Propositions~\ref{prop-application.deg8}, \ref{prop-application.deg7}, \ref{prop-application.deg6.1},~\ref{prop-application.deg6.2}). Finally, we exhibit the explicit calculation tables of genus $0$ Welschinger invariants for $\C P^3\sharp \overline{\C P^3}$ with the standard real structure and for  $\C P^1\times\C P^1\times \C P^1$ with  two distinct real structures, also computed using a Maple program. 

\subsection*{Acknowledgments.} The present article forms part of the author's Ph.D. thesis, written at Nantes University under the supervision of Erwan Brugall\'e. The author is sincerely grateful to him for his invaluable support, insightful discussions, and numerous suggestions throughout the development and completion of this work. She would like to thank Xujia Chen for her helpful comments and computational assistance, as well as Duc Pham for his valuable programming skills, which were utilized in the application sections.

\section{Pencils of surfaces in the linear system \texorpdfstring{$\boldsymbol{|-\frac{1}{2}K_X|}$}{|-(1/2)K\textunderscore X|}} \label{Section-Pencils}

Let $\mathcal{Q}$ denote a pencil of surfaces in the linear system $|-\frac{1}{2}K_X|$. Let $E$ be its base locus, and let  $\Sigma$ be a general member of the pencil $\mathcal{Q}$. We first establish that the base locus $E$ is an elliptic curve (Proposition~\ref{Prop-E.is.elliptic}) and that $\Sigma$ is a non-singular del Pezzo surface of the same degree as $X$ (see Proposition~\ref{Prop-delPezzo.surfaces}).  We then explain how the monodromy action is related to enumerative problems. In the following subsection, we provide the main ingredient for resolving the \textit{$1$-dimensional problem} of Theorem~\ref{Theorem1-GW} (see Proposition~\ref{Prop-degree.case.complex}). In the final subsection, by analyzing the properties of real pencils of surfaces, we present the main ingredient for solving the \textit{$1$-dimensional real  problem} of Theorem~\ref{Theorem2-W} (see Proposition~\ref{prop-corresponding.surface.picE}, Corollary~\ref{coro-degree.case.real}).

\subsection{Properties of pencils of surfaces in the linear system $\boldsymbol{|-\frac{1}{2}K_X|}$}

\begin{proposition} 
\label{Prop-E.is.elliptic}
 If the pencil $\mathcal{Q}$ is generic enough, then its base locus is a non-singular elliptic curve. Moreover, this curve realizes the homology class $\frac{1}{4}K^2_X$ in $H_2(X;\mathbb{Z})$.
\end{proposition}

\begin{proof}
 By Bertini's theorem (see \cite{griffiths1978principles}) and by \cite[Theorem 1.2]{vsokurov1980smoothness}, a general element $\Sigma \in |-\frac{1}{2}K_X|$ is a non-singular surface.  The base locus $E$ of the pencil $\mathcal{Q}$ realizes the homology class $[\Sigma].[\Sigma]$, that is, the self-intersection product of the homology class $[\Sigma]\in H_4(X;\Z)$. 
 By the adjunction formula, we have 
$$K_\Sigma =(K_X + [\Sigma] ).[\Sigma] = -[\Sigma]^2.$$

Thus, the curve $E$ realizes the homology class $-K_\Sigma$, or equivalently the homology class $\frac{1}{4}K^2_X$ in $H_2(X;\mathbb{Z})$. Since two general elements in the pencil have  traversal intersections, the curve  $E$ is in fact a non-singular representative of the class $-K_\Sigma$. By the adjunction formula, it is easily deduced that the genus of the curve $E$ is $1$. Hence, the proposition follows.
\end{proof}

The next proposition explains the interest of studying surfaces in the linear system $|-\frac{1}{2}K_X|$.

\begin{proposition}\label{Prop-delPezzo.surfaces}
Every non-singular surface $\Sigma\in |-\frac{1}{2}K_X|$ is a del Pezzo surface whose degree is equal to $\deg(X)$.
\end{proposition}

\begin{proof}
  By \cite[Corollary 1.11]{iskovskikh1977fano}, every non-singular surface in $ |-\frac{1}{2}K_X|$  is a del Pezzo surface. By definition, the degree of the del Pezzo surface $\Sigma$ is equal to $K_\Sigma^2$, while the degree of the del Pezzo threefold $X$ is given by $\frac{1}{8}(-K_X)^3$. The adjunction formula yields $K_{\Sigma}^2 = \frac{1}{4}K_X^2 \cdot [\Sigma]$. Since $[\Sigma] = -\frac{1}{2}K_X$, it follows that $\deg(\Sigma) = \deg(X)$.
\end{proof}

\begin{proposition} \label{prop-elliptic.implies.pencil}
Let $E$ be a non-singular elliptic curve that realizes the homology class $\frac{1}{4}K^2_X$ in $H_2(X;\mathbb{Z})$ and is contained in a non-singular surface $\Sigma\in |-\frac{1}{2}K_X|$. Then, $E$ is the base locus of a pencil of surfaces in the linear system $|-\frac{1}{2}K_X|$.
\end{proposition}

\begin{proof}
By the adjunction formula, the elliptic curve $E$ realizes the anti-canonical class $-K_\Sigma \in H_2(\Sigma; \mathbb{Z})$.

We claim that $E$ is the complete intersection of $\Sigma$ with another non-singular surface in $|-\frac{1}{2}K_X|$. It suffices 
to prove that the anti-canonical bundle of $\Sigma$ is the restriction of the line bundle $\mathcal{O}_X(-\frac{1}{2}K_X)$ to $\Sigma$. The isomorphism of these line bundles implies an isomorphism of their global sections, that is,
 $$H^0\left(X,\mathcal{O}_X\left(-\tfrac{1}{2}K_X\right)|_\Sigma\right)=H^0\left(\Sigma,\mathcal{O}_\Sigma\left(-K_\Sigma\right)\right).$$
 
It is clear that $H^0(\Sigma,\mathcal{O}_\Sigma(-K_\Sigma)) \subset H^0(X,\mathcal{O}_X(-\frac{1}{2}K_X)|_\Sigma)$ as a subspace. We will show that these two vector spaces have the same dimension.
 Consider the short exact sequence of sheaves
 $$0\longrightarrow \mathcal{O}_X\longrightarrow \mathcal{O}_X\left(-\tfrac{1}{2}K_X\right)\longrightarrow \mathcal{O}_X\left(-\tfrac{1}{2}K_X\right)|_\Sigma \longrightarrow 0$$
 that induces the following long exact sequence in cohomology:

 \begin{center}
\begin{tikzcd}
 0 \arrow[r]& H^0(X,\mathcal{O}_X) \arrow[r]
& H^0\left(X, \mathcal{O}_X\left(-\tfrac{1}{2}K_X\right)\right) \arrow[r]
\arrow[d, phantom, ""{coordinate, name=Z}]
&  H^0\left(X, \mathcal{O}_X\left(-\tfrac{1}{2}K_X\right)|_\Sigma\right) \arrow[dll,
"\delta", rounded corners, to path={ -- ([xshift=2ex]\tikztostart.east)
|- (Z) [near end]\tikztonodes
-| ([xshift=-2ex]\tikztotarget.west)
-- (\tikztotarget)}] \\
&H^1(X,\mathcal{O}_X) \arrow[r]
&  \cdots. \quad \quad \quad \quad \quad \quad \quad
&
\end{tikzcd}
\end{center}

We have the following facts (see \cite[Section~3.2]{iskovskikh2010algebraic}): 
\begin{itemize}
\item We have $\dim (H^0(X,\mathcal{O}_X))=1$ and $H^1(X,\mathcal{O}_X)=0$.
\item By the Riemann--Roch theorem,
 $$\dim H^0\left(X, \mathcal{O}_X\left(-\tfrac{1}{2}K_X\right)\right)=\left(-\tfrac{1}{2}K_X\right)^3+3-1=\deg (X) +2$$ 
 and 
 $$\dim H^0\left(\Sigma,\mathcal{O}_\Sigma\left(-K_\Sigma\right)\right)=\left(-K_\Sigma\right)^2+2-1=\deg (\Sigma) +1.$$ 
\end{itemize}
It follows from the long exact sequence and the vanishing $H^1(X, \mathcal{O}_X) = 0$ that
 $$\dim H^0\left(X, \mathcal{O}_X\left(-\tfrac{1}{2}K_X\right)|_\Sigma\right)=\dim H^0\left(X, \mathcal{O}_X\left(-\tfrac{1}{2}K_X\right)\right)-1=\deg (X)+1.$$
Since $\deg(X) = \deg(\Sigma)$, we have 
$$\dim H^0(\Sigma,\mathcal{O}_\Sigma(-K_\Sigma))=\dim  H^0\left(X,\mathcal{O}_X\left(-\tfrac{1}{2}K_X\right)|_\Sigma\right),$$
as required.
  \end{proof}

\begin{proposition} \label{Prop-why.degree.is.678}
  For every non-singular surface $\Sigma\in |-\frac{1}{2}K_X|$, there exists a surjective map
  $$\psi \colon  H_2(\Sigma;\mathbb{Z})\lra H_2(X;\mathbb{Z}),$$
  which is induced by the inclusion $\Sigma$ in $X$.
 
In particular, if\,  $X$ is one of\, $\{\C P^3,\; \C P^3\sharp \overline{ \C P^3},\;\C P^1\times \C P^1 \times \C P^1\}$, then $\ker(\psi)\cong \mathbb{Z}$. 
\end{proposition}

\begin{proof}
  One has a non-singular subvariety $\Sigma \subset X$ of codimension~$1$. Moreover, as a divisor, $\Sigma$ is ample since $-\frac{1}{2}K_X$ is ample, by the Lefschetz hyperplane theorem (see, for example, \cite[Theorem C.8]{eisenbud20163264}), so the map
$$ H_{\dim(\Sigma)}(\Sigma;\mathbb{Z})\lra H_{\dim(\Sigma)}(X;\mathbb{Z})$$
is surjective. Hence, the first statement follows. 

In particular, we have the following cases: 
\begin{itemize}
\item If $X=\C P^3$, then $\Sigma$ is a quadric surface in $\C P^3$. In this case, $H_2(X;\mathbb{Z})\cong \Z$ and $H_2(\Sigma;\mathbb{Z})\cong \Z^2$.
\item If $X=\C P^3\sharp \overline{ \C P^3}$, then $\Sigma$ is a quadric surface in $\C P^3$ blown up at one point. In this case, we have $H_2(X;\mathbb{Z})\cong \Z^2$ and $H_2(\Sigma;\mathbb{Z})\cong \Z^3$.
\item If $X=\C P^1\times \C P^1 \times \C P^1$, then $\Sigma$ is a quadric surface in $\C P^3$ blown up at two points. In this case, $H_2(X;\mathbb{Z})\cong \Z^3$ and $H_2(\Sigma;\mathbb{Z})\cong \Z^4$.
\end{itemize}
By the first isomorphism theorem for groups, we have $H_2(X;\mathbb{Z})\cong H_2(\Sigma;\mathbb{Z})/\ker(\psi)$. Therefore, if  $X$ is in $\{\C P^3, \C P^3\sharp \overline{ \C P^3} ,\C P^1\times \C P^1 \times \C P^1\}$ and $\Sigma \in  |-\frac{1}{2}K_X|$ is a non-singular surface, then  $\ker(\psi)\cong\mathbb{Z}$. 
\end{proof}

It should be noted that a generic pencil of del Pezzo surfaces in the linear system $|-\frac{1}{2}K_X|$ is a 
 Lefschetz pencil: All singular fibers are irreducible with a single ordinary double point. 

\begin{lemma} \label{lem-4.singulars}
Let $X$ be a non-singular del Pezzo threefold of index $2$ and of degree $d$ $(d\geq 2)$.  
Then a generic pencil of surfaces in the linear system $|-\tfrac{1}{2}K_X|$ has exactly
\[
N \;=\;- \chi_{\mathrm{top}}(X)+24-2d
\]
singular members.  

In particular, if\,  $X$ is one of\, $\{\C P^3,\; \C P^3\sharp \overline{ \C P^3},\;\C P^1\times \C P^1 \times \C P^1\}$, then $N=4$.
\end{lemma}

\begin{proof}
 Fix  a general pencil $\mathcal{Q}$ of surfaces in the linear system $|-\frac{1}{2}K_X|$. If $d\ge 2$, this linear system  is base-point-free.  Moreover, as previously shown, the base locus $E$ of the pencil $\mathcal{Q}$ is an elliptic curve of class $\frac{1}{4}K_X^2$. Let $\pi\colon \widetilde X:=\mathrm{Bl}_E X\to X$ be the blow-up of $X$ along the base locus $E$, and let
$f\colon \widetilde X\to\mathbb{C} P^1$ be the morphism induced by the pencil. Since the fibration $f$ is a Lefschetz fibration, the number of singular members of $\mathcal{Q}$ equals the number of critical points of $f$.

Let $L:=\mathcal O_X(-\frac{1}{2}K_X)$. The critical points of the Lefschetz fibration associated to a pencil in
$|L|$ are counted by the top Chern class of the first jet bundle $J^1(L)$ fitting into the exact sequence
\[
0 \;\longrightarrow\; \Omega_X^1 \otimes L \;\longrightarrow\; J^1(L) \;\longrightarrow\; L \;\longrightarrow\; 0.
\]
This means that
\[
\#\Crit(f)\;=\;\int_X c_3\left(J^1(L)\right).
\]
By the Whitney multiplicativity of total Chern classes, $c(J_1(L))=c(\Omega_X^1\otimes L)\cdot c(L)$. Hence, taking degree~$3$, we have $c_3(J_1(L))=c_1(L)\cdot c_2(\Omega_X^1\otimes L)+c_3(\Omega_X^1\otimes L)$.
The splitting principle applied on a threefold with $l:=c_1(L)$ yields
\[
c_3\left(J^1(L)\right)
\;=\;-c_3(X)+2lc_2(X)-3l^2c_1(X)+4l^3.
\]
So, the number of singular members in a generic pencil of $|-\frac{1}{2}K_X|$ is
\begin{equation*}
\#\Sing(\mathcal{Q})\;=\;-\chi_\mathrm{top}(X)+2L\cdot c_2(X)-3L^2\cdot c_1(X)+4L^3.
\end{equation*}
Since $X$ is a del Pezzo threefold of index $2$ and degree $d$, meaning that $c_1(X)=2L$ and $L^3=d$, the above equation can be simplified as
\begin{equation}\label{eq:discriminant}
\#\Sing(\mathcal{Q})\;=\;-\chi_\mathrm{top}(X)+2c_2(X)\cdot  \left(-\tfrac{1}{2}K_X\right)-2d.
\end{equation}
It remains to compute $c_2(X)\cdot (-\frac{1}{2}K_X)$ for every del Pezzo threefold. By Kodaira vanishing and the Riemann--Roch theorem on a  non-singular threefold, we have 
\[
h^0\left(X,\mathcal O_X\left(-\tfrac{1}{2}K_X\right)\right)\;=\;\chi\left(\mathcal O_X\left(-\tfrac{1}{2}K_X\right)\right)
\]
and
\[
\chi\left(\mathcal O_X\left(-\tfrac{1}{2}K_X\right)\right)
= \frac{\left(-\frac{1}{2}K_X\right)^3}{6}+\frac{c_2(X)\cdot \left(-\frac{1}{2}K_X\right)}{12}+\frac{\left(-\frac{1}{2}K_X\right)^2\cdot c_1(X)}{4}
+\frac{\left(-\frac{1}{2}K_X\right)\cdot c_1(X)^2}{12}+\chi(\mathcal O_X).
\]
With $c_1(X)=-K_X$ and $\chi(\mathcal O_X)=1$, we obtain
\[
h^0\left(X,\mathcal O_X\left(-\tfrac{1}{2}K_X\right)\right)\;=\; d + \frac{c_2(X)\cdot \left(-\frac{1}{2}K_X\right)}{12} + 1.
\]
On the other hand, for a del Pezzo threefold, one has $h^0(X,\mathcal O_X(-\frac{1}{2}K_X))=d+2$. This implies that $c_2(X)\cdot (-\frac{1}{2}K_X) \;=\; 12$.

Substituting $c_2(X)\cdot (-\frac{1}{2}K_X)=12$ into Formula \eqref{eq:discriminant}, we get
\[
\#\Sing(\mathcal{Q})\;=\;-\chi_\mathrm{top}(X)+24-2d,
\]
as required.
The values of $\chi_\mathrm{top}(X)$ and $\#\Sing(\mathcal{Q})$ determined by Equation (\ref{eq:discriminant}) are presented in Table~\ref{Tab:singular.surfaces} for $d\in \{6,7,8\}$.
\begin{table}[h]
\centering\renewcommand{\arraystretch}{1.2}
\begin{tabular}{|c|c|c|c|}
\hline
Degree $d$ & \emph{$X$} & $\#\Sing(\mathcal{Q})$ & $\chi_\mathrm{top}(X)$ \\
\hline
$6$ & $\mathbb{P}(T_{\mathbb{C}P^2}) \simeq (1,1)$-divisor in $\mathbb{C}P^2 \times \mathbb{C}P^2$ & $6$ & $6$ \\
    &$\mathbb{C}P^1 \times \mathbb{C}P^1 \times \mathbb{C}P^1$  & $4$ & $8$ \\
\hline
$7$ & $\mathbb{C}P^3\#\overline{\mathbb{C}P^3}$ & $4$ & $6$ \\
\hline
$8$ & $\mathbb{C}P^3$ & $4$ & $4$ \\
\hline
\end{tabular}
\caption{Topological Euler characteristic of a del Pezzo threefold of degree  $6$, $7$ or $8$ and the number of singular elements in a generic pencil of surfaces in $|-\frac{1}{2}K_X|$.} \label{Tab:singular.surfaces}
\end{table}
\end{proof}

%%%%%%%%%%%%%%%%%%%%%%%%%%%%%%%%%%%
\subsection{Monodromy} \label{Sect-Monodromy}
We employ the same notation as in the proof of Lemma~\ref{lem-4.singulars}. Let $\mathcal{Q}$ denote a Lefschetz pencil of surfaces in the linear system $|-\frac{1}{2}K_X|$, and let $f\colon \tilde{X} \to \mathbb{C}P^1$ be the corresponding Lefschetz fibration. We denote by $\Sigma$ a non-singular fiber over $t\in \C P^1$ and by $\Sigma_{t_i}$ a singular fiber over $t_i\in \C P^1$, for $i\in \{1,\ldots,4\}$. The monodromy transformation on $H_2(\Sigma; \mathbb{Z})$ that corresponds to the vanishing cycle $S_i$, is induced by traversing a path in the base of the fibration $f$ from a regular value $t$ to the critical value $t_i$; see \cite[Section 2.2.2]{nguyen2022real} for more details. This transformation is given by the Picard--Lefschetz formula as follows:
 \begin{align*}
     T_{(S_i)_*}\colon  H_2(\Sigma;\mathbb{Z})&\lra H_2(\Sigma;\mathbb{Z})\\
   D&\longmapsto D+(D.S_i)S_i.
 \end{align*}

It should be noted that, under the monodromy transformation, the vanishing cycle $S_i$ is mapped to its inverse, and its self-intersection number is $-2$.  The monodromy transformation does not depend on the choice of $S_i$ or $-S_i$; that is, $T_{{(-S_i)}_*}(D)=T_{(S_i)_*}(D)$. 

\begin{lemma} \label{Lemma-Vanishing.Cycle}
  Let $X$ be a $3$-dimensional del Pezzo variety such that
  $$\ker(\psi \colon H_2(\Sigma;\mathbb{Z})\to H_2(X;\mathbb{Z}))\cong \Z,$$
where $\Sigma\in |-\frac{1}{2}K_X|$ is a non-singular surface. 
Then, the group $\ker (\psi)$ is generated by a vanishing cycle $S_i$, for every $i\in \{1,\ldots,4\}$.
\end{lemma}

\begin{proof}
By the definition of a vanishing cycle $S_i\in H_2(\Sigma;\Z)$, it follows that $S_i\in \ker (\psi)$ for every $i\in \{1,\ldots,4\}$. Suppose that $S$ is a generator of $\ker (\psi)\cong \Z$. We can write $S_i=kS$, where $k\in \Z$. Since $S_i^2=-2$, we obtain $k^2=1$ and hence $k\in \{\pm 1\}$. As a result, $S_i=\pm S$ for every $i\in \{1,\ldots,4\}$.
\end{proof}

Now suppose  that $X$ is as in Lemma~\ref{Lemma-Vanishing.Cycle}. In this setting, we may identify each vanishing cycle $S_i$ with a generator $S$ of  $\ker(\psi)$, for any $i\in \{1,\ldots,4\}$, when studying monodromy transformations. From here on, 
we simply denote by $T$---instead of $T_{(S_i)_*}$---the monodromy transformation on $H_2(\Sigma;\Z)$ given by $T(D):=D+(D.S)S$. We define the following equivalence relation $\sim$ on $H_2(\Sigma;\Z)$ (where $\Sigma \in \mathcal{Q}$ is non-singular): 
 $$D_1\sim D_2 \quad\text{if and only if}\quad  D_2=D_1 \text{ or } D_2=D_1+(D_1.S)S.$$
 This relation is indeed an equivalence relation, as the monodromy transformation is an involution on the relevant set. In particular, it is symmetric.  
 We denote by $\overline{D}$  the equivalence class of $D$ in $H_2(\Sigma;\Z)/_\sim$.  For any pair of non-singular elements  $\Sigma$ and $\Sigma'$ in the pencil $\mathcal{Q}$, the corresponding Lefschetz fibration induces an isomorphism
$$H_2(\Sigma;\Z)/_\sim \;\cong H_2(\Sigma';\Z)/_\sim .$$

\begin{example}
   In the case $X=\cpp$, we consider the Lefschetz fibration on $\C P^1$ whose non-singular fibers are isomorphic to $\pp$ (see Figure~\ref{fig-monodromy}). 
   \begin{figure}
 \centering
\includegraphics[scale=0.4]{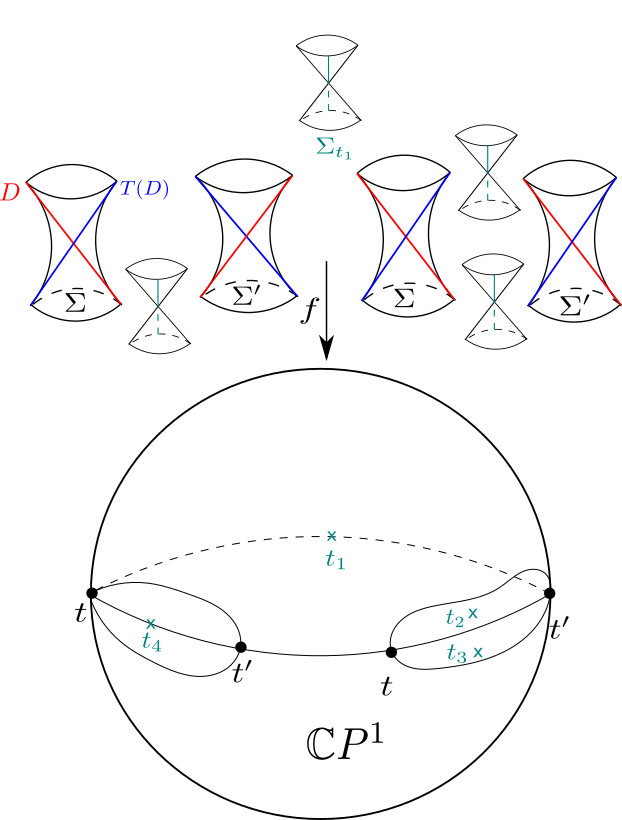}
\caption{Example of a monodromy
transformation where the general fibers $\Sigma$ are quadric surfaces and the singular ones are quadric cones in $\cpp$.}
\label{fig-monodromy}
\end{figure}
   Let $(L_1,L_2)=([\C P^1\times\{\pt\}],[\{\pt\}\times \C P^1])$ be a basis of $H_2(\pp;\Z)$. The vanishing cycle is $S=L_1-L_2\in H_2(\pp;\Z)$.  It is straightforward to verify that, for any $(a, b) \in \mathbb{Z}^2$, the homology classes $aL_1+bL_2$ and $bL_1+aL_2$ are related by the monodromy transformation, that is, $bL_1+aL_2=T(aL_1+bL_2)$.
Thus, the monodromy interchanges the coefficients of $L_1$ and $L_2$. In other words, two classes $aL_1 + bL_2$ and $bL_1 + aL_2$ are equivalent under the equivalence relation induced by monodromy.    Therefore, for any non-singular fiber $\Sigma$ in this fibration, the quotient $H_2(\Sigma;\Z)/_{\sim}$ can be identified  with $H_2(\pp;\Z)\, /\, \langle L_1 - L_2 \rangle$. 
\end{example}

     From here on, we will use interchangeably the notation for homology classes in $H_2(\Sigma;\mathbb{Z})$ and  divisor classes in $\Pic(\Sigma)$, where $\Sigma$ is a del Pezzo surface. Given a divisor $D \in \Pic(\Sigma)$, we denote by $D|_{E,\Sigma} \in \Pic_{D.E}(E)$ its restriction to $E$. Here, $\Pic_{k}(E)$ denotes the group of divisor classes on $E$ of degree $k$ for any $k\in \mathbb{Z}$.  Since $K_\Sigma.S=0$, for $K_\Sigma$ the anti-canonical divisor of any non-singular surface $\Sigma$ in the pencil $\mathcal{Q}$ and $[E]=-K_\Sigma$, it follows that  the restriction of the vanishing cycle to $E$  lies in $\Pic_0(E)$.

\subsection{Enumeration on pencils of surfaces} \label{Sect-Pencils}
Recall that $E$ is the base locus of a pencil of surfaces in the linear system $|-\frac{1}{2}K_X|$ and $S$ is a generator of the kernel of the map $\psi$ as described after Lemma~\ref{Lemma-Vanishing.Cycle}.  Since $S.E=0$, for every divisor $D\in \Pic(\Sigma)$, the restrictions $D|_{E,\Sigma}$ and $T(D)|_{E,\Sigma}$ lie in $\Pic_{D.E}(E)$. Given any divisor $D$ on $\Sigma$ and its image under the monodromy transformation $T(D)$, we define the following holomorphic map:
\begin{align*}
    \phi\colon   \mathcal{Q}& \longrightarrow \Pic_{2D.E}(E)\\
    \Sigma_t &\longmapsto  (D+T(D))|_{E,\Sigma_t}, 
\end{align*}
where $D + T(D)$ is viewed as a divisor on $\Sigma_t$ and $|_{E, \Sigma_t}$ denotes its restriction to $E$.

\begin{lemma} \label{lem-constant.map}
The map $\phi$ is constant.
\end{lemma} 

\begin{proof}
  The pencil $\mathcal{Q}$ is parametrized by $\mathbb{C}P^1$, which is a compact Riemann surface of genus $0$, while $\Pic_{2D.E}(E)$ is an elliptic curve, that is, a compact Riemann surface of genus $1$. 
Any holomorphic map from a surface of lower genus to one of higher genus must be constant unless it is branched (which is not the case here). Therefore, the lemma follows. 
\end{proof}
  
We denote by $\overline{D}_{\phi}$ the divisor $(D+T(D))|_{E,\Sigma_t}$ or, equivalently, the divisor  $2D|_{E,\Sigma_t}+(D.S) S|_{E,\Sigma_t}$. The order~$2$ transformation $x \mapsto \overline{D}_{\phi} - x$ induces the following equivalence  relation on $\Pic_{D.E} (E)$:
$$x \sim y \quad \text{if and only if}\quad y=x \text{ or } y=\overline{D}_{\phi}-x.$$
We denote by $\overline{x}$ the equivalence class of $x$ under this relation.
Note that $\Pic_{D.E} (E)/_{x \sim \overline{D}_{\phi}-x}$ is the quotient of the elliptic curve $E$ by a fixed-point-free involution, and this quotient is holomorphically  
isomorphic to $\C {P}^1$.
We can define the following 
holomorphic map between Riemann spheres:
    \begin{align*}
    \phi_{\overline{D}}\colon  \mathcal{Q} &\longrightarrow \Pic_{D.E} (E)/_{x \sim \overline{D}_{\phi}-x}\\
    \Sigma_t &\longmapsto \overline{D|_{E,\Sigma_t}}, 
\end{align*}
where, for each $\Sigma_t \in \mathcal{Q}$, the class $\overline{D|_{E, \Sigma_t}}$ denotes the equivalence class of $D|_{E, \Sigma_t}$ in $\Pic_{D.E}(E)$ under the relation $x \sim \overline{D}_{\phi} - x$.

\begin{lemma} \label{lem-existence.of.L}
There exists a pair $\{L,T(L)\}\subset H_2(\Sigma;\Z)$ such that $(L.S)^2=1$. In particular, both classes $L$ and  $T(L)$ can be chosen to be effective. 
\end{lemma} 

\begin{proof}
According to \cite[Claim (6), p.~26]{demazure2006seminaire}, the vanishing cycle $S$ can always be expressed as the difference of two effective generators of $H_2(\Sigma; \mathbb{Z})$, and the monodromy transformation $T$ interchanges these two classes. Suppose that $S$ can be written as $L_1-L_2$, where $(L_1,L_2,L_3,\ldots, L_k)$ is a $k$-tuple of effective generators for $H_2(\Sigma;\Z)\cong \Z^k$. The intersection form on $H_2(\Sigma;\Z)$ implies that we can choose $L=L_2$ and $T(L)=L_1$, or \textit{vice versa}, so that $L.S = \pm 1$.  
\end{proof}

\begin{lemma}\label{lem-degree.phi(S)}
 One has $\deg (\phi_{\overline{S}})= 4$.
\end{lemma}

\begin{proof} 
  By Lemma~\ref{lem-existence.of.L}, for any pair of effective divisors $L$ and $T(L)$ on $\Sigma$ such that $L.S= \pm 1$, the pair $(L|_E, T(L)|_E)$  of their restrictions to the base locus $E$ determines a unique surface in the pencil $\mathcal{Q}$.  More precisely, if $L|_E \neq T(L)|_E$, this correspondence yields a non-singular surface in the pencil. Conversely, if $L|_E = T(L)|_E$, then $L$ and $T(L)$ restrict to the same divisor on $E$, which corresponds to a ramification point of the corresponding double covering of the parameter space. In this case, the corresponding surface in the pencil is singular, typically having an ordinary double point. Thus, the singular elements in the pencil are in bijection with the points where $L|_E = T(L)|_E$.  This relationship can be summarized by the following commutative diagram:
\begin{center}
 \begin{tikzcd}
\Pic_{L.E} (E) \arrow[r, "\hat{\pi}\colon x \mapsto \Sigma_x"] \arrow[d, "2:1"]
& \mathcal{Q} \arrow[dl, "\cong"] \\
\Pic_{L.E} (E)/_{x\sim \overline{L}_{\phi} -x}\rlap{\,,} & 
\end{tikzcd}
\end{center}
where $\hat{\pi}$ is a  ramified double covering map. Observe that $\hat{\pi}(\overline{L}_{\phi}-x)=\hat{\pi}(x)$, meaning that $\Sigma_{\overline{L}_{\phi}-x}=\Sigma_x$.
By analyzing the ramification points of $\hat{\pi}$, we obtain exactly four singular surfaces in the pencil $\mathcal{Q}$, which is consistent with Lemma~\ref{lem-4.singulars}. Furthermore, we have the following commutative diagram:
\begin{center}
 \begin{tikzcd}
&& \mathcal{Q} \arrow[d, "\phi_{\overline{S}}"] \arrow[dll] \\
\Pic_{L.E} (E)/_{x\sim \overline{L}_{\phi} -x}   \arrow[rr, "\overline{x}\mapsto \overline{\overline{L}_{\phi}-2 x}"]&&  \Pic_{0} (E)/_{y\sim -y}\rlap{\,.}
\end{tikzcd}
\end{center}

These diagrams together show that  each singular surface in the pencil $\mathcal{Q}$ corresponds to an element $x\in \Pic_{L.E}(E)$ such that $x=\overline{L}_{\phi}-x$, which is mapped to $\overline{0}\in \Pic_0(E)/_{y\sim -y}$. Consequently, the degree of the map $\phi_{\overline{S}}$  is exactly the number of singular surfaces in the pencil $\mathcal{Q}$, which is equal to $4$.
\end{proof}

\begin{corollary} \label{coro-isomorphism.pencil.picard}
Let $\{L,T(L)\}\subset H_2(\Sigma;\Z)$ be such that $(L.S)^2=1$. Then, the corresponding map $\phi_{\overline{L}}$ is an isomorphism. 
\end{corollary}

The following proposition is the main point for solving the \textit{$1$-dimensional problem} that we mentioned at the beginning of this section.

\begin{proposition} \label{Prop-degree.case.complex}
  Let $X$ be a $3$-dimensional del Pezzo variety such that
  $$\ker(\psi\colon  H_2(\Sigma;\mathbb{Z})\to H_2(X;\mathbb{Z}))\cong \Z,$$
  where $\Sigma\in \mathcal{Q}$ is a non-singular surface. Let $D\in \Pic(\Sigma)$ be a divisor. Let  $S\in \ker(\psi)$ be a generator.  Then, the map $\phi_{\overline{D}}$ has degree $(D.S)^2$.
\end{proposition}

\begin{proof}
Consider the following commutative diagram of holomorphic maps between Riemann spheres:
\begin{equation}\tag{I}\label{I}
\centering
    \begin{tikzcd}
\mathcal{Q} \arrow[r, "\phi_{\overline{D}}"] \arrow[d, "\phi_{\overline{S}}"]
& \Pic_{D.E} (E)/_{x\sim \overline{D}_{\phi}-x} \arrow[d, "\theta_1"] \\
\Pic_{0} (E)/_{y\sim -y} \arrow[r, "\theta_2"]
& \Pic_{0} (E)/_{y\sim -y}\rlap{\,,}
\end{tikzcd}
\end{equation}
where
$$\theta_1(\overline{x})=\overline{\overline{D}_{\phi}-2x}\quad \text{and}\quad \theta_2(\overline{y})=(D.S)\overline{y}.$$
As a result, the following equality holds:
$$\deg\left(\phi_{\overline{D}}\right)\times \deg(\theta_1)=\deg\left(\phi_{\overline{S}}\right)\times \deg(\theta_2).$$
By Lemma~\ref{lem-degree.phi(S)}, we have $\deg (\phi_{\overline{S}})= 4$. In order to compute $\deg(\phi_{\overline{D}})$, we proceed to compute $\deg(\theta_1)$ and $\deg (\theta_2)$ as follows.

We first show that $\deg (\theta_2)= (D.S)^2$. Let $\overline{y_1} \in \Pic_{0} (E)/_{y\sim -y}$ (the quotient that is in the lower right corner of Diagram \eqref{I}). The degree of $\theta_2$ can be expressed as
$$\deg(\theta_2)=\left|\left\{\overline{y}\in \Pic_{0} (E)/_{y\sim -y} \; : \; (D.S)\overline{y}=\overline{ y_1}\right\}\right|.$$
The commutative diagram
\begin{center}
    \begin{tikzcd}
\Pic_{0} (E) \arrow[r, "y\mapsto (D.S)y"] \arrow[d, "2:1"]
& \Pic_{0} (E) \arrow[d, "2:1"] \\
\Pic_{0} (E)/_{y\sim -y} \arrow[r, "\theta_2"]
& \Pic_{0} (E)/_{y\sim -y}
\end{tikzcd}
\end{center}
 implies that $\deg(\theta_2)$ is exactly the number of solutions $y\in \Pic_0(E)$ to the equation $(D.S)y=y_1$ in $\Pic_0(E)$, in  which the divisor $y_1$ is given. Moreover, any two solutions $y$ to this equation differ by a torsion point of order $D.S$. This equation then has exactly $(D.S)^2$ solutions $y$. Hence, $\deg(\theta_2)=(D.S)^2$, as required.

We now show that  $\deg (\theta_1)= 4$ in a similar way. Let $\overline{y_1} \in \Pic_{0} (E)/_{y\sim -y}$ (the quotient that is in the lower right corner of Diagram \eqref{I}. The degree of $\theta_1$ can be expressed as
$$\deg(\theta_1)=\left|\left\{\overline{x}\in \Pic_{D.E} (E)/_{x\sim \overline{D}_{\phi} -x} \; : \; \overline{ \overline{D}_{\phi}-2x}=\overline{ y_1}\right\}\right|.$$
The commutative diagram
\begin{center}
    \begin{tikzcd}
\Pic_{D.E} (E) \arrow[r, "x\mapsto \overline{D}_{\phi}-2x"] \arrow[d, "2:1"]
& \Pic_{0} (E) \arrow[d, "2:1"] \\
\Pic_{D.E} (E)/_{x\sim \overline{D}_{\phi} -x} \arrow[r, "\theta_1"]
& \Pic_{0} (E)/_{y\sim -y}
\end{tikzcd}
\end{center}
implies that $\deg(\theta_1)$ is exactly the number of solutions $x\in \Pic_{D.E}(E)$ to the equation $\overline{D}_{\phi}-2x=y_1$ in $\Pic_0(E)$, where the two divisors $y_1$ and $\overline{D}_{\phi}$ are given. Moreover, any two solutions $x$ to this equation differ by a torsion point of order $2$. This equation then has exactly $2^2$ solutions $x$.  Hence, $\deg(\theta_1)=4$, as required.
This completes the proof of the proposition.
\end{proof}

\subsection{Real pencils of real del Pezzo surfaces}
Throughout this subsection,  we restrict our attention to a real $3$-dimensional del Pezzo variety $(X,\tau)$  such that $\ker(\psi) \cong \Z$, where $\Sigma \in |-\frac{1}{2}K_X|$ is a non-singular surface, and the surjection $\psi\colon  H_2(\Sigma;\mathbb{Z})	\twoheadrightarrow H_2(X;\mathbb{Z})$ is induced by the inclusion $\Sigma\hookrightarrow X$. Henceforth, the notions of vanishing cycle and generator of $\ker (\psi)$ may be used interchangeably (see Lemma~\ref{Lemma-Vanishing.Cycle}). Let $\mathcal{Q}$ denote a pencil of surfaces in the linear system $|-\frac{1}{2}K_X|$. Let the elliptic curve $E$ be its base locus. Recall that $E$ realizes the homology class $\frac{1}{4}K_X^2$ (see Proposition~\ref{Prop-E.is.elliptic}). Now suppose  that both the curve $E$ and the pencil $\mathcal{Q}$ are real with non-empty real parts. We first examine certain properties of the real pencil $\mathcal{Q}$ that distinguish them from the complex case.

\subsubsection{Properties of real pencils of real surfaces}

\begin{lemma} \label{lem-surfaces.in.real.pencil}
The surfaces in the real pencil $\mathcal{Q}$ containing a real rational curve in $(X,\tau)$ are in the real part $\R \mathcal{Q}$ of this pencil.
\end{lemma}

\begin{proof}
  Let $C$ be a real rational curve in $(X,\tau)$. Suppose that the surface $\Sigma\in \mathcal{Q}$ contains $C$. Since $\mathcal{Q}$ is a real pencil, we have $\tau(\Sigma)\in \mathcal{Q}$. Since $C$ is real, we have $\tau(C)=C$.  If $\Sigma\notin \R \mathcal{Q}$, \textit{i.e.}, $\tau(\Sigma)\neq \Sigma$, then $C\subset (\Sigma\cap \tau(\Sigma))=E$, which is not possible since $C$ is a rational curve. Thus, by contradiction  $\tau(\Sigma)=\Sigma$. In other words, $\Sigma$ is in the real part of the pencil $\mathcal{Q}$.
\end{proof}

Note that if $(\Sigma_1,\tau|_{\Sigma_1})$ and $(\Sigma_2,\tau|_{\Sigma_2})$ are two real non-singular surfaces in $\R \mathcal{Q}$ such that $\Sigma_2$ is a surgery of $\Sigma_1$ along the vanishing cycle (as formalized in \cite{brugalle2018surgery}), then the Euler characteristics of the real parts of these two surfaces differ by~$2$. Suppose that $\chi(\R \Sigma_2)=\chi(\R \Sigma_1)+ 2$. Also note  from \cite{brugalle2018surgery} that the vanishing cycle class $S$ in $H_2(\Sigma_1;\Z)$ is $\tau|_{\Sigma_1}$-anti-invariant if and only if it is $\tau|_{\Sigma_2}$-invariant.  Let $\alpha \in \{D,T(D)\}$. We say that $\alpha$ is $\tau|_{\Sigma_t}$-anti-invariant if we have both
$\tau|_{\Sigma_t*}(D)=-D$ and
$\tau|_{\Sigma_t*}(T(D))=-T(D)$.

\begin{proposition} \label{prop-tau.vs.mu}
Let $d\in H^{-\tau}(X;\Z)$. Let $(\Sigma_t,\tau|_{\Sigma_t}) \in \R \mathcal{Q}$ be a real del Pezzo surface containing a real rational curve in $X$ $(t\in {\R})$, and let $\psi_t\colon  H_2(\Sigma_t;\mathbb{Z})\rightarrow H_2(X;\mathbb{Z})$ be the surjective map  induced from the inclusion $\Sigma_t\hookrightarrow X$. Let $\alpha \in \{D,T(D)\}\subset \psi_t^{-1}(d)$ be such that $D.S\neq 0$. In the real surface $(\Sigma_t,\tau|_{\Sigma_t})$, the vanishing cycle is $\tau|_{\Sigma_t}$-anti-invariant if and only if $\alpha$ is $\tau|_{\Sigma_t}$-anti-invariant.
\end{proposition}  

\begin{proof}
The sufficient condition follows directly. Indeed, suppose that
$$\begin{cases}
\tau|_{\Sigma_t*}(D)=-D,\\
\tau|_{\Sigma_t*}(T(D))=-(D+(D.S)S)
\end{cases}$$
hold. Then, $$(D.S)\tau|_{\Sigma_t*}(S)=\tau|_{\Sigma_t*}(T(D)-D)=-(D.S)S.$$
If $D \cdot S \neq 0$, we conclude that $\tau|_{\Sigma_t*}(S) = -S$; that is, $S$ is $\tau|_{\Sigma_t}$-anti-invariant. This completes the proof of the sufficiency.

The necessary condition is established in the following two steps.

\textit{Step 1. We show that $\tau|_{\Sigma_t*}(\{D,T(D)\})=\{-D,-T(D)\}$.}
Suppose that $\alpha=D$. It suffices to prove that if $D$ satisfies $\psi_t(D)=d\in H^{-\tau}(X;\Z)$, then one of the following holds:
\begin{enumerate}
    \item $\tau|_{\Sigma_t*}(D)=-D$; \textit{i.e.}, $D$ is $\tau|_{\Sigma_t}$-anti-invariant.
    \item $\tau|_{\Sigma_t*}(D)=-(D+(D.S)S)$, where $S$ is a generator of $\ker(\psi_t)$.
\end{enumerate}
Indeed, since $\tau_* (d)=-d$, we have $\tau_*(\psi_t(D))=-\psi_t (D)$. Because $\psi_t\circ\tau|_{\Sigma_t*}=\tau_*\circ \psi_t$, it follows that $\psi_t (\tau|_{\Sigma_t*}(D)+D)=0$. Moreover, since $\ker (\psi_t) =\Z S$, there exists an integer $k\in \Z$ such that $\tau|_{\Sigma_t*}(D)+D=kS$. The intersection form is preserved under $\tau|_{\Sigma_t*}$; that is, $(\tau|_{\Sigma_t*}(D))^2=D^2$. Thus, $k$ must be either $0$ or $-D.S$. This completes the first step.

\textit{Step 2. We show that if $(\Sigma_t,\tau|_{\Sigma_t})$ is  such that the vanishing cycle $S$ is $\tau|_{\Sigma_t}$-anti-invariant in $H_2(\Sigma_t; \Z)$, then both $D$ and $T(D)$ must be $\tau|_{\Sigma_t}$-anti-invariant. } Suppose, toward a contradiction, that $\tau|_{\Sigma_t*}(D)=-T(D)$. One has $\tau|_{\Sigma_t*}(D)=-D-(D.S)S=\tau|_{\Sigma_t*}(T(D))-(D.S)S$. This implies that 

$$\tau|_{\Sigma_t*}(T(D))-\tau|_{\Sigma_t*}(D)=(D.S)S.$$
Moreover, $$\tau|_{\Sigma_t*}(T(D))=\tau|_{\Sigma_t*}(D+(D.S)S)=\tau|_{\Sigma_t*}(D)+(D.S) \tau|_{\Sigma_t*}(S).$$
Comparing the two expressions, we obtain $S=\tau|_{\Sigma_t*}(S)$. This means that $S$ is not $\tau|_{\Sigma_t}$-anti-invariant, which contradicts our assumption. Thus, the second step is complete, and the proposition follows.
\end{proof}

As a consequence, the enumerative problem of counting real curves representing $\tau$-anti-invariant classes in $X$ can be reduced to the enumerative problem of counting real curves representing $\tau|_{\Sigma_t}$-anti-invariant  classes in some surface $\Sigma_t$. The latter problem is, in general, simpler and has been extensively studied in the literature on real enumerative geometry.

To make a meaningful comparison between these two real enumerative problems, we analyze the real base locus $E$ as in the complex case. It is well known that the choice of a point $x_0\in E$ induces an isomorphism $E\cong \Pic_k(E)$  for any $k\in \mathbb{Z}$. Since $E$ is a real elliptic curve and $\R E\neq \emptyset$, the real parts  $\R E$ and $\R \Pic (E)$ have the same topological type. Consequently, there exists a homeomorphism $\R E \simeq \R \Pic_k(E)$ for any $k\in \mathbb{Z}$. Assume that $x_0\in \R E$. The real part $\R E$ is either connected or disconnected with two connected components. In the former case, we write $\mathbb{R} E =\mathbb{S}^1$. In the latter case, we write $\mathbb{R} E =\mathbb{S}^1_0\sqcup \mathbb{S}^1_1$ and assume that $x_0\in\mathbb{S}^1_0$. We refer to $\mathbb{S}^1_0$ as the \textit{pointed component} and $\mathbb{S}^1_1$ as the \textit{non-pointed component} of $\R E$.

Let $\underline{x}_{d}$ be a real configuration of $|D.E|$ points on $E$. Suppose that $\alpha|_{E,\Sigma_t}=[\underline{x}_{d}]$ in $\Pic_{D.E}(E)$, where $\alpha$ is as in Proposition~\ref{prop-tau.vs.mu}. 
We consider the  maps
    \begin{align*}
    \mu_t\colon \Pic(\Sigma_t)&\lra \Pic(\Sigma_t)\\
    \sum_i n_i [C_i]&\longmapsto \sum_i n_i \left[\tau|_{\Sigma_t}(C_i)\right]
\end{align*}
and
\begin{align*}
    \mu_E\colon \Pic(E)&\lra \Pic(E)\\
    \sum_i n_i [P_i]&\longmapsto \sum_i n_i \left[\tau|_{E}(P_i)\right].
\end{align*}
The restriction of a divisor $D=\sum_i n_i [C_i]\in \Pic(\Sigma_t)$ to $\Pic(E)$ is given by $$D|_{E,\Sigma_t}=\sum_i n_i [C_i\cap E].$$ 
The restriction of $\Pic(\Sigma_t)$ to $\Pic(E)$ induces the following commutative diagram:
 \begin{center}
    \begin{tikzcd}
\Pic(\Sigma_t) \; \; \;  \arrow[rr, "\mu_t"] \arrow[d]
&& \; \; \;  \Pic(\Sigma_t)\arrow[d] \\
\Pic(E)  \; \; \;  \arrow[rr, "\mu_E"]
&& \; \; \;  \Pic(E)\rlap{.}
    \end{tikzcd}
        \end{center}
 
     \begin{lemma} \label{lem-anti-invariant.and.real.divisors}
         If the class $D$ is $\tau|_{\Sigma_t}$-anti-invariant, then its restriction $D|_{E,\Sigma_t}$ lies in $\R \Pic(E)$.
     \end{lemma}

     \begin{proof}
   We show that if $D$ is 
$\tau|_{\Sigma_t}$-anti-invariant, then  $\mu_E(D|_{E,\Sigma_t})=D|_{E,\Sigma_t}$.  Indeed, since the involution $\tau$ reverses the orientation, meaning that $\mu_t=-\tau|_{\Sigma_t*}$, one has $D|_{E,\Sigma_t}=\mu_t(D)|_{E,\Sigma_t}$. The above commutative diagram implies that $\mu_E(D|_{E,\Sigma_t})=\mu_t(D)|_{E,\Sigma_t}$. Hence, the lemma follows.
\end{proof}

Given $D \in \Pic(\Sigma_t)$, we consider the following ramified double covering:
\[
\pi\colon  \Pic(E) \longrightarrow \Pic(E) \big/ \sim, \quad x \longmapsto \overline{x},
\]
where the equivalence relation is defined by $x \sim \overline{D}_{\phi} - x$. The critical points of the map $\pi$ are precisely the solutions to the equation $2x = \overline{D}_{\phi}$ in $\Pic(E)$. There exists a homeomorphism
\[
\mathbb{R} \left( \Pic(E) \big/ \sim \right) \simeq \mathbb{S}^1.
\]
Moreover, the preimage of the real locus under $\pi$ is given by
\[
\pi^{-1} \left( \mathbb{R} \left( \Pic(E) \big/ \sim \right) \right )
= \left\{\, (x,\, \overline{D}_{\phi} - x) \; : \; x \in \mathbb{R}\Pic(E) \, \right\}
\cup \left\{\, (y,\, \tau|_E(y)) \; : \; y \in \Pic(E) \setminus \mathbb{R}\Pic(E),\; y + \tau|_E(y) = \overline{D}_{\phi} \, \right\},
\]
where $\tau|_E$ denotes the real structure on $E$. The two types of real divisors on  $\mathbb{R}(\Pic(E)/_{x \sim \overline{D}_{\phi}-x})$ are as illustrated in  Figure~\ref{fig-RealElliptic}. 

\begin{figure}[ht!]
\centering
\includegraphics[scale=0.4]{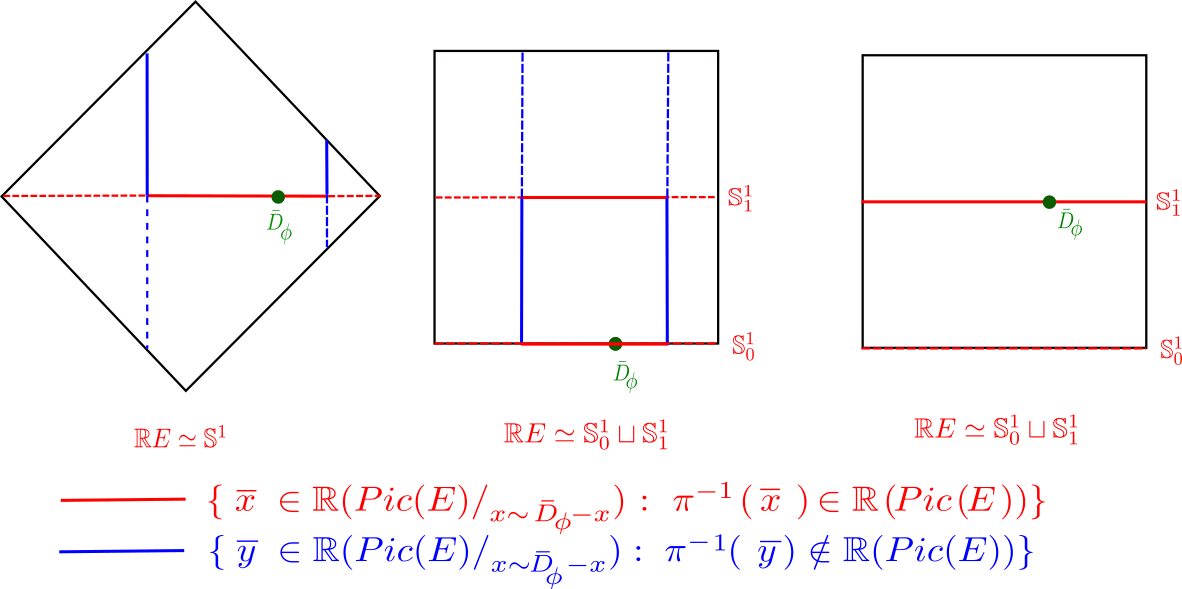}
\caption{Real divisors in the Picard group $\mathbb{R}(\Pic(E)/_{x \sim \overline{D}_{\phi}-x})$ of a real elliptic curve with non-empty real part.}
\label{fig-RealElliptic}
\end{figure}

\subsubsection{Enumerations on real pencils of real del Pezzo surfaces}
We keep the same notation as in  Section~\ref{Sect-Pencils}: the vanishing cycle  $S$ of the Lefschetz fibration induced from the pencil $\mathcal{Q}$, the constant map $\phi$ whose image is the divisor  $\overline{D}_{\phi}=(2D+(D.S)S)|_{E,\Sigma_t}$ in  $\Pic_{2D.E} (E)$, which is independent of $\Sigma_t \in \mathcal{Q}$, and the degree $(D.S)^2$ map $\phi_{\overline{D}}$ (see Proposition~\ref{Prop-degree.case.complex}). 

Assume that $\{L,T(L)\}\subset H_2(\Sigma_t;\Z)$ is as in  Lemma~\ref{lem-existence.of.L}. Recall, in particular, that $\overline{L}_{\phi}$ is  in $\Pic_{2L.E}(E)$ and $\phi_{\overline{L}}$ is an isomorphism (Corollary~\ref{coro-isomorphism.pencil.picard}). As previously shown, the divisor $\overline{D}_{\phi}$ is real, for every class $D \in H^{-\tau}(\Sigma_t;\Z)$ (see Lemma~\ref{lem-anti-invariant.and.real.divisors}). In particular, the divisor $\overline{L}_{\phi}$ is real, and both $\phi_{\overline{D}}$ and $\phi_{\overline{L}}$ are real maps.
Let $\mathcal{R}(\underline{x}_d)$ be a set of real irreducible rational curves in $X$ representing a given homology class $d\in H^{-\tau}(X;\Z)$ and passing through a real configuration $\underline{x}_d$ of $k_d=\frac{1}{2}c_1(X)\cdot d$ points on $E$.

\begin{lemma} \label{lem-constant.divisor.L.D}
The divisor $(D-|D.S|L)|_{E,\Sigma_t}$, which has degree $(D.E)-|D.S|(L.E)$ in $\Pic(E)$, is denoted by $\phi(\overline{L},\overline{D})$ and does not depend on the choice of the surface $\Sigma_t$ in the pencil $\mathcal{Q}$.
\end{lemma}

\begin{proof}
Clearly, the divisor  $(D-|D.S|L)|_{E,\Sigma_t}$ has degree $(D.E)-|D.S|(L.E)$ in $\Pic(E)$. Without loss of generality, suppose that $D.S> 0$. We have
  \begin{align*}
      2\left(D-(D.S)L\right)|_{E,\Sigma_t} &= \left(2D-2(D.S)L\right)|_{E,\Sigma_t}\\
      &= \left((2D+(D.S)S)-(D.S) (2L+S)\right) |_{E,\Sigma_t}.
  \end{align*}
   
As previously shown, both $(2L+S)|_{E,\Sigma_t}=\overline{L}_{\phi}$ and $(2D+(D.S)S)|_{E,\Sigma_t}=\overline{D}_{\phi}$ are constant with respect to $\Sigma_t$. Consequently, the divisor $ 2\left(D-(D.S)L\right)|_{E,\Sigma_t}$ does not depend on the choice of $\Sigma_t$. Therefore, by continuity, the divisor $ \left(D-(D.S)L\right)|_{E,\Sigma_t}$  is also independent of $\Sigma_t$. 
\end{proof}

The next proposition is the main ingredient to solving the \textit{$1$-dimensional real problem} of Theorem~\ref{Theorem2-W}.

\begin{proposition} \label{prop-corresponding.surface.picE}
Let $\Sigma \in \mathcal{Q}$ be a fixed real non-singular surface.
Let $d\in H^{-\tau}(X;\Z)$, and let $\alpha$ be an element of 
$\{D,T (D)\}\subset \psi^{-1}(d)$. Then, for any real surface $(\Sigma_t, \tau|_{\Sigma_t}) \in \mathbb{R} \mathcal{Q}$ containing a curve in the set $\mathcal{R}(\underline{x}_d)$, the corresponding vanishing cycle is $\tau|_{\Sigma_t}$-anti-invariant.

Furthermore, such real surfaces correspond to the real solutions $l\in \R \Pic(E)$ of the equation
\begin{equation}\label{eq-solutions.in.PicE}
    |D.S|l=[\underline{x}_d]-\phi\left(\overline{L},\overline{D}\right),
\end{equation}
where $[\underline{x}_d]$ and $\phi(\overline{L}, \overline{D})$  are as previously defined.
\end{proposition}

\begin{proof}
  The first assertion follows directly from the fact that if the vanishing cycle is not $\tau|_{\Sigma_t}$-anti-invariant, then $D\notin  H_2^{-\tau}(\Sigma_t;\Z)$ for any $D\in H_2(\Sigma_t;\Z)$.
  
  We now address the second assertion. Since the equality $D.S=-T(D).S$ holds, we may assume without loss of generality that $D.S> 0$.
   Then, one has
  $$\phi\left(\overline{L},\overline{D}\right)=\left(D-(D.S)L\right)|_{E,\Sigma_t}=\left(D+(T(D).S)L\right)|_{E,\Sigma_t}.$$ 
 
  Let $\Sigma_t \in \mathcal{Q}$  be a surface containing a curve in  $\mathcal{R}(\underline{x}_d)$, and suppose that this curve represents the homology class $D\in H_2^{-\tau}(\Sigma_t;\Z)$. Then, $[\underline{x}_d]=D|_{E,\Sigma_t}$ is a real divisor in $\Pic_{D.E}(E)$.

We can write  $D=(D.S)L+D-(D.S)L$. Restricting to $E$, and noting that $L$ is also $\tau|_{\Sigma_t}$-anti-invariant, we obtain the following equation in $\Pic(E)$:
  $$\left[\underline{x}_d\right]=(D.S)L|_{E,\Sigma_t}+\left(D-(D.S)L\right)|_{E,\Sigma_t}.$$
 By Lemma~\ref{lem-constant.divisor.L.D}, the divisor $\left(D-(D.S)L\right)|_{E,\Sigma_t}=\phi(\overline{L},\overline{D})$ does not depend on the choice of the surface $\Sigma_t$ in the pencil. 
 Therefore, for every surface $\Sigma_t \in \mathcal{Q}$ containing a curve in $\mathcal{R}(\underline{x}_d)$ of homology class $D\in H_2^{-\tau}(\Sigma_t;\Z)$ where $D.S> 0$, the following equation holds in $\Pic(E)$:
 $$(D.S)L|_{E,\Sigma_t}=[\underline{x}_d]-\phi\left(\overline{L},\overline{D}\right).$$
 Consequently, for every real surface $\Sigma_t \in  \R \mathcal{Q}$, on which a curve in $\mathcal{R}(\underline{x}_d)$ represents a $\tau|_{\Sigma_t}$-anti-invariant class $\alpha\in \{D,T (D)\}$, we have
 \begin{equation} \label{eq-proof-eq-solutions.in.PicE}
     (|D.S|L)|_{E,\Sigma_t}=[\underline{x}_d]-\phi(\overline{L},\overline{D})
 \end{equation}
  in $ \Pic(E)$. Moreover, the map
  $$   \phi_{\overline{L}}\colon  \mathcal{Q} \overset{\lowcong}\lra \Pic_{L.E} (E)/_{x\sim \overline{L}_{\phi}-x}, \quad    \Sigma_t \longmapsto \overline{L|_{E,\Sigma_t}}$$
   is a real isomorphism. Together with the remarks on real structures on the pencil $\mathcal{Q}$ and on the quotient of the Picard group of its base locus, this implies that such real surfaces $\Sigma_t \in  \R \mathcal{Q}$ are precisely the real solutions in  $\R \Pic(E)$ to Equation (\ref{eq-proof-eq-solutions.in.PicE}). The proposition is thus established.
  \end{proof}

As a consequence of Proposition~\ref{prop-corresponding.surface.picE}, there is a bijection between the set of such real surfaces in the pencil~$\mathcal{Q}$ and the set of real solutions to Equation (\ref{eq-solutions.in.PicE}). Thus, the number of these real surfaces is determined by the number of real $|D.S|$-torsion points on $\R \Pic(E)$. In particular, this number depends on the real structure of  $E$ and, when $\R E$ is disconnected, on the position of the real divisors  $\phi(\overline{L},\overline{D})$ and $[\underline{x}_d]$ in $\R \Pic(E)$. More precisely, this number is described in the following corollary; see also Figure~\ref{fig-real.solutions.in.PicE}.

\begin{corollary}\label{coro-degree.case.real}
Let $d \in H^{-\tau}(X;\mathbb{Z})$ and $\alpha \in \{D, T(D)\} \subset \psi^{-1}(d)$. Consider the set of real surfaces $(\Sigma_t, \tau|_{\Sigma_t})$ in $\mathbb{R} \mathcal{Q}$ such that each $\Sigma_t$ contains a real curve of class $\alpha$ belonging to $\mathcal{R}(\underline{x}_d)$. Then, the cardinality of this set is given as follows:
\begin{enumerate}[label={\rm(\alph*)}, ref={\rm\alph*}]
         \item $|D.S|$ if either $\R E$ is connected, or $\R E$ is disconnected and $D.S$ is odd;
      \item $2|D.S|$ if\, $\R E$ is disconnected, $D.S$ is even, and      both $[\underline{x}_d]$ and $\phi(\overline{L},\overline{D})$ lie on the same component of\, $\R E$; 
      \item\label{c-d.c.r-c} $0$ if\, $\R E$ is disconnected, $D.S$ is even, and $[\underline{x}_d]$ and $\phi(\overline{L},\overline{D})$ lie on different components of\, $\R E$. 
  \end{enumerate}
  \end{corollary}

 \begin{proof}
We first enumerate the real $|D.S|$-torsion points, that is, the real solutions to the equation $|D.S|l=0$. The number and distribution of these points depend on the real structure of $E$, as follows:
  \begin{enumerate}[(1)]
 \item If $\R E$ is connected, there are exactly $|D.S|$  real $|D.S|$-torsion points.
 \item If $\R E=\mathbb{S}^1_0\sqcup \mathbb{S}^1_1$ is disconnected, then 
\begin{itemize}
\item if $|D.S|$ is odd, there are exactly $|D.S|$ real $|D.S|$-torsion points, all lying in the pointed component  $\mathbb{S}^1_0$; 
\item if $|D.S|$ is even, there are exactly $2|D.S|$  real $|D.S|$-torsion points, with half of them contained in each connected component of $\R E$.
\end{itemize} 
 \end{enumerate}
Next, we count the real solutions $l$ of Equation (\ref{eq-solutions.in.PicE}). When $\R E$ is disconnected, the number of solutions depends on the positions of the real divisors $\phi(\overline{L},\overline{D})$ and $[\underline{x}_d]$ on $\R E$. More precisely,
\begin{enumerate}[(1)]
\item if $\R E=\mathbb{S}^1$, there are exactly $|D.S|$ real solutions to Equation (\ref{eq-solutions.in.PicE}), see Figure~\ref{fig-real.solutions.in.PicE}-$a1)$;
\item if $\R E=\mathbb{S}^1_0\sqcup \mathbb{S}^1_1$, then
\begin{itemize}
\item  if $|D.S|$ is odd, there are exactly $|D.S|$ real solutions to Equation (\ref{eq-solutions.in.PicE}),
regardless of whether $\phi(\overline{L}, \overline{D}) - [\underline{x}_d]$ lies in $\mathbb{S}^1_0$ or $\mathbb{S}^1_1$, see Figure~\ref{fig-real.solutions.in.PicE}-$a2)$;
\item if $|D.S|$ is even, then 
\begin{itemize}
\item there are exactly $2|D.S|$ real solutions to Equation (\ref{eq-solutions.in.PicE}) if $[\underline{x}_d]-\phi(\overline{L},\overline{D})$ is in the pointed component $\mathbb{S}^1_0$ (equivalently, if $\phi(\overline{L},\overline{D})$ and $[\underline{x}_d]$ belong to the same connected component of $\mathbb{R} E$), see Figure~\ref{fig-real.solutions.in.PicE}-$b)$;
\item there are no real solutions to Equation (\ref{eq-solutions.in.PicE}) if  $[\underline{x}_d]-\phi(\overline{L},\overline{D})$ is in $\mathbb{S}^1_1$ (equivalently, if $\phi(\overline{L},\overline{D})$ and $[\underline{x}_d]$ are in different connected components of $\R E$), see Figure~\ref{fig-real.solutions.in.PicE}-$c)$.
\end{itemize}
\end{itemize}
\end{enumerate}

\begin{figure}[ht!]
\centering
\includegraphics[scale=0.5]{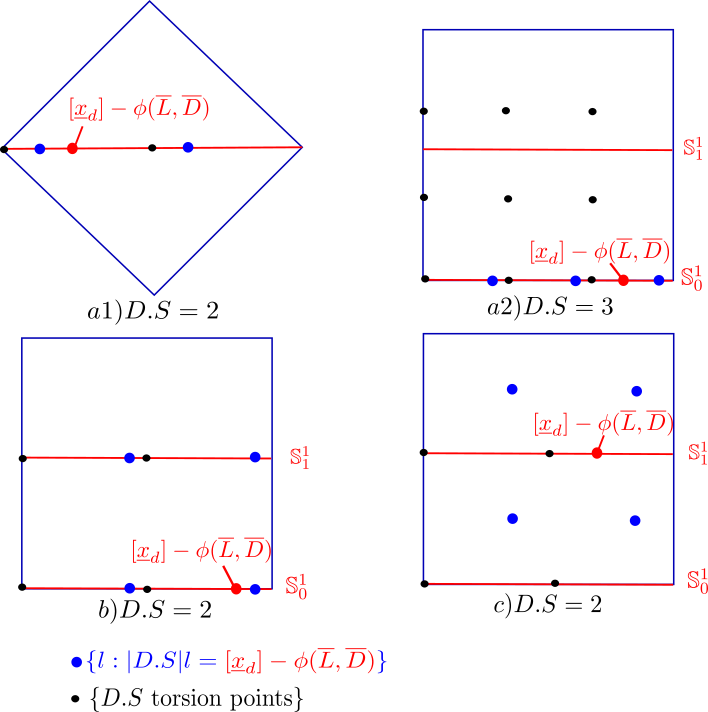}
\caption{The real solutions of Equation (\ref{eq-solutions.in.PicE}) corresponding to the three cases of Corollary~\ref{coro-degree.case.real}.}
\label{fig-real.solutions.in.PicE}
\end{figure} 
This completes the proof of the corollary.
\end{proof} 

%%%%%%%%%%%%%%%%%%%%%%%%%
\section{Rational curves on del Pezzo varieties of dimension 3 and 2: Relationship } \label{Section-ratcurve}
Recall from the last section that a non-singular element $\Sigma\in |-\frac{1}{2}K_X|$ is a non-singular del Pezzo surface, where $X$ is a del Pezzo variety of dimension~$3$. Also recall  that $\psi\colon  H_2(\Sigma;\mathbb{Z})\to H_2(X;\mathbb{Z})$ is the surjective map induced by the inclusion $\Sigma\hookrightarrow X$.

 In this section, we continue to use the notation $\mathcal{Q}$ to denote a pencil of surfaces in the linear system $|-\frac{1}{2}K_X|$, whose base locus is the elliptic curve $E$ representing the class $\frac{1}{4}K_X^2$.  Let $d\in H_2(X;\mathbb{Z})$ and $k_d=\frac{1}{2}c_1(X)\cdot d$. Let $D\in H_2(\Sigma;\mathbb{Z})$ and $k_D=c_1(\Sigma)\cdot D-1$. We proceed to analyze the case where the configuration points are constrained to lie on the elliptic curve $E$. As a first step, we  establish a generalization of a result due to J.~Koll\'ar (Proposition~\ref{Prop-Kollar}). Subsequently, we demonstrate that the selected specialized configuration gives us the genus $0$ Gromov--Witten invariants of the surface $\Sigma$ (see Proposition~\ref{Prop.Modulispace1}, Corollary~\ref{Corollary.Modulispace1}) and of the variety~$X$ (see Proposition~\ref{Prop.Modulispace2}). As a result, this allows us to relate $\GW_X(d)$ to $\GW_\Sigma(D)$.

\subsection{Specifications of configuration of points} \label{SectionSpecificationConfiguration}
 \begin{lemma} \label{lemkDandkd}
Given a $3$-dimensional del Pezzo variety $X$, if\,  $\Sigma\in |-\frac{1}{2}K_X|$ is a non-singular surface and $D\in \psi^{-1}(d)$, then one has $k_D=k_d-1$. 
\end{lemma}

 \begin{proof}
   By the Poincar\'e duality $\tfrac{1}{2}c_1(X)=\PD(-\tfrac{1}{2}K_X)$ and by the hypothesis that $\Sigma\in |-\tfrac{1}{2}K_X|$ is non-singular, we have
   $$k_d=\tfrac{1}{2}c_1(X)\cdot d=[\Sigma]. d .$$
By the adjunction formula, by the hypothesis $D\in \psi^{-1}(d)$ 
and by the Poincar\'e duality $c_1(\Sigma) =\PD(-K_\Sigma)$, we obtain
\begin{equation*}
[\Sigma].d=-K_\Sigma.D=k_D+1.
\end{equation*}
Hence, the lemma follows. 
\end{proof}

  The following proposition generalizes \cite[Proposition 3]{kollar2015examples} by J.~Koll\'ar for the case $\cpp$.

\begin{proposition} \label{Prop-Kollar}
Let $d\in H_2(X;\Z)$, and let  $\underline{x}_d$ be a configuration of\, $k_d=\tfrac{1}{2}c_1(X)\cdot d$ points in general position in the base locus $E$ of pencil $\mathcal{Q}$ of surfaces in the linear system $|-\frac{1}{2}K_X|$. Let $\mathcal{C}(\underline{x}_d)$ be the set of \textit{connected algebraic curves of arithmetic genus $0$ representing the homology class $d$ and passing through $\underline{x}_d$}. Then every curve in $\mathcal{C}(\underline{x}_d)$ is irreducible and is contained in some non-singular surface of\,  $\mathcal{Q}$.
 \end{proposition}
 
 \begin{remark}
     Such non-singular surfaces correspond to solutions $l\in \Pic(E)$ of Equation (\ref{eq-solutions.in.PicE}).
 \end{remark}

 \begin{proof}
Let $C_d$ be a curve in the set $\mathcal{C}(\underline{x}_d)$. We prove the following claims.

\begin{claim}{1}\label{claim3.2-1}
Suppose that $C_d$ is irreducible. Then,  the whole curve $C_d$ has to be contained in some non-singular surface of the pencil $\mathcal{Q}$.
\end{claim}

The first observation is that any point of $C_d\setminus\underline{x}_d$ is contained in some  surface, which is denoted by $\Sigma_t$, of the pencil $\mathcal{Q}$. Then, the number of intersection points between $C_d$ and  $\Sigma_t$ is at least $\frac{1}{2}c_1(X)\cdot d$. By B\'ezout's theorem, the whole curve $C_d$ has to be contained in this surface. We subsequently show that such a surface $\Sigma_t$ is non-singular provided that the chosen points of $\underline{x}_d$ are sufficiently generic. Indeed, suppose, toward a contradiction, that the curve $C_d$ (passing through the configuration $\underline{x}_d$ on $E$) is contained in a singular surface $\Sigma_{\sing}$ of the pencil $\mathcal{Q}$. Let $\rank(\Pic(\Sigma_{\sing}))=k-1$, where  $k\geq 2$. Note that  $\rank(\Pic(\Sigma_{\sing}))=\rank(\Pic(\Sigma))-1$, where $\Sigma$ is a non-singular surface in $\mathcal{Q}$. Let $(L_1,L_2\ldots,L_k)$ be a basis of $ H_2(\Sigma;\Z)$, such that the vanishing cycle $S$ satisfies $S=L_1-L_2\in H_2(\Sigma;\Z)$. Let $(\tilde{L}_{12},\tilde{L}_3\ldots,\tilde{L}_k)$ be a basis of $ H_2(\Sigma_{\sing};\Z)$. Consider the  homomorphism given by
\begin{align*}
    \gamma\colon H_2(\Sigma;\Z_2)&\lra H_2(\Sigma_{\sing};\Z)\\
    L_1-L_2& \longmapsto 0 \\
    L_1+L_2& \longmapsto \tilde{L}_{12}\\
    L_i & \longmapsto \tilde{L}_i,\quad \forall i \in \{3,\ldots,k\}.
\end{align*}

Suppose that the curve $C_d$ represents the homology class $[C_d]=\alpha_{1}\tilde{L}_{12}+\alpha_3 \tilde{L}_3+\cdots + \alpha_k \tilde{L}_k$ in $H_2(\Sigma_{\sing};\Z)$. By this morphism, such a curve realizes the homology class $D=\alpha_{1}L_{1}+\alpha_1 L_2+\alpha_3 L_3+\cdots + \alpha_k L_k$ in $H_2(\Sigma;\Z)$. By Equation (\ref{eq-solutions.in.PicE}) in Proposition~\ref{prop-corresponding.surface.picE}, and by the remark that $D.S=0$, we have
$[\underline{x}_d]=-\phi (\overline{L},\overline{D})$ in $\Pic_{k_d}(E)$. This contradicts the genericity of the configuration $\underline{x}_d$. 

To prove that $C_d$ is irreducible, we argue by contradiction, making use of the following three claims. Suppose that $C_d$ can be decomposed into $n$ ($n\geq 0$) irreducible rational curves $C_{d_i}$, that is, $C_d=\bigcup_{i=1}^{n} C_{d_i}$, where $C_{d_i}$ realizes a homology class $d_i\in H_2(X;\mathbb{Z})$ provided that $\sum_{i=1}^{n}d_i=d$, and  $C_d$  passes through $\underline{x}_{d_i}$ such that $\bigcup_{i=1}^{n}\underline{x}_{d_i}=\underline{x}_d$.

\begin{claim}{2A}\label{claim2}
  Each component curve $C_{d_i}$ passes though exactly $k_{d_i}=\frac{1}{2}c_1(X)\cdot d_i$ points of configuration $\underline{x}_d$.
  \end{claim} 
Indeed, suppose that there exists a component of $C_d$, denoted by $C_{d_1}$, passing through more than $\frac{1}{2}c_1(X)\cdot d_1$ points of $\underline{x}_d$. Then the curve $C_{d_1}$ intersects with every surface of the pencil $\mathcal{Q}$ at more than $\frac{1}{2}c_1(X)\cdot d_1$ points. By B\'ezout's theorem, this implies that $C_{d_1}$ must be contained in every surface of the pencil $\mathcal{Q}$, and hence $C_{d_1} = E$. This is impossible since $C_{d_1}$ is a genus $0$ curve. Hence, every component curve $C_{d_i}$ passes through a configuration $\underline{x}_{d_i}$ of exactly $k_{d_i}$
points of $\underline{x}_d$. As a consequence, the different $C_{d_i}$ passes though different points of $\underline{x}_d$ for all $i$; \textit{i.e.}, $\bigsqcup_{i=1}^{n}\underline{x}_{d_i}=\underline{x}_d$ and $\sum_{i=1}^{n}k_{d_i}=k_d$, where $k_d$ and $k_{d_i}$ are the corresponding integers of the classes $d$ and $d_i$, respectively.

\begin{claim}{2B}\label{claim2B}
  The reducible curve $C_d=\bigcup_{i=1}^{n} C_{d_i}$ is contained in one non-singular surface of the pencil $\mathcal{Q}$.
\end{claim}
Note that the curve $C_d$ is connected. Suppose, toward a contradiction, that two component curves $C_{d_i}$ and $C_{d_j}$, whose intersection is not empty, are contained in two different surfaces $\Sigma_i\neq \Sigma_j$ of the pencil $\mathcal{Q}$. It follows that $\emptyset \neq C_{d_i}\cap C_{d_j}\subset \Sigma_i\cap\Sigma_j=E$. This is in contradiction with Claim~\ref{claim2}. 
Hence the reducible curve $C_d$ is contained in only one  surface of the pencil $\mathcal{Q}$. Moreover, such a surface is non-singular by the genericity of the chosen configuration. One can suppose that the curve $C_d$ is contained in $\Sigma_t\in \mathcal{Q}$ and  present its homology class as $[C_d]=\sum_{i=1}^{n} [C_{d_i}]\in H_2(\Sigma_t;\Z)$.

\begin{claim}{2C}\label{lem44}
  If the configuration $\underline{x}_d$ is generic, then the curve $C_d$ is irreducible.
\end{claim}

We prove this claim  by contradiction as well. 
Fix an integer $k_{d_1}$ such that $0<k_{d_1}<k_d$. Let $\underline{x}_{d_1}\subsetneq \underline{x}_d$ be a strict subconfiguration consisting of points on $E$ such that its cardinality is equal to $k_{d_1}$. Suppose that the curve $C_{d_1}$ realizes a homology class $D_{1}$ in $H_2(\Sigma;\Z)$. By the proof of Proposition~\ref{prop-corresponding.surface.picE}, we have
$$\left[\underline{x}_{d_1}\right]=|D_1\cdot S|l+\phi\left(\overline{L},\overline{D}_1\right).$$
Together with Equation (\ref{eq-solutions.in.PicE}) in Proposition~\ref{prop-corresponding.surface.picE}, we get the following equation in $\Pic(E)$: 
\begin{equation}\label{eq3.2}
    |D\cdot S|\left[\underline{x}_{d_1}\right]=|D_1\cdot S|\left[\underline{x}_d\right]-|D_1\cdot S|\phi \left(\overline{L},\overline{D}\right)+|D\cdot S|\phi\left(\overline{L},\overline{D}_1\right).
\end{equation}
By Lemma~\ref{lem-constant.divisor.L.D}, both $\phi (\overline{L},\overline{D})$ and $\phi(\overline{L},\overline{D}_1)$ are fixed divisors in $\Pic(E)$. Hence, Equation \eqref{eq3.2} is possible only when $\underline{x}_{d_1}=\underline{x}_{d}$. Therefore, there are no reducible curves $C_d$ for a general choice of the configuration $\underline{x}_{d}$.  
This last claim finishes our proof of Proposition~\ref{Prop-Kollar}.
    \end{proof}

\subsection{Moduli spaces and enumeration of balanced rational  curves}
In this technical part, we show, by specializing our configuration of points, that we can recover the genus~$0$ Gromov--Witten invariants of the $3$-dimensional del Pezzo variety $X$ and of a non-singular del Pezzo surface in $|-\frac{1}{2}K_X|$; see \cite[Section 5]{brugalle2016pencils}. 

\subsubsection{Moduli spaces of stable maps and balanced rational curves} 

Consider the moduli space $\mathcal{M}_{0,k_d}(X,d)$  of irreducible genus $0$ $k_d$-pointed stable maps
$$f\colon \left(\mathbb{C}P^1;p_1,\ldots,p_{k_d}\right)\lra X$$
of homology class $d\in H_2(X;\Z)$. Let $\mathcal{M}^*_{0,k_d}(X,d)$ denote the subspace of $\mathcal{M}_{0,k_d}(X,d)$ consisting of  
stable maps  that are simple, \textit{i.e.}, without multiply covered components. Let $\overline{\mathcal{M}}_{0,k_d}(X,d)$ represent the compactification of  $\mathcal{M}_{0,k_d}(X,d)$ due to M.~Kontsevich, and let $\ev_X$ denote the corresponding total evaluation map. The configuration of points $\underline{x}_d$ in $X$ is said to be generic in the sense that every stable map $[(\C P^1;p_1,\ldots,p_{k_d});f]$ in $\overline{\mathcal{M}}_{0,k_d}(X,d)$ satisfying   $\underline{x}_d=f(\{p_1,\ldots,p_{k_d}\})$ is a regular point of  $\ev_{X}$. As a consequence, the value of the invariants $\GW_X(d)$ can be identified with the enumeration of regular elements of $\ev_{X}$.

\begin{definition}
A rational curve parametrized by $f\colon \C P^1\to X$  is \emph{balanced} if $f$ is balanced, \textit{i.e.}, $f$ is an immersion and the normal bundle  $f^*TX/T\C P^1$ is isomorphic to the direct sum of two holomorphic line bundles of the same degree.
\end{definition}

The following lemma is a consequence of \cite[Proposition 2.2]{welschinger2005stongly} with the observation that the complex structures on del Pezzo threefolds are generic enough for the regularity in this proposition. 

\begin{lemma}[\textit{cf.}~Welschinger \cite{welschinger2005stongly}] \label{Lemma-Regular.Balanced} 
An element $[(\C P^1;p_1,\ldots,p_{k_d});f]$ in $\overline{\mathcal{M}}_{0,k_d}(X,d)$ is a regular point of the map $\ev_{X}$ if and only if $f$ is balanced.
\end{lemma}

 \begin{lemma} \label{Lemma-not.intersection.balanced}
  Let $f\colon \mathbb{C}P^1\to X$ be an immersion such that $f(\C P^1)\subset \Sigma$. If $f(\C P^1)$ is not a complete intersection in $X$, then $f$ is balanced.
   \end{lemma}    

\begin{proof}
 Let $\mathcal{N}_f=f^*TX/T\C P^1$ and   $\mathcal{N}'_f=f^*T\Sigma/T\C P^1$ be normal bundles.
We consider the following short exact sequence of holomorphic vector bundles over $\C P^1$:
\begin{equation*}\tag{*}\label{*}
    0\longrightarrow \mathcal{N}'_f \longrightarrow \mathcal{N}_f  \longrightarrow f^*TX/T\Sigma \longrightarrow 0.
\end{equation*}

\begin{claim}{1}\label{claim3.5-1}
  The map $f$ is balanced if and only if the short exact sequence \eqref{*} does not split.
  \end{claim}

Suppose that $f_*[\C P^1]=D\in H_2(\Sigma;\Z)$.
We have observed that the rank $2$ normal bundle $\mathcal{N}_f$ has degree $c_1(X)\cdot d-2=2k_d-2$, where $d=\psi(D)\in H_2(X;\Z)$. By definition,  the map $f$ is said to be balanced if the normal bundle $\mathcal{N}_f$ is isomorphic to $\mathcal{O}_{\C P^1}(k_d-1)\oplus \mathcal{O}_{\C P^1}(k_d-1)$. Furthermore, the normal line bundle $\mathcal{N}'_f$ is isomorphic to $\mathcal{O}_{\C P^1}(c_1(\Sigma)\cdot D -2)$ and thus has degree $c_1(\Sigma)\cdot D -2=k_d-2$. It follows that the normal line bundle $f^*TX/T\Sigma$ has degree  $(2k_d-2)-(k_d-2)=k_d$. Therefore, if the map $f$ is balanced, the short exact sequence \eqref{*} does not split. Conversely, if the short exact sequence \eqref{*} does not split, then $f$ is balanced. Indeed, we suppose toward a contradiction that $\mathcal{N}_f\cong\mathcal{O}_{\C P^1}(k_1)\oplus \mathcal{O}_{\C P^1}(k_2)$ with $k_1>k_2$, \textit{i.e.}, $k_1>c_1(\mathcal{N}'_f)$. This implies that the morphism $\mathcal{O}_{\C P^1}(k_1)\to \mathcal{N}'_f$ is null. Then, there is an isomorphism $\mathcal{O}_{\C P^1}(k_1)\to \mathcal{N}_f/\mathcal{N}'_f$, where $\mathcal{N}_f/\mathcal{N}'_f\cong \mathcal{O}_{\C P^1}(k_d)$. Hence, the short exact sequence \eqref{*} splits, and we obtain a contradiction (see \cite[p.~252]{griffiths1983infinitesimal}).

\begin{claim}{2}\label{claim3.5-2}
  If the exact sequence \eqref{*}  splits, then $f(\C P^1)$ is a complete intersection in $X$.
  \end{claim}

The idea of a proof of this claim is to use the first-order deformation of the surface $\Sigma$ in a pencil $\mathcal{Q}$ generated by an elliptic curve $E$, which realizes the homology class $\frac{1}{4}K^2_X$, and to considering the divisor classes restricted on $E$ by the first-order deformation of $f(\mathbb{C}P^1)$ in $\Sigma$.  By extending a result in the proof of \cite[Theorem~(4.f.3)]{griffiths1983infinitesimal}, we see that the following two statements hold: 
\begin{enumerate}
\item If the exact sequence \eqref{*} splits, then the section of $\mathcal{N}_f$ that vanishes at $D.E$ points in the intersection $f(\C P^1)\cap E$ implies that the restricted divisor class  $[f(\mathbb{C}P^1)]|_E$ is constant when we deform $f(\C P^1)$ along the pencil. 
\item If $f(\C P^1)$ is not a complete intersection with $\Sigma$ in $X$, then the restricted divisor class $[f(\mathbb{C}P^1)]|_E$ in $\Pic(E)$ must vary as we deform $f(\C P^1)$. Indeed, let $f_\epsilon(\C P^1)$ be a first deformation of $f(\C P^1)$ in the pencil, \textit{i.e.}, $f_\epsilon(\C P^1)\subset \Sigma_\epsilon\in \mathcal{Q}$. One has
\begin{equation*}
\begin{aligned}
      &2(L.S)\left[f_\epsilon(\C P^1)\right]|_{E,\Sigma_\epsilon}\\
    &\sim (L.S) \overline{D}_{\phi}- (L.S)(D.S)S|_{E,\Sigma_\epsilon}\\
    &\sim (L.S)\overline{D}_{\phi}-(D.S)\overline{L}_{\phi}+ 2(D.S)L|_{E,\Sigma_\epsilon}.
\end{aligned}
\end{equation*}
  By Lemmas~\ref{lem-constant.map} and~\ref{lem-existence.of.L} and by the fact that $L|_{E,\Sigma_\epsilon}$ is not a constant as we deform $\Sigma_\epsilon$, the statement follows.
\end{enumerate}

By the sufficient condition of Claims~\ref{claim3.5-1} and~\ref{claim3.5-2}, if $f(\C P^1)$ is not a complete intersection, the exact sequence \eqref{*} does not split, and hence $f$ is balanced.
 \end{proof}

\begin{remark}
       If an irreducible curve $C$ is a complete intersection of a non-singular surface of the  pencil $\mathcal{Q}$ in~$X$, then $[C].S=0$.
    \end{remark}

\subsubsection{Enumeration of balanced rational curves}

 In this part, we demonstrate that it is possible to choose a configuration of points on $E$ such that every balanced, irreducible rational curve passing through this configuration is a regular point of the corresponding total evaluation map.  
 
     We denote by $V_{k_d}$ the set of configurations of $k_d$ distinct points on $E$. Recall that for each homology class $D\in H_2(\Sigma;\Z)$ satisfying $D\in \psi^{-1}(d)$, where $\psi\colon H_2(\Sigma;\Z)	\twoheadrightarrow H_2(X;\Z)$, we have $k_D=k_d-1$ (see Lemma~\ref{lemkDandkd}).
     Let $\underline{y}_D\subset E^{k_D}$ be a configuration of $k_D$ distinct points on $E$.  For every non-singular surface $\Sigma_t$ in the pencil~$\mathcal{Q}$, let $\mathcal{C}_{\Sigma_t,D}(\underline{y}_D)$ be the set of $k_D$-pointed stable maps $f\colon (C_0;p_1,\ldots,p_{k_D})\to \Sigma_t$, where $C_0$ is a connected nodal curve of arithmetic genus $0$,
   such that $f(C_0)$ represents the homology class $D\in H_2(\Sigma_t;\Z)$  
    and $f(\{p_1,\ldots,p_{k_D}\})=\underline{y}_D$.
    
\begin{proposition} \label{Prop.Modulispace1}
There exists a dense open subset $U_{k_D}$ of\, $V_{k_D}$ such that for every configuration $\underline{y}_D$ in $U_{k_D}$ and for every element $f$ in $\mathcal{C}_{\Sigma_t,D}(\underline{y}_D)$, one has $C_0= \C P^1$ and $f$ is an immersion.
\end{proposition}

 \begin{proof}
   This proof is inspired by the proof of \cite[Proposition~5.2]{brugalle2016pencils}.
   Let $\mathcal{V}_D$ denote the set of irreducible nodal rational curves of homology class $D\in H_2(\Sigma_t;\Z)$. The set $\mathcal{V}_D$ is a quasiprojective subvariety of dimension $k_D$ of the linear system $|D|$. 
    Let $\overline{\mathcal{V}_D}$ be the Zariski closure of $\mathcal{V}_D$.
    
   Let $\mathcal{U}_D:=|D|_E|$ denote the complete linear system on $E$ defined by the restriction of the divisor $D$. Since the degree of $\mathcal{U}_D$ is $D.E=k_D+1$, it follows from the Riemann--Roch theorem that each element of $\mathcal{U}_D$ is determined by $k_D$ of its points on $E$. Hence, every configuration $\underline{y}_D\in V_{k_D}$ induces an element $[ \underline{y}_D]\in \mathcal{U}_D$. Consider the following generically finite map:
   \begin{align*}
    h\colon \overline{\mathcal{V}_D} &\lra \mathcal{U}_D\\
    C_d&\longmapsto [C_d\cap E] .
\end{align*}

   Note that $\dim (h(\overline{\mathcal{V}_D}\setminus \mathcal{V}_D))\leq k_D-1$. It follows that if the configuration $\underline{y}_D\in V_{k_D}$ has the property $[{\underline{y}}_D] \in\mathcal{U}_D\setminus h(\overline{\mathcal{V}_D}\setminus \mathcal{V}_D)$, then 
    for every stable map $f\in \mathcal{C}_{\Sigma_t,D}(\underline{y}_D)$, the curve $f(C_0)$ must be an irreducible nodal rational curve in $\Sigma_t$. 
    In other words, the curve $C_0$ is non-singular. As a result, $C_0=\C P^1$ and $f$ is an immersion.
 \end{proof}
 
 \begin{corollary}\label{Corollary.Modulispace1}
 Let $\underline{y}_D\in U_{k_D}$ be a configuration of points as in Proposition~\ref{Prop.Modulispace1}. 
 Then, each non-singular surface $\Sigma_t$ in the pencil $\mathcal{Q}$ that contains a curve of class $D$ of the set $\mathcal{C}(\underline{x}_d)$  contains  precisely $\GW_{\Sigma_t}(D)$ such curves.
    \end{corollary}

 \begin{proof}
We consider the moduli space of stable maps $\mathcal{M}^*_{0,k_D}(\Sigma_t,D)$ and its corresponding total evaluation map $\ev_{\Sigma_t}$. By Proposition~\ref{Prop.Modulispace1} and Lemma~\ref{Lemma-Regular.Balanced}, it follows that for every map $f\in\mathcal{C}_{\Sigma_t,D}(\underline{y}_D)$, where $\underline{y}_D\in U_{k_D}$, the point $[(\C P^1;p_1,\ldots,p_{k_D});f]$ is a regular point of  $\ev_{\Sigma_t}$ provided that $f$ is an immersion. Consequently, we obtain the equality $|\mathcal{C}_{\Sigma_t,D}(\underline{y}_D)|=\GW_{\Sigma_t}(D)$.
 \end{proof}
 
     \begin{proposition}\label{Prop.Modulispace2}
Regular values of the total evaluation map $\ev_{X}$ contained in $V_{k_d}$ form a dense open subset $U_{k_d}\subset V_{k_d}$.
In particular, if\, $\underline{x}_d\in U_{k_d}$, then the cardinality of the set $\mathcal{C}_{d}(\underline{x}_d)$ is equal to $\GW_X(d)$.
    \end{proposition}

    \begin{proof}
    Let $\underline{x}_d$ be a configuration of $k_d$ distinct points on $E$ for which the conclusion of Proposition~\ref{Prop-Kollar} holds. 
    By Lemma~\ref{Lemma-Regular.Balanced}, an element $[(\C P^1;p_1,\ldots,p_{k_d});f]$ is a regular point of $\ev_X$ if and only if $f$ is a balanced immersion. It should be noted that if $\underline{x}_d$  is replaced  by another linearly equivalent configuration  $\tilde{\underline{x}}_d$ of points on $E$, then the set of surfaces of the pencil $\mathcal{Q}$ containing a curve in $\mathcal{C}(\tilde{\underline{x}}_d)$ coincides with the set of surfaces containing a curve in $\mathcal{C}(\underline{x}_d)$.

Furthermore, by the Riemann--Roch theorem and Proposition~\ref{Prop.Modulispace1}, we may choose ${\underline{x}_d}$ generically so that every curve in $\mathcal{C}({\underline{x}_d})$  is parametrized by an algebraic immersion (see also  \cite[Proof of Proposition 5.3]{brugalle2016pencils}). The property that $f$ is balanced follows from Lemma~\ref{Lemma-not.intersection.balanced}.
    \end{proof}

 %%%%%%%%%%%%%%%%%%%%%%%%%%%%%%%%%%%%%%%%%%%%%%%%%%
  \section{The sign problems}\label{Sect-Welschinger}
  
  \subsection{Welschinger's signs of  real curves on real del Pezzo surfaces} \label{subsect-Welschinger.surfaces}

Let $(\Sigma,\tau)$ be a real del Pezzo surface. 
Let $D\in H_2^{-\tau}(\Sigma;\Z)$, and let $k_D=c_1(X)\cdot D-1$ 
be the associated integer. 
Let $C$ be an irreducible real rational curve in $X$ that represents the homology class $D$ and is nodal, \textit{i.e.}, has only ordinary real nodes as singularities.
By the adjunction formula, the total number of complex nodes of $C$, which consist of real nodes and pairs of complex conjugate nodes, is equal to
$$g(D)=\tfrac{1}{2}(-c_1(\Sigma)\cdot D+D^2+2).$$  
 A real node of $C$ is either hyperbolic or elliptic in the following sense.

 \begin{definition}
   A \emph{hyperbolic real node} of a curve is  a real point where the 
   local equation of the curve is given by $x^2-y^2=0$ over $\R$ (\textit{i.e.}, it is the intersection of two real branches).
   An \emph{elliptic real node} (or an isolated real node) of a curve is a real point where the  local equation of the curve is given by $x^2+y^2=0$ over  $\R$ (\textit{i.e.}, it is the intersection of two complex conjugate branches).
\end{definition}

 Let  $\delta_E(C)$ denote the number of elliptic real nodes of the curve $C$, $\delta_H(C)$  the numbers of hyperbolic real nodes, and $\delta_{\C}(C)$ the numbers of pairs of complex conjugate nodes. The following equation holds:
 $$\delta_E(C)+\delta_H(C)+2\delta_{\C}(C)=g(D).$$

\begin{definition} \label{def-sign.curve.w.node.dim2}
The \emph{Welschinger's sign} of an irreducible real rational curve $C$ in $(\Sigma,\tau)$, denoted by $s_{\R \Sigma}(C)$, is defined as  
$$s_{\R \Sigma}(C):=(-1)^{\delta_E(C)}.$$
\end{definition}

Let $\underline{x}_D$ denote a real generic  configuration of $k_D$ points consisting of $r$ real points and $l$ pairs of complex conjugate points in $X$, where $r+2l=k_D$.  
Let $\mathcal{R}(\underline{x}_D)$ be a set of real  rational curves in $\Sigma$ representing the homology class $D$ and passing through $\underline{x}_D$.
It should be noted that the genericity of $\underline{x}_D$  implies that every real curve in $\mathcal{R}(\underline{x}_D)$ is nodal and irreducible. 

\begin{definition}\label{def-Welschinger.Inv.Nodes}
Let $(\Sigma,\tau)$ be a real del Pezzo surface and $D\in H_2^{-\tau}(\Sigma;\Z)$. We associate to each real curve $C\in \mathcal{R}(\underline{x}_D)$ its Welschinger's sign $s_{\R \Sigma}(C)$. The \emph{Welschinger invariant}  of $(\Sigma,\tau)$ of the homology class $D$ is defined as 
 $$W_{\R \Sigma}(D,l):=\sum_{C\in \mathcal{R}(\underline{x}_D)}s_{\R \Sigma}(C).$$
\end{definition}

\begin{remark}
If either $D\in H_2(\Sigma;\Z) \setminus H_2^{-\tau}(\Sigma;\Z)$ or $k_D\leq 0$, the corresponding Welschinger invariants $W_{\R \Sigma} (D,l)$ vanish.
\end{remark}

\begin{theorem}[\textit{cf.} Welschinger \cite{welschinger2005fourfolds}, Brugall\'e \cite{brugalle2021invariance}]
The Welschinger invariants in Definition~\ref{def-Welschinger.Inv.Nodes} do not depend on the choice of a real generic  configuration of points if the number of real points or, equivalently, the number of complex conjugate points is given.
\end{theorem}

\subsection{Welschinger's signs of  real curves on real del Pezzo varieties of dimension 3}\label{Sect-spinor.dim.3.with.normal.bundle}
According to J.-Y.~Welschinger, see for example \cite[Section 2.2]{welschinger2005spinor} or \cite[Section 1.2]{welschinger2005stongly}, given a $\Spin_n$- or $\Pin^\pm_n$-structure on the real part of real projective Fano varieties,\footnote{They are examples of strongly semipositive real symplectic manifolds as described in \cite{welschinger2005stongly}.} we may define a sign (also known as the spinor state) $\pm 1$ for each immersed balanced real rational curve on such varieties.

\subsubsection{Spinor states of balanced real rational curves on a real 3-dimensional del Pezzo variety} \label{subsect-Spinor.states}
Let $(X,\tau)$ be a real $3$-dimensional del Pezzo variety with $\R X\neq \emptyset$. We equip the real part $\R X$ with a Riemannian metric $g_X$ and consider the tangent vector bundle $T \R X$ over $\R X$. We restrict our attention to cases in which  $\R X$ is orientable and equip $\R X$ with a $\Spin_3$-structure, denoted by $\mathfrak{s}_X:=(\Spin(T \R X,\mathfrak{o}_X),q_X)$. Here, with a slight abuse of notation, we use  $\mathfrak{o}_X$ to denote an orientation on $T \R X$.

 Let $d \in H_2^{-\tau}(X;\Z)$, and let $k_d=\frac{1}{2}c_1(X)\cdot d$ be the associated integer. Let $f\colon (\C P^1,\tau_1)\to (X,\tau)$ be a real balanced immersion. Consider the associated short exact sequence of holomorphic vector bundles over $\C P^1$ 
 $$0\lra T \C P^1 \lra f^*TX\lra \mathcal{N}_f\lra 0, $$ 
 where $\mathcal{N}_f=f^*TX/T\C P^1$ is a rank $2$ holomorphic normal vector bundle  over $\C P^1$.
 Since both $X$ and $f$ are real, there is a corresponding short exact sequence of real vector bundles over $\R P^1$
$$0\lra T \R P^1 \lra f|_{\R P^1}^*T\R X \lra \R \mathcal{N}_f\lra 0. $$ 
 Up to homotopy,  we have the decomposition  $f|_{\R P^1}^*T\R X = T \R P^1 \oplus \R \mathcal{N}_f$. Assume that $f_*[\C P^1]=d$. 
 Let $C= \Im(f)$ with $\R C \neq \emptyset$. The image  $f(\R P^1) \subset \R X$ is an immersed knot. In order to define a spinor state for the curve $C$, we first establish a consistent method for decomposing the real vector bundle  $f|_{\R P^1}^*T\R X$ into a sum of real line subbundles.  We then proceed to define a loop of orthonormal frames for this bundle as follows. 

 It should be noticed that the balancing property  of $\mathcal{N}_f$ implies its projectivization  $\mathbb{P}(\mathcal{N}_f)$ is the $\supth{0}$ Hirzebruch surface $\C P^1 \times \C P^1$.  Up to real isotopy, there are two holomorphic real line subbundles of $ \mathcal{N}_f$ of degree $k_d-2$ whose projectivization in $ \mathbb{P}(\mathcal{N}_f)$ is a section of bidegree $(1,1)$ in the basis $(L_1,L_2)=([\C P^1\times\{p\}],[\{q\}\times \C P^1])$, where $p,q\in \C P^1$. Note that the orientation on $\R X$ and the orientation on $T \R P^1$ induce an orientation on  $\R \mathcal{N}_f$. Then, these two real isotopy classes are characterized by the direction in which the real part of a fiber rotates in $\R \mathcal{N}_f$. We denote by $E^+$ the isotopy class of the real part of the real line subbundle of  $\mathcal{N}_f$ whose real part of a fiber rotates positively  in $\R \mathcal{N}_f$, and by $E^-$ the isotopy class corresponding to the opposite orientation. We choose a real line subbundle  $L\subset \R \mathcal{N}_f$ as follows:  
\begin{itemize}
    \item If $k_d$ is even, $L$ is chosen in the isotopy class $E^+$. 
    \item If {$k_d$ is odd},  $L$ is chosen in the isotopy class $E^-$. 
\end{itemize}

In both cases, there exists a unique choice of a real line subbundle $\tilde{L}$ in $\R \mathcal{N}_f$ such that $$f|_{\R P^1}^*T \R X=T \R P^1 \oplus L \oplus \tilde{L}.$$

We choose some trivialization of each factor of the above decomposition of $f|_{\R P^1}^*T \R X$ given by non-vanishing sections $v_i\;(i \in \{1,2,3\})$ and then define a  loop of direct orthogonal frames of the vector bundle $ f|_{\R P^1}^*T \R X$. Assume that  $v_1(\xi)\in T\R P^1$ ($\xi\in \R P^1$)  and $(v_1,v_2,v_3)$ forms a direct basis. In the case where $k_d$ is even, we may choose\footnote{ This choice is not unique since we can also choose $-v_2(\xi)\in \R L$.} 
a non-vanishing section $v_2(\xi)\in L$ such that $(v_1(\xi),v_2(\xi))_{\xi\in \R P^1}$  is an orthogonal section of $T \R P^1 \oplus L$. Then, the unique non-vanishing section $v_3(\xi)\in \tilde{L}$ forms a loop of direct orthogonal frames $(v_1(\xi),v_2(\xi),v_3(p))_{\xi\in \R P^1}$ of the vector bundle $ f|_{\R P^1}^*T \R X$; see Figure~\ref{fig-spinor}. In the case where $k_d$ is odd, we first choose the non-vanishing sections over $[0,\pi]\subset \R P^1$ given by  $(v_1(\xi),v_2(\xi),v_3(\xi))_{\xi\in [0,\pi]\subset \R P^1}$. Then we apply to this basis the action of $\begin{psmallmatrix} 1& 0&0\\ 0&\cos(\xi) & \sin(\xi) \\ 0 & -\sin(\xi) & \cos(\xi) \end{psmallmatrix} \in \SO_3(\R)$.
We hence obtain
a new decomposition of $f|_{\R P^1}^*T \R X$ as a direct sum of orientable real line bundles.

\begin{figure}[ht!]
 \centering
\includegraphics[width=0.3\textwidth]{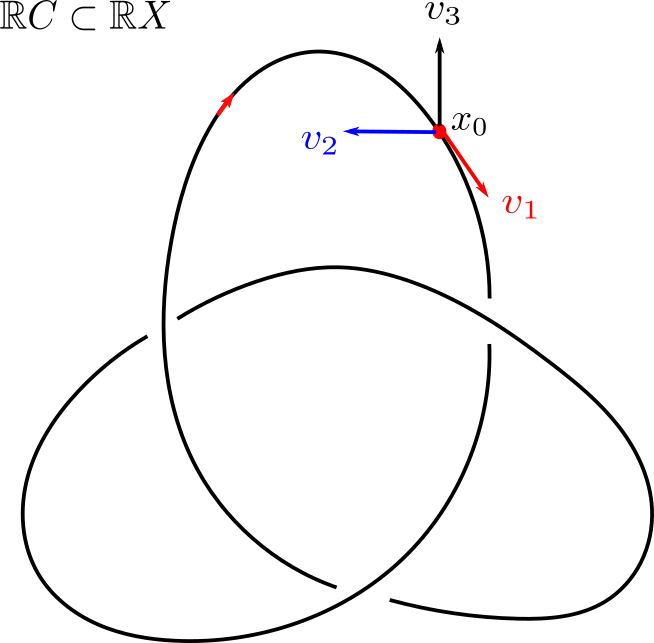}
\caption{A direct orthogonal frame $(v_1(x_0),v_2(x_0)),v_3(x_0)$ at $x_0=f(\xi)\in \R C$ of the vector bundle $T \R X|_{\R C}$.}
\label{fig-spinor}
\end{figure}

  Note that the pull-back under the immersion $f$ of a principal $\SO_3$-bundle $\SO(T\R X, \mathfrak{o}_{X})$ over $\R X$ is itself a principal $\SO_3$-bundle, denoted by $f^*\SO(T\R X,\mathfrak{o}_{X})$, over $\R P^1$. Moreover, a $\Spin_3$-structure on $\SO(T\R X,\mathfrak{o}_X)$ induces a corresponding $\Spin_3$-structure on this bundle. Accordingly, we refer to a loop in the above construction as a \textit{loop of the principal $\SO_3$-bundle} $f^*\SO(T\R X,\mathfrak{o}_X)$.  With these preliminaries in place, we are now prepared to define the spinor state of the curve  $C$.

  \begin{definition} \label{def-spinor3dim}
 Let  $((X,\tau),\mathfrak{s}_X,\mathfrak{o}_X))$ denote a real $3$-dimensional del Pezzo variety whose real part is orientable and equipped  with a $\Spin_3$-structure $\mathfrak{s}_X=(\Spin(T \R X,\mathfrak{o}_X),q_X)$. Let $C=\Im(f)$ be a real balanced irreducible rational curve immersed in $X$. Let $(v_1(\xi),v_2(\xi),v_3(\xi))_{\xi\in \R P^1}$ be a loop of the principal $\SO_3$-bundle $f^*\SO(T\R X,\mathfrak{o}_X)$ defined as above. The \emph{spinor state} of $C$, denoted by $\sp_{\mathfrak{s}_X,\mathfrak{o}_X}(C)$, is given by
  $$ \sp_{\mathfrak{s}_X,\mathfrak{o}_X}(C)=\begin{cases}
		+1 & \text{ if this loop lifts to a loop of } f^*\Spin(T\R X,\mathfrak{o}_X),\\
           -1 & \text{ otherwise}.
		 \end{cases}$$
   \end{definition}

\begin{remark}
The spinor state of a real rational curve in Definition~\ref{def-spinor3dim} does not depend on the choice of either the sections $v_i \;(i\in \{1,2,3\})$ or the metric $g_X$ or the parametrization $f$.  This spinor state depends on the choice of a $\Spin_3$-structure on $\R X$ as well as an orientation on $\R X$.
\end{remark}

Geometrically, reversing the orientation of $\R X$ interchanges the two isotopy classes $E^+$ and $E^-$. Since $E^+$ and $E^-$ differ by precisely one full rotation---that is, by a generator of $\pi_1(\SO_2(\R))$---the corresponding loops in $\pi_1(\SO_3(\R))$ differ by a non-trivial element. Moreover, recall that the set of $\Spin_3$-structures on the bundle $\SO(T\R X,\mathfrak{o}_X)$ is in one-to-one correspondence with the group $H^1(\R X;\Z_2)$, which is determined by mapping in circles.  We summarize these properties in the following proposition.

\begin{proposition}[\textit{cf.} {\cite[Lemma 1.7]{welschinger2005stongly}}]\label{Prop-SpinorStates1}
Let $\eta\in H^1(\R X;\Z_2)$. One has
 $$\sp_{\mathfrak{s}_X,\mathfrak{-o}_X}(C)=-\sp_{\mathfrak{s}_X,\mathfrak{o}_X}(C)$$
 and
 $$\sp_{\eta\mathfrak{s}_X,\mathfrak{o}_X}(C)=(-1)^{\eta\cdot (f_*[\R P^1])}\sp_{\mathfrak{s}_X,\mathfrak{o}_X}(C).$$
\end{proposition}

  \begin{example}
Consider $(\cpp,\tau_3)$, whose real part is $\R P^3$.    Since $H_1(\R P^3;\Z_2)= \Z_2$, there exist precisely two $\Spin_3$-structures on $\SO(T \R P^3,\mathfrak{o}_{\R P^3})$.   We may choose a $\Spin_3$-structure $\mathfrak{s}_{\R P^3}$ on $\SO(T \R P^3,\mathfrak{o}_{\R P^3})$ such that the spinor state of any line $L$ in $\cpp$ is $\sp_{\mathfrak{s}_{\R P^3},\mathfrak{o}_{\R P^3}}(L)=+1$. This choice coincides with the trivialization of  $T \R P^3$ considered by Brugall\'e--Georgieva in \cite{brugalle2016pencils}.
\end{example}

 Let $\underline{x}_d$ be a real generic configuration of $k_d$ points consisting of $r$ real points and $l$ pairs of complex conjugate points in $X$, where $r+2l=k_d$. 
  Denote by $\mathcal{R}(\underline{x}_d)$ the set of real balanced irreducible  rational curves in $X$ representing the homology class $d$ and passing through $\underline{x}_d$.

   \begin{definition} \label{def-W.inv.dim3}
   Let $((X,\tau),\mathfrak{s}_X,\mathfrak{o}_X))$ denote a real $3$-dimensional del Pezzo variety and $d\in H_2^{-\tau}(X;\Z)$. We associate each real curve $C\in \mathcal{R}(\underline{x}_d)$ to its spinor state $\sp_{\mathfrak{s}_X,\mathfrak{o}_X}(C)$ as in Definition~\ref{def-spinor3dim}.   The \emph{Welschinger invariant} of  $((X,\tau),\mathfrak{s}_X,\mathfrak{o}_X))$ of the homology class $d$ is defined as
   $$W_{\R X}^{\mathfrak{s}_X,\mathfrak{o}_X}(d,l):=\sum_{C\in \mathcal{R}(\underline{x}_d)}\sp_{\mathfrak{s}_X,\mathfrak{o}_X}(C).$$
   \end{definition}
   
   \begin{remark}
If either $d\in H_2(X;\Z)\setminus H_2^{-\tau}(X;\Z)$ or $k_d\leq 0$,  the corresponding Welschinger invariants $W_{\R X}^{\mathfrak{s}_X,\mathfrak{o}_X}(d,l)$ vanish.
\end{remark} 

   \begin{theorem}[Welschinger, \textit{cf.} \cite{welschinger2005spinor,welschinger2005stongly}]
The Welschinger invariant in Definition~\ref{def-W.inv.dim3} does not depend on the choice of a real generic  configuration $\underline{x}_d$  up to isotopy if the number of real points is given. This invariant depends on the $\Spin_3$-structure $\mathfrak{s}_X$ on $\SO(T\R X,\mathfrak{o}_X)$, 
the homology class $d$ and the number $l$.
\end{theorem}

\subsubsection{Spinor states of rational curves: Revisited} \label{subsect-Spinorstates.revisited}
In this paragraph, we assume that $C$ is a real balanced rational curve parametrized by a real balanced immersion $f\colon (\C P^1,\tau_1)\to (X,\tau)$ such that $f_*[\C P^1]=d\in  H_2^{-\tau}(X;\Z)$ and $f (\C P^1)\subset \Sigma$, where $\Sigma\in |-\frac{1}{2}K_X|$ is a real non-singular surface. The curve $C$ also represents a homology class $D\in H_2^{-\tau}(\Sigma;\Z)$, where $D\in \psi^{-1}(d)$ and $\psi\colon  H_2(\Sigma;\Z)\twoheadrightarrow H_2(X;\Z)$ is the natural projection.  
We will examine certain properties  of the holomorphic line subbundle $\mathcal{N}'_f=f^*T\Sigma/T\C P^1$, which is a holomorphic line subbundle of the normal bundle $\mathcal{N}_f$,   and subsequently relate this subbundle to the definition of spinor states of curves.

\begin{remark}
  The isotopy class of the normal bundle of $C$ in $\Sigma$ depends only on the homology class of $C$ in $H_2(\Sigma;\mathbb{Z})$, and not on the particular choice of the curve $C$ itself. We therefore denote this class as $\mathcal{N}_{D/\Sigma}$ instead of $\mathcal{N}_{C/\Sigma}$. With this notation, we have $\mathcal{N}'_f=f^*\mathcal{N}_{D/\Sigma}$.
\end{remark}

By the adjunction formula, we have
$$\deg(\mathcal{N}_{D/\Sigma})=c_1(\Sigma)\cdot D-2=k_D-1.$$ 
It follows that $\deg(\mathcal{N}'_f)=k_d-2$. Consequently, the real line bundle $\R \mathcal{N}'_f$  is orientable if  $k_D$ is odd and non-orientable otherwise. Recall from Section~\ref{subsect-Spinor.states} that there are two isotopy classes, denoted by  $E^+$ and $E^-$, of the real part of a holomorphic real line subbundle  of degree $k_d-2$ of $\mathcal{N}_{f}$. These classes are characterized by the direction in which the real part of a fiber rotates within $\R \mathcal{N}_{f}$. Also note  that $k_D=k_{T(D)}$ for any pair $\{D, T(D)\}\subset H_2(\Sigma;\Z)$, where $T$ denotes the monodromy transformation  $T(D)=D+(D.S)S$.

\begin{definition} \label{def-epsilon.function}
Let $f\colon  \C P^1 \to \Sigma\subset X$ be a real immersion such that $f_*[\C P^1]=D \in H_2^{-\tau}(\Sigma;\Z)$, and let $\R \mathcal{N}'_f=\R f^*\mathcal{N}_{D/\Sigma}$ be as described above. We define the following  homomorphism:
\begin{align*}
    \epsilon\colon H_2(\Sigma;\mathbb{Z})&\lra \mathbb{Z}_2\\
    D&\longmapsto \begin{cases}
  k_D+1 \mod 2&\text{if } \R \mathcal{N}'_f \text{ realizes the isotopy class } E^+,\\
  k_D \mod 2& \text{if } \R \mathcal{N}'_f\text{ realizes the isotopy class } E^-.
  \end{cases}
\end{align*}
\end{definition}

\begin{lemma} \label{lem-epsilon.L.and.TL}
Let $\{L,T (L)\}\subset H_2(\Sigma;\Z)$ be such that $|L.S|=1$. 
One has
$$\epsilon(T(L))\neq \epsilon(L).$$
\end{lemma}

\begin{proof}
  As established in Corollary~\ref{coro-isomorphism.pencil.picard}, a pair of divisors $L_t|_E\neq T(L_t)|_E$ such that $|L_t.S|=1$ uniquely determines a non-singular surface $\Sigma_t$ in the pencil $\mathcal{Q}$.  Let $\{L_0,T(L_0)\}\subset H_2^{-\tau}(\Sigma_0;\Z)$.  Let $\R \mathcal{N}_a'=\R f^*\mathcal{N}_{L_0/\Sigma}$ and  $\R \mathcal{N}_b'=\R f^*\mathcal{N}_{T(L_0)/\Sigma}$ denote the real parts of the normal bundles of $L_0$ and $T(L_0)$, respectively. 
  Due to the monodromy phenomenon, the roles of $L_0$ and $T(L_0)$ in $H_2^{-\tau}(\Sigma_0;\Z)$ are interchanged when traversing a loop around a singular fiber. 
In particular, as one moves along the pencil from $\Sigma_0$ around a singular surface and returns to $\Sigma_0$, the directions of rotation of $L_0$ and $T(L_0)$  in $\Sigma_0$  are reversed. Consequently, the real line bundles $\R \mathcal{N}_a'$ and $\R \mathcal{N}_b'$ rotate in opposite directions in $\R X$; see Figure~\ref{fig-opposite.rotations}.
\begin{figure}[ht!]
\centering
\includegraphics[scale=0.4]{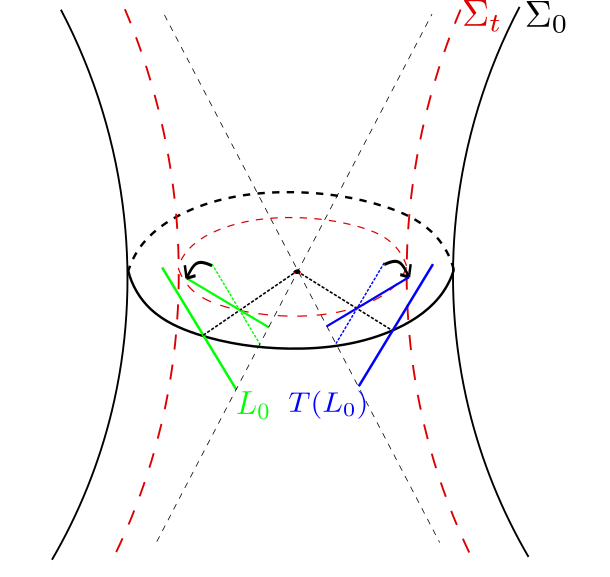}
\caption{$L_0$ and $T(L_0)$ rotate in opposite directions in $\R X$.}
\label{fig-opposite.rotations}
\end{figure}
\end{proof}

We may now assume that $L$ and $T(L)$ are chosen such that $L.S=-1$ and $\R \mathcal{N}'_{L/\Sigma}=E^+$. As a consequence, we have  $T(L).S=1$ and $\R \mathcal{N}'_{T(L)/\Sigma}=E^-$.

 \begin{proposition}\label{prop-function.epsilon1}
For any $\{D,T(D)\}\subset H_2^{-\tau}(\Sigma;\Z)$ such that $D.S\neq 0$, one has
$$\epsilon(T(D))\equiv \epsilon(D)+1 \mod 2.$$
\end{proposition}

 \begin{proof}
 By Lemma~\ref{lem-epsilon.L.and.TL}, the proposition holds in the case where $D=L$. Let $\Sigma_0$ be a real non-singular surface in the pencil $\mathcal{Q}$, and let $\Sigma_t$ denote the first-order real deformation of $\Sigma_0$ within $\mathcal{Q}$.

\begin{claim}{1}\label{claim4.11-1}
The isotopy class of the real line subbundle $\R f^*\mathcal{N}_{D/\Sigma_0}$ in $\R \mathcal{N}_f$  is determined by the direction of the vector $\frac{dD|_{E,\Sigma_t}}{dt}\Big|_{t=0}$.
\end{claim}

A proof of this proposition proceeds analogously to the proof of \cite[Proposition 4.4]{brugalle2016pencils}. 
Specifically, fix a point  $x_0$ in the transverse intersection of $f(\C P^1)$ with the base locus $E$ of the pencil $\mathcal{Q}$ such that $p_0=f^{-1}(x_0) \in \R P^1$. 
Consider the first-order real deformation  $f_t$ of $f$ within the pencil $\mathcal{Q}$ such that for each $t$, $f_t(\C P^1)$ passes through all intersection  points of $f(\C P^1)\cap E$ except $x_0$. This allows us to define a holomorphic line subbundle $M$ of degree $k_d-1$ in $\mathcal{N}_{f}$ characterized by the property that its fiber at $p_0$ is $M|_{p_0}=\sigma(p_0)$, where $\sigma\colon \C P^1\to \mathcal{N}_f$ is a non-vanishing holomorphic section that vanishes precisely at $f^{-1}(E\setminus \{x_0\})$.  Since this deformation stays in the base locus $E$, it follows that  $\sigma(p_0)\in f^*\mathcal{N}_{D/\Sigma_0}$. We may then compare the relative positions of the fibers  $\R M|_{p_t}$ and $\R f^*\mathcal{N}_{D/\Sigma_0}|_{p_t}$, for $p_t\in \R P^1$ in a neighborhood of $p_0$, by demonstrating that the isotopy class of $\R f^*\mathcal{N}_{D/\Sigma_0}$ is determined by the direction of the vector $\sigma(p_0)$.  
 Moreover, the deformation $x_t=f(p_t)$ of $x_0$ as an intersection point of $f_t(\C P^1)\cap E$ is determined by the deformation of the divisor class $D|_{E,\Sigma_t}$ as $t$ approaches $0$. Hence, Claim~\ref{claim4.11-1} follows.

 We now consider the deformation of divisors $D|_{E,\Sigma_t}$ and $T(D)|_{E,\Sigma_t}$ on $\Pic_{D.E}(E)$ as $t$ approaches $0$.

\begin{claim}{2}\label{claim4.11-2}
The direction of vectors $\frac{dD|_{E,\Sigma_t}}{dt}\Big|_{t=0}$ and $\frac{dT(D)|_{E,\Sigma_t}}{dt}\Big|_{t=0}$  are opposite.
\end{claim}

Indeed, this claim follows directly from the fact that $D.S\neq 0$ and that $D|_{E,\Sigma_t}+T(D)|_{E,\Sigma_t}=\overline{D}_{\phi}$ is a constant divisor class for every $t$ (see Figure~\ref{fig-TwoIsotopyClasses}).
\begin{figure}[ht!]
\centering
\includegraphics[scale=0.4]{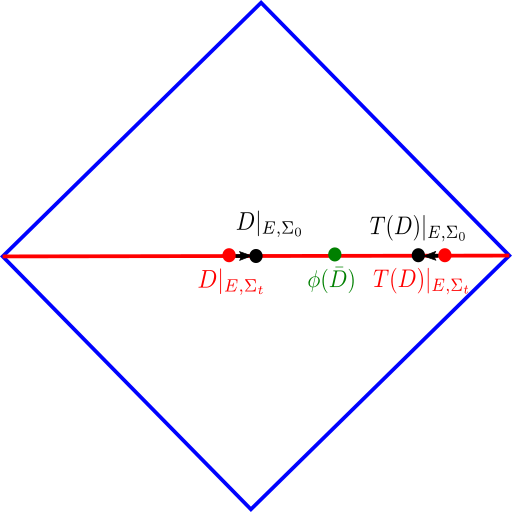}
\caption{Two isotopy classes of a real line subbundle of $\R \mathcal{N}_{f}$ are represented in  $\Pic(E)$ in the case where $\R E$ is connected.}
\label{fig-TwoIsotopyClasses}
\end{figure}

\begin{claim}{3}\label{claim4.11-3}
  The class $D$ belongs to the same isotopy class as $L$ if and only if\, $T(D)$ does not.
  \end{claim}

This claim follows from Lemma~\ref{lem-constant.divisor.L.D}. Indeed, we have 

$$\frac{dD|_{E,\Sigma_t}}{dt}\Big|_{t=0}=|D.S|\frac{dL|_{E,\Sigma_t}}{dt}\Big|_{t=0}.$$
Since $L$ and $T(L)$ belong to different isotopy classes, it follows that $D$ and $T(D)$ do as well. Therefore, the statement of the proposition holds.
\end{proof}

\begin{proposition}\label{prop-function.epsilon2} For any pair $D,\,D'\in H_2(\Sigma;\Z)$ such that $D.S\neq 0$ and $D'.S\neq 0$, one has
$$\epsilon(D)=\epsilon(D')\Longleftrightarrow (D.S)(D'.S)>0.$$
\end{proposition}

\begin{proof}
Consider the pencil $\mathcal{Q}$, and let $\Sigma_0$ denote the fiber over $t_0 \in \mathbb{C}P^1$, while $\Sigma_t$ denotes the fiber over $t \in \mathbb{C}P^1$ in a neighborhood of $t_0$. We establish the following two claims.

\begin{claim}{1}\label{claim4.12-1}
We have  $\frac{dD|_{E,\Sigma_t}}{dt}\Big|_{t=0}=\frac{-D.S}{2}\frac{dS|_{E,\Sigma_t}}{dt}\Big|_{t=0}.$
  \end{claim}

Indeed, this claim is directly implied from the fact that 
$$2D|_{E,\Sigma_0}+(D.S)S|_{E,\Sigma_0}=2D|_{E,\Sigma_t}+(D.S)S|_{E,\Sigma_t}(=\overline{D}_{\phi}).$$

\begin{claim}{2}\label{claim4.12-2}
We have $\frac{dD|_{E,\Sigma_t}}{dt}\Big|_{t=0}\times \frac{dD'|_{E,\Sigma_t}}{dt}\Big|_{t=0}>0$ if and only if $(D.S)(D'.S)>0$.
\end{claim}
  
Indeed, Claim~\ref{claim4.12-1} implies that the  equality 
$$\frac{dD|_{E,\Sigma_t}}{dt}\Big|_{t=0}\times \frac{dD'|_{E,\Sigma_t}}{dt}\Big|_{t=0}= \left( \frac{-D.S}{2}\frac{dS|_{E,\Sigma_t}}{dt}\Big|_{t=0}\right)\times \left( \frac{-D'.S}{2}\frac{dS|_{E,\Sigma_t}}{dt}\Big|_{t=0}\right)$$
holds. In other words, we obtain
$$\frac{dD|_{E,\Sigma_t}}{dt}\Big|_{t=0}\times \frac{dD'|_{E,\Sigma_t}}{dt}\Big|_{t=0}= \frac{(D.S)(D'.S)}{4}\left(\frac{dS|_{E,\Sigma_t}}{dt}\Big|_{t=0}\right)^2.$$
As a consequence, the statement follows.
\end{proof}

%%%%%%%%%%%%%%%%%%%%%%%%%%%%%%%%%%%%%%%%
\subsection{Welschinger's signs of real rational curves on real del Pezzo varieties of dimension 2 and 3: Comparison} \label{Sect-signs.in.comparison}
We adopt the following notational conventions throughout this section. Let $V$ be a rank $n$ vector bundle equipped with a Riemannian metric over a topological space $B$. We denote by $O(V)$ a principal $O_n$-bundle over $B$. If $V$ is orientable with orientation $\mathfrak{o}_V$, we denote by  $\SO(V,\mathfrak{o}_{V})$  a principal $\SO_n$-bundle over $B$. The set of $\Pin^\pm_n$-structures on $O(V)$ is denoted by $\mathcal{P}_n^\pm(V)$, and  the set of $\Spin_n$-structures on $\SO(V,\mathfrak{o}_V)$ is denoted by $\mathcal{SP}_n(V,\mathfrak{o}_V)$. In the special case where $V = TB$ is the tangent bundle, we may, for brevity, refer to $\Pin^\pm$- (respectively, $\Spin$-)structures over $B$ instead of on $O(TB)$ (respectively, $\SO(TB, \mathfrak{o}_{TB})$).

In this subsection, we investigate the properties of the $\Pin_2^-$-structure over $\R \Sigma$, where $\Sigma\in |-\frac{1}{2}K_X|$ is a real non-singular surface, which is restricted from a $\Spin_3$-structure over $\R X$  (see Proposition~\ref{Prop-Spin3.Pin2}). The associated quasi-quadratic enhancement enables us to relate the spinor states of rational curves in $(X,\tau)$ to the Welschinger's signs of these curves in  $(\Sigma,\tau|_\Sigma)$ (see Proposition~\ref{Prop-RelationW2W3}).

\subsubsection{Construction of the quasi-quadratic enhancement corresponding to a $\boldsymbol{\Pin^-_2}$-structure on a  topological surface}\label{subsect-construction_quasiquad_enhancement}

\begin{definition} \label{def-quasiquadratic.enhancement}
  A function $s\colon  H_1(\R \Sigma;\Z_2) \to \Z_2$ is called \emph{a quasi-quadratic enhancement} if it satisfies
  $$s(x+y)=s(x)+s(y)+x.y+(w_1(\R \Sigma)\cdot x)(w_1(\R \Sigma)\cdot y).$$
\end{definition}

It is well known that every non-singular topological surface admits a $\Pin_2^-$-structure on its tangent bundle. We begin by equipping  $T\R \Sigma$ with a $\Pin_2^-$-structure, denoted by $\mathfrak{p}_{T\R \Sigma}:=(\Pin^-(T\R \Sigma),q_{T\R \Sigma})$.  We then proceed to construct a quasi-quadratic enhancement corresponding to such a $\Pin^-_2$-structure.

Let  $\alpha\colon \mathbb{S}^1\to \R \Sigma$ be an immersion, and let $\mathcal{N}_\alpha:=
\alpha^*T \R \Sigma / T \mathbb{S}^1$ denote the normal bundle of this map. 
We consider the short exact sequence of real vector bundles over $\mathbb{S}^1$ determined by $\alpha$
\begin{equation*}
\tag{${\rm ses}_\alpha$}\label{ses_alpha}
    0\lra T \mathbb{S}^1 \lra \alpha^* T \R \Sigma \lra \mathcal{N}_\alpha \lra 0.
\end{equation*}
Up to homotopy, there is a splitting $\alpha^*T \R \Sigma=T \mathbb{S}^1\oplus \mathcal{N}_\alpha$.  We begin by choosing the bounding $\Spin_1$-structure on $T \mathbb{S}^1$, denoted by $\mathfrak{s}_0$, which corresponds to the disconnected double cover of $\mathbb{S}^1$. On the normal line bundle $\mathcal{N}_\alpha$, there exists a canonical  $\Pin^-_1$-structure, denoted by $\mathfrak{p}_0$, which corresponds to $0 \in \mathbb{Z}_2$ under the natural bijection $\mathcal{P}_1^-(\mathcal{N}_\alpha)\to \Z_2$ (see \cite[Lemma 8.2]{solomon2006intersection} or \cite[Section 11]{chen2024spin}). We say that a given $\Pin^-_2$-structure on $\R \Sigma$ is \textit{compatible} with the short exact sequence \eqref{ses_alpha}  if $\alpha^*\mathfrak{p}_{\R \Sigma}$ is mapped from $<\mathfrak{s}_0,\mathfrak{p}_0>_\alpha$ by the following natural $H^1(\mathbb{S}^1;\Z_2)$-biequivariant map:
$$<\cdot,\cdot>_\alpha\colon  \mathcal{SP}_1(T \mathbb{S}^1)\times \mathcal{P}^-_1(\mathcal{N}_\alpha)\lra \mathcal{P}_2^-(\alpha^*\R \Sigma).$$
 We are then prepared to define the following function:
\begin{align*}
  s_{\mathfrak{p}_{\R \Sigma}}\colon \{ \text{immersed circles in }\R \Sigma \} &\lra \mathbb{Z}_2\\
     \alpha&\longmapsto \begin{cases}
  1& \text{if } \alpha^*\mathfrak{p}_{\R \Sigma}=<\mathfrak{s}_0,\mathfrak{p}_0>_\alpha, \\
  0& \text{otherwise. } 
  \end{cases}
\end{align*}

\begin{remark}
  Since  a loop $\tilde{\alpha}$ that lifts $\alpha$ in the principal $O_2$-bundle $\alpha^*(O(T\R \Sigma))$ 
  can be written as  $r$ ($r\in \Z$) times a generator of $\pi_1(O_2(\R))\cong \Z$, we can interpret $s_{\mathfrak{p}_{\R \Sigma}}(\alpha)\in \{0,1\}$ as 
\begin{align*}
     s_{\mathfrak{p}_{\R \Sigma}}(\alpha)= 1+r \mod 2.
\end{align*}
It can be shown that the liftings of $\alpha$ in the principal $\Pin^-_2$-bundle $\alpha^*(\Pin^-(T\R \Sigma))$ are either all loops or all non-closed paths.
\end{remark}

By definition, the function $s_{\mathfrak{p}_{\R \Sigma}}$ is additive on disjoint unions of embedded circles and does not depend on the orientation of $\alpha$.
A union of immersed circles in $\R \Sigma$ can be deformed such that all their intersections and self-intersections are transverse. The rule of smoothing a transverse intersection is shown in Figure~\ref{fig-smoothing}.

\begin{figure}[ht!]
 \centering
\includegraphics[width=0.4\textwidth]{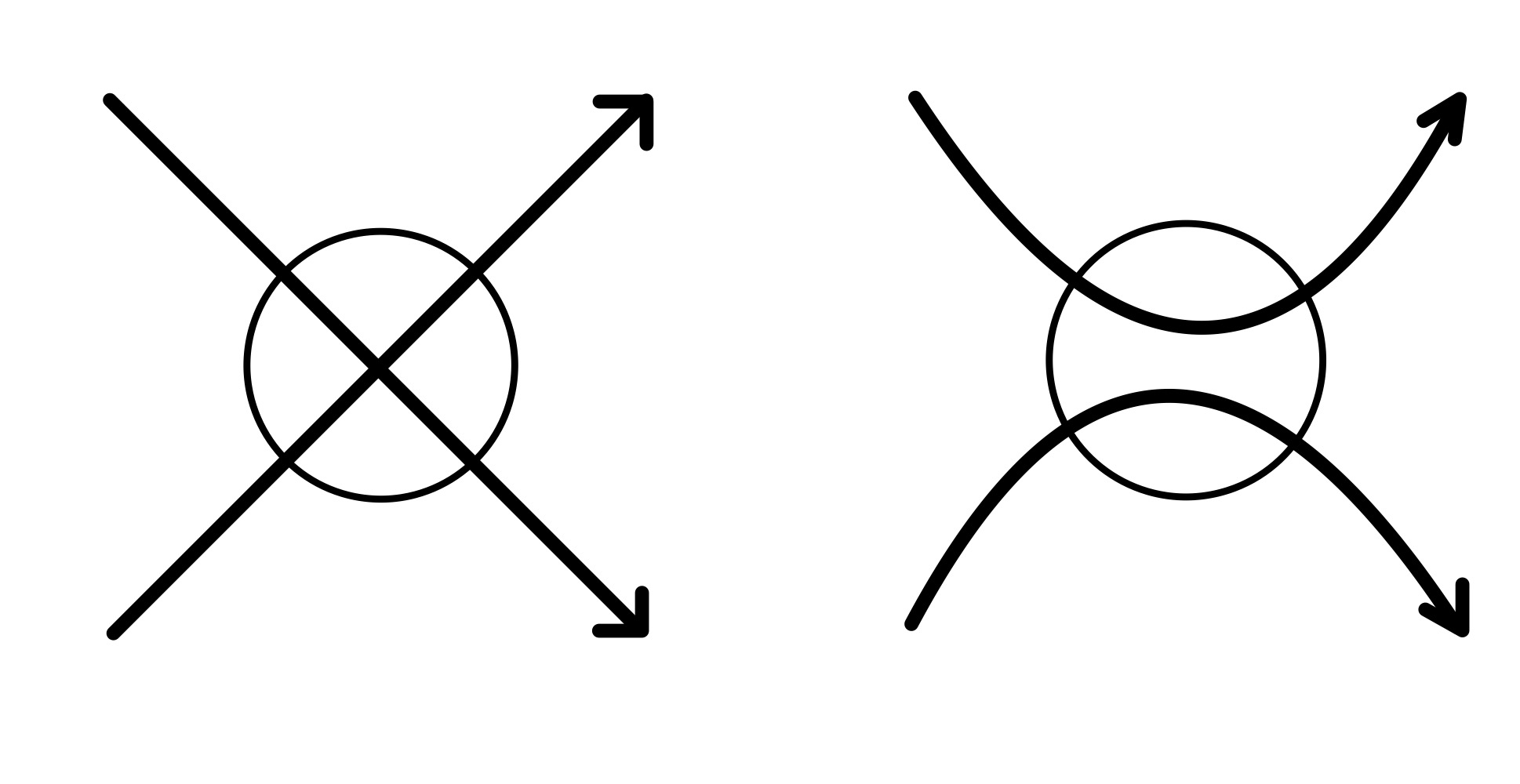}
\caption{Smoothing a transverse intersection.}
\label{fig-smoothing}
\end{figure}

 After smoothing a union $\bigcup_{i=1}^{k} C_{i}$ of $k$ immersed circles in $\R \Sigma$, we obtain a disjoint union $\bigsqcup_{i=1}^{k'} C'_{i}$ of $k'$ embedded circles in $\R \Sigma$.  
 In particular,
 $$\sum\limits_{i=1}^k[C_i]=\sum\limits_{i=1}^{k'}[C_i']\in H_1(\R \Sigma;\Z_2).$$ 
 Let $\delta (\alpha)$ denote the number of self-intersection points of $C=\Im(\alpha)$. It follows that
$$k' \equiv k+\sum_{i=1}^k \delta(\alpha_i)+\sum_{i<j}[C_i].[C_j] \mod 2.$$
We have the following proposition.

\begin{proposition}[\textit{cf.} {\cite[Theorem 11.1]{chen2024spin}}] \label{prop-quasi-quadratic.enhancement.def.prop}
The function $s_{\mathfrak{p}_{\R \Sigma}}$ induces a quasi-quadratic enhancement  $$s^{\mathfrak{p}_{\R \Sigma}}\colon H_1(\R \Sigma;\Z_2)\lra \mathbb{Z}_2$$ given by
$$s^{\mathfrak{p}_{\R \Sigma}}([\alpha])\equiv s_{\mathfrak{p}_{\R \Sigma}}(\alpha)+\delta(\alpha) \mod 2.$$
\end{proposition}
 
\begin{remark}
Given any basis of $H_1(\R \Sigma;\Z_2)$, it is possible to make $s^{\mathfrak{p}_{\R \Sigma}}$ take any specified values on the basis elements by choosing an
appropriate structure $\mathfrak{p}_{\R \Sigma}$.
\end{remark}

\subsubsection{Quasi-quadratic enhancements and restricted $\boldsymbol{\Pin^-_2}$-structures} \label{subsect-quasi.quad-pin2}

The restriction from $\Spin_3$- to $\Pin_2^-$-structures is described in the  following proposition.

\begin{proposition}\label{Prop-Spin3.Pin2}
  A $\Spin_3$-structure on $\SO(T \R X,\mathfrak{o}_{T\R X})$ can be restricted naturally to a $\Pin^-_2$-structure on $O(T \R \Sigma)$.
 
 In particular, if both $\R X$ and $\R \Sigma$ are orientable, then $\SO(T\R \Sigma,\mathfrak{o}_{T \R \Sigma})$ inherits a $\Spin_2$-structure from a $\Spin_3$-structure on $\SO(T\R X,\mathfrak{o}_{T \R X})$.
\end{proposition}

\begin{proof}
It is sufficient to prove that $O(T\R \Sigma)$ can be embedded as a subspace of $\SO(T \R X,\mathfrak{o}_{T \R X})|_{\R \Sigma}$. A proposition of Kirby--Taylor \cite{kirby1990p} establishes the existence of a natural bijection
$$\mathcal{P}_2^-(T \R \Sigma)\cong \mathcal{SP}_3(T \R \Sigma \oplus \det T \R \Sigma,\mathfrak{o}_{T \R \Sigma \oplus \det T \R \Sigma}).$$
Let $ \mathcal{N}_{\Sigma/X}$ be the holomorphic normal line bundle of $\Sigma$ in  $X$. 
The Riemannian metric on $\R X$ allows us to identify the real bundles $\R \mathcal{N}_{\Sigma/X}$ and $ \mathcal{N}_{\R \Sigma/\R X}$.
By the adjunction formula, the bundle isomorphism  $\mathcal{N}_{\R \Sigma / \R X} \cong \det T\R \Sigma$ holds. Up to homotopy, this implies the decomposition
$$T\R X= T \R \Sigma \oplus \det T \R \Sigma.$$
Let $(x,(v_1,v_2))\in O(T\R \Sigma)$.  
There exists a unique vector $v_3 := v_1 \wedge v_2$ in $\det T \R \Sigma$ such that $(x,(v_1,v_2,v_3))$ defines an element of  $\SO(T \R X,\mathfrak{o}_{T \R X})|_{\R \Sigma}$, with  orientation consistent with that induced by the $\Spin_3$-structure on $\SO(T \R X,\mathfrak{o}_{T \R X})$.
\end{proof}

  Let $\mathfrak{s}_{X}:=(\Spin (T\mathbb{R} X,\mathfrak{o}_{T\mathbb{R} X}), q_{T \R X})$ denote a $\Spin_3$-structure on $\SO(T \R X,\mathfrak{o}_{T \R X})$.
 We denote by $\mathfrak{s}_{X|\Sigma}=(\Pin^-(X|\Sigma),q_{X|\Sigma})$ the restricted  $\Pin^-_2$-structure to $O(T \R \Sigma)$ induced from the $\Spin_3$-structure $\mathfrak{s}_{X}$, as described in Proposition~\ref{Prop-Spin3.Pin2}. Recall that there is a natural homomorphism
$$\rho\colon  H_2^{-\tau}(\Sigma;\Z) \lra H_1(\R \Sigma;\Z_2), \quad D\longmapsto \rho D.$$
 Therefore, for each $D \in H_2^{-\tau}(\Sigma;\Z)$ and each $\Pin^-_2$-structure ${\mathfrak{p}_{T\R \Sigma}}$ on $O(T \R \Sigma)$, we may associate a quasi-quadratic enhancement $s^{\mathfrak{p}_{T\R \Sigma}}$  evaluated at  $\rho D$, as described in the previous subsection. In particular, the $\Pin^-_2$-structure $\mathfrak{s}_{X|\Sigma}$ on $O(T \R \Sigma)$ implies the following properties on $s^{\mathfrak{s}_{X|\Sigma}}$.
 
\begin{proposition} \label{prop-quasi-quad.and.monodromy}
Assume that $\ker (\psi\colon  H_2(\Sigma;\mathbb{Z})\to H_2(X;\mathbb{Z}))\cong \Z$. Let $S\in H_2^{-\tau}(\Sigma;\Z)$ be a generator of $\ker (\psi)$. 
One has
\begin{enumerate}
    \item\label{pqq.a.m-1} $s^{\mathfrak{s}_{X|\Sigma}}(\rho S)=0$;
    \item\label{pqq.a.m-2} $s^{\mathfrak{s}_{X|\Sigma}}(\rho ({D+(D.S)S})) \equiv s^{\mathfrak{s}_{X|\Sigma}}(\rho D) +(D.S) \mod 2$ 
      for all $D\in H_2^{-\tau }(\Sigma;\Z)$. 
\end{enumerate}
\end{proposition}

\begin{proof}
\eqref{pqq.a.m-1}~
 The first observation is that $\psi  (S)=0\in H_2^{-\tau}(X;\Z)$ and there is a natural homomorphism 
$$\rho\colon  H_2^{-\tau}(X;\Z) \lra H_1(\R X;\Z_2), \quad d\longmapsto \rho d.$$
Consequently,  $\rho \psi (S)=0\in H_1(\R X;\Z_2)$.    Since $S$ is also the vanishing cycle of a Lefschetz fibration, we have $\rho \psi (S)=\alpha_*[\mathbb{S}^1]$, where $\alpha\colon  \mathbb{S}^1\to \R X$ is an immersion and $\Im(\alpha)$ bounds a $2$-dimensional disk $\mathbb{D}^2$ in $\R X$. Let $\mathfrak{s}_{\mathbb{D}^2}=(\Spin(T \mathbb{D}^2,\mathfrak{o}_{T \mathbb{D}^2}),q_{T \mathbb{D}^2})$ be the  unique $\Spin_2$-structure on $\mathbb{D}^2$. According to \cite[Example 11.3]{chen2024spin},  $\mathfrak{s}_{\mathbb{D}^2}$ is not compatible with the short exact sequence of real vector bundles over $\mathbb{S}^1$ 
$$0\lra T \mathbb{S}^1 \lra \alpha^* T \mathbb{D}^2 \lra \mathcal{N}_\alpha \lra 0.$$  
It follows that  $s_{\mathfrak{s}_{\mathbb{D}^2}}(\alpha)=0$, and thus $s^{\mathfrak{s}_{\mathbb{D}^2}}([\alpha])=0$. The statement then follows, noting that the  $\Spin_3$-structure over $\mathbb{D}^2$ is canonically induced from the  $\Spin_2$-structure.

\eqref{pqq.a.m-2}~
With abuse of notation, in the remainder of this proof,  the notation $\omega_1(\rho D)$ refers to the dual pairing $\omega_1(\R \Sigma)\cdot (\rho D)$. By  claim~\eqref{pqq.a.m-1} and by properties of a quasi-quadratic enhancement, we have
\begin{align*}
    s^{\mathfrak{s}_{X|\Sigma}}(\rho (D+(D.S)S)&
      = \; s^{\mathfrak{s}_{X|\Sigma}}(\rho D+(D.S)\rho S)\\
     & =\; s^{\mathfrak{s}_{X|\Sigma}}(\rho D)+(D.S)s^{\mathfrak{s}_{X|\Sigma}}(\rho S)+(D.S)(\rho D. \rho S)+(D.S)\omega_1(\rho D)\omega_1(\rho S)\\
    &\equiv  \; s^{\mathfrak{s}_{X|\Sigma}}(\rho D)+0+(D.S)^2+(D.S)(D)^2(-2) \mod 2\\
    &\equiv  \;s^{\mathfrak{s}_{X|\Sigma}}(\rho D)+(D.S)^2 \mod 2\\
    &\equiv  \; s^{\mathfrak{s}_{X|\Sigma}}(\rho D)+(D.S) \mod 2.
\end{align*} 
As a result, the proposition follows.
\end{proof}

\subsubsection{Quasi-quadratic enhancements and spinor states}
As in Section~\ref{subsect-Spinorstates.revisited}, we assume that~$C$ is a real balanced rational curve parametrized by a real balanced immersion $f\colon (\C P^1,\tau_1)\to (X,\tau)$  such that $f_*[\C P^1]=d\in  H_2^{-\tau}(X;\Z)$. Furthermore, we assume that $f (\C P^1)\subset \Sigma$, where $\Sigma\in |-\frac{1}{2}K_X|$ is a real non-singular surface, and that $C$ represents the homology class $D\in H_2^{-\tau}(\Sigma;\Z)$. Since $f$ is real, the restriction $f|_{\R P^1}\colon  \R P^1 \to \R X$ is an immersion such that $f|_{\R P^1}(\R P^1)\subset \R \Sigma$.

Let $k_d$ and $k_D$ be the corresponding integers of the classes $d$ and $D$, respectively. Recall that $k_D=k_d-1$. For each $\Pin_2^-$-structure $\mathfrak{p}_{T\R \Sigma}$ on $O(T \R \Sigma)$, we define 
$$s_{\mathfrak{p}_{T\R \Sigma}}(C):=s_{\mathfrak{p}_{T\R \Sigma}}\left(f|_{\R P^1}\right) \in \{0,1\}.$$ 
The number of hyperbolic real nodes  $\delta_H(C)$ of $C$ is precisely equal to $\delta(f|_{\R P^1})$. From Section~\ref{subsect-Welschinger.surfaces}, we deduce that
$$\delta_H(C)= g(D)+\delta_E(C) \mod 2,$$
 where $\delta_E(C)$ denotes the number of elliptic real nodes of $C$ and $g(D)=\frac{1}{2}(K_\Sigma.D+D^2)+1$.
 It then follows from  Proposition~\ref{prop-quasi-quadratic.enhancement.def.prop}  that there is a quasi-quadratic enhancement
 $$s^{\mathfrak{p}_{T\R \Sigma}}\colon H_1(\R \Sigma;\Z_2)\lra \Z_2$$
 satisfying
 \begin{equation} \label{eq-RelationW2W2}
     s^{\mathfrak{p}_{T\R \Sigma}}(\rho D)= g(D)+\delta_E(C)+s_{\mathfrak{p}_{T\R \Sigma}}(C) \mod 2.
 \end{equation}

The quasi-quadratic enhancement $s^{\mathfrak{p}_{T\R \Sigma}}(\rho D)$ allows us to relate the spinor state $\sp_{\mathfrak{s}_X,\mathfrak{o}_X}(C)$ in  Definition~\ref{def-spinor3dim} to Welschinger's sign $s_{\R \Sigma}(C)$ in Definition~\ref{def-Welschinger.Inv.Nodes}, as shown in the following proposition.

\begin{proposition}\label{Prop-RelationW2W3}
Given  a $\Spin_3$-structure $\mathfrak{s}_X$  on $\SO(T\R X,\mathfrak{o}_X)$ and the restricted $\Pin_2^-$-structure  $\mathfrak{s}_{X|\Sigma}$ over $\R \Sigma$, one has
 \begin{equation} \label{eq-Prop-RelationW2W3}
     \sp_{\mathfrak{s}_X,\mathfrak{o}_X}(C)=(-1)^{\epsilon(D)+g(D)+s^{\mathfrak{s}_{X|\Sigma}}(\rho D)}s_{\R \Sigma}(C), 
 \end{equation}
 where  $\epsilon$ is the homomorphism defined in Definition~\ref{def-epsilon.function}.
 \end{proposition}

 \begin{proof}
On the left-hand side of (\ref{eq-Prop-RelationW2W3}), the spinor state $\sp_{\mathfrak{s}_X,\mathfrak{o}_X}(C)$ is determined by two parameters: The isotopy class of the real line subbundle  $L \subset \R \mathcal{N}_f$ and the $\Spin_3$-structure $\mathfrak{s}_X$. 
 
On the one hand, recall, moreover, that the holomorphic normal line bundle $\mathcal{N}'_f$  has the same degree as the bundle whose real part is $L$, and that $\mathcal{N}'_f$ is a subbundle of the holomorphic normal plane bundle $\mathcal{N}_f$.  
The map $\epsilon$ (see Definition~\ref{def-epsilon.function}) allows us to compare the isotopy class of two real line subbundles $\R \mathcal{N}_f'$ and~$L$ in $\R \mathcal{N}_f$. As a consequence, the sign difference of $(-1)^{\epsilon(D)}$ in Equation (\ref{eq-Prop-RelationW2W3}) is deduced.

On the other hand, the $\Spin_3$-structure $\mathfrak{s}_X$ induces a canonical $\Pin^-_2$-structure $\mathfrak{s}_{X|\Sigma}$ on $O(T \R \Sigma)$, and consequently its corresponding quasi-quadratic enhancement $s^{\mathfrak{s}_{X|\Sigma}}$. It follows that
\begin{align*}
    \sp_{\mathfrak{s}_X,\mathfrak{o}_X}(C)&=(-1)^{\epsilon(D)}\times (-1)^{s_{\mathfrak{s}_{X|\Sigma}}(C)}\\
    &=(-1)^{\epsilon(D)}\times (-1)^{g(D)+s^{\mathfrak{s}_{X|\Sigma}}}\times (-1)^{\delta_E(C)}.
\end{align*}
 This enhancement then contributes the difference of $(-1)^{g(D)+s^{\mathfrak{s}_{X|\Sigma}}(\rho D)}$ in Equation (\ref{eq-Prop-RelationW2W3}). As a result, Equation~(\ref{eq-Prop-RelationW2W3}) holds.
 \end{proof}

%%%%%%%%%%%%%%%%%%%%%%%%%%%%%%%%%%%%%%%%%
\section{Proofs of main results}
\label{Sect-proofs}

\subsection{Proof of Theorem~\ref{Theorem1-GW}}\label{Sect-Proof.GW}
The main idea of the proof is to reduce the computation of genus $0$ Gromov--Witten invariants for del Pezzo threefolds to the analogous problem for del Pezzo surfaces. This reduction relies primarily on Propositions~\ref{Prop-why.degree.is.678} and~\ref{Prop-Kollar}. We begin by recalling the following hypotheses and notational conventions:
\begin{itemize}
    \item $X$ is a del Pezzo variety of dimension~$3$  such that $\ker(\psi) \cong \Z$, where $\Sigma\in |-\frac{1}{2}K_X|$ is any non-singular surface and $\psi\colon H_2(\Sigma;\Z)\twoheadrightarrow H_2(X;\Z)$ is induced by the inclusion $\Sigma\hookrightarrow X$.
    \item $\mathcal{Q}$ is a pencil of surfaces  in the linear system $|-\frac{1}{2}K_X|$; its base locus is an elliptic curve denoted by~$E$.
    \item $E\subset X$ realizes the homology class $\frac{1}{4}K^2_X$.
\end{itemize}

Let $d\in H_2(X;\Z)$, and let $k_d$ be the associated integer.  Recall that $\mathcal{C}(\underline{x}_d)$ denotes the set of connected algebraic curves of arithmetic genus $0$ representing the given homology class $d$ and passing through a generic configuration $\underline{x}_d$ of $k_d$ points on $E$ (see Proposition~\ref{Prop-Kollar}).  Let ${U}_{k_d}\subset E^{k_d}$ be as in Proposition~\ref{Prop.Modulispace2}. Choose a configuration $\underline{x}_d\in {U}_{k_d}$ and a subconfiguration $\underline{y}_D\subset \underline{x}_d$ consisting of $k_D=k_d-1$ points.  By Proposition~\ref{Prop-Kollar} and by the Riemann--Roch theorem, every irreducible rational curve contained in a non-singular surface of the pencil~$\mathcal{Q}$ and passing through $\underline{y}_D$ intersects the elliptic curve $E$ at the $\supth{k_d}$ point $x_{k_d}$, which is precisely the point $\underline{x}_d\setminus \underline{y}_D$.  Let $\Sigma$ be a surface of $\mathcal{Q}$ that contains a curve in $\mathcal{C}(\underline{x}_d)$.  As discussed in Section~\ref{Sect-Monodromy}, up to monodromy, the second homology group $H_2(\Sigma_t;\Z)$ can be canonically identified with $H_2(\Sigma;\Z)$ for every non-singular surface $\Sigma_t$ in the pencil $\mathcal{Q}$.  Let $S$ be a generator of $\ker(\psi)$, which can also be identified with the vanishing cycle in the monodromy transformation $T(D)=D+(D.S)S$.  Let $\overline{D}\in H_2(\Sigma;\Z)/_ {D\sim T(D)}$.  The surjectivity of $\psi$ induces a surjective map $\overline{\psi}\colon H_2(\Sigma;\Z)/_ {D\sim T(D)}\twoheadrightarrow H_2(X;\Z)$. By Proposition~\ref{Prop-Kollar}, we obtain the following equality:
    \begin{equation*}
        \left|\mathcal{C}\left(\underline{x}_d\right)\right|=\sum_{\overline{D}\in \overline{\psi}^{-1}(d)}\;\sum_{\substack{\Sigma_t\in\mathcal{Q}:\\
                 \overline{D}|_{E,\Sigma_t}\sim [\underline{y}_D\sqcup \{x_{k_d}\}]
                 }}\left|\left\{C\in \mathcal{C}\left(\underline{x}_d\right): C \subset \Sigma_t \text{ and } \overline{[C]}=\overline{D} \right\}\right|.
    \end{equation*}
    By Proposition~\ref{Prop.Modulispace2}, we have
    $$\left|\mathcal{C}(\underline{x}_d)\right|=\GW_X(d).$$

 Moreover, the total evaluation map  $\ev_{X}$ is regular at every point $[(\C P^1;p_1,\ldots,p_{k_d});f]$ such that $f(\{p_1,\ldots,p_{k_d}\})=\underline{x}_d$. Consequently, every curve in $\mathcal{C}({\underline{x}_d})$  is parametrized by an algebraic immersion $f\colon \mathbb{C}P^1\to X$. It is worth noting that the genus $0$ Gromov--Witten invariants are symmetric under mono\-dromy transformations. That is, for every $D\in H_2(\Sigma;\Z)$, we have
  $$\GW_{\Sigma}(D)=\GW_{\Sigma}(D+(D.S)S).$$
 We set
 $$\GW_{\Sigma_t}(\overline{D}):= \GW_{\Sigma_t}(D)=\GW_{\Sigma_t}(T(D)).$$
 Therefore, the cardinality of the set  $\{C\in \mathcal{C}(\underline{x}_d): C \subset \Sigma_t \text{ and } \overline{[C]}=\overline{D} \}$  is exactly $\GW_{\Sigma_t}({D})$, and thus also $\GW_{\Sigma_t}(\overline{D})$  (see Proposition~\ref{Prop.Modulispace1} and Corollary~\ref{Corollary.Modulispace1}).  Moreover, the equivalence class $\overline{D}|_{E,\Sigma_t}=\overline{D|_{E,\Sigma_t}}$ lies in $\Pic_{D.E} (E)/_{x\sim \overline{D}_{\phi}-x}$. Accordingly, the number of surfaces $\Sigma_t\in\mathcal{Q}$ such that $\overline{D}|_{E,\Sigma_t}\sim [\underline{y}_D\sqcup \{x_{k_d}\}]$ is exactly the degree of  the map  $\phi_{\overline{D}}$ (see  Proposition~\ref{Prop-degree.case.complex}). Since we are considering cases in which a generator of $\ker (\psi)$ can be identified with the vanishing cycle (see Lemma~\ref{Lemma-Vanishing.Cycle}), the degree of the corresponding map  $\phi_{\overline{D}}$ is exactly $(D.S)^2$. Therefore, we have 
\begin{equation}\tag{$\star$}\label{star}
     \begin{aligned}
        \GW_X(d) &=\sum_{\overline{D}\in \overline{\psi}^{-1}(d)} \GW_{\Sigma_t}\left(\overline{D}\right) \times  \left|\left\{\Sigma_t\in\mathcal{Q}: \overline{D}|_{E,\Sigma_t}\sim \left[\underline{y}_D\sqcup \left\{x_{k_d}\right\}\right]\right\}\right| \\
       &=\sum_{\overline{D}\in \overline{\psi}^{-1}(d)}  (D.S)^2 \GW_{\Sigma_t}\left(\overline{D}\right).
    \end{aligned}
\end{equation}

A straightforward computation shows that $(D.S)^2=(T(D).S)^2$. Consequently, the equality \eqref{star} simplifies to
 
$$ 
    \GW_X(d)=\sum_{\substack{\{D,T(D)\}\subset H_2(\Sigma;\Z):\\D\in \psi^{-1}(d)}
      }(D.S)^2 \GW_{\Sigma}(D).$$
      
Thus, we obtain the desired formula
          $$\GW_X(d)=\frac{1}{2}\sum_{\substack{D\in H_2(\Sigma;\Z):\\
          D\in \psi^{-1}(d)}}(D.S)^2 \GW_{\Sigma}(D).$$

%%%%%%%%%%%%%%%%%%%%%%%%%%%%%%%
\subsection{Proof of Theorem~\ref{Theorem2-W}}
The main idea of the proof is to reduce the computation of genus $0$ Welschinger invariants for real del Pezzo threefolds to the analogous problem for real del Pezzo surfaces, primarily using Propositions~\ref{Prop-Kollar} and~\ref{prop-tau.vs.mu}.  In contrast to the Gromov--Witten case, an additional challenge arises from the necessity to resolve a \textit{sign problem} that comes  from the distinction between two definitions of Welschinger's sign for curves in these varieties. We begin by recalling the following hypotheses and notational conventions:
\begin{itemize}
    \item $X$ is a del Pezzo variety of dimension~$3$ as described in the proof of Theorem~\ref{Theorem1-GW}, which we equip with a real structure $\tau$.
    \item  $\mathcal{Q}$ is a real pencil of surfaces in the linear system $|-\frac{1}{2}K_X|$; its real part is $\R \mathcal{Q}\cong \R P^1$, and its base locus is a real elliptic curve $E$.
    \item  $E\subset X$, whose real part is $\R E\neq \emptyset$,  realizes the homology class $\frac{1}{4}K^2_X$.
\end{itemize}
Let $d\in H_2^{-\tau}(X;\Z)$, and let $k_d$ be the associated integer.
Let $\mathcal{R}(\underline{x}_d)$ be the set of real 
irreducible rational curves 
representing the homology class $d$ and passing through a real generic configuration $\underline{x}_d$ of $k_d$ points on~$E$, containing $l$ pairs of complex conjugate points. Suppose that $\underline{x}_d$ has at least one real point. Let $U_{k_d}\subset E^{k_d}$ be as in Proposition~\ref{Prop.Modulispace2}. Choose a configuration $\underline{x}_d\in U_{k_d}$ and a subconfiguration $\underline{y}_D\subset \underline{x}_d$ of $k_D=k_d-1$ points. By Proposition~\ref{Prop.Modulispace2} and Lemma~\ref{lem-surfaces.in.real.pencil}, there exists a real non-singular del Pezzo surface $(\Sigma,\tau|_{\Sigma}) \in \R \mathcal{Q}$  containing a real curve in $\mathcal{R}(\underline{x}_d)$. Let $S$ be a generator of $\ker(\psi)$. By Proposition~\ref{prop-tau.vs.mu}, in such a real surface, the vanishing cycle realizes a $\tau|_{\Sigma}$-anti-invariant class.
An analogous observation concerning monodromy transformations, as discussed in the proof of Theorem~\ref{Theorem1-GW}, also applies here. Let $\overline{D}\in H_2^{-\tau}(\Sigma;\Z)/_{D \sim T(D)}$, and consider the surjective map $\overline{\psi}\colon  H_2^{-\tau}(\Sigma;\Z)/_{D \sim T(D)}\twoheadrightarrow H_2^{-\tau}(X;\Z)$. It is also worth noting that genus $0$ Welschinger invariants are invariant under monodromy transformations.  That is, for every $D\in H_2^{-\tau}(\Sigma;\Z)$, we have
  $$W_{\R \Sigma}(D,l)=W_{\R \Sigma}(D+(D.S)S,l).$$
We set 
$$W_{\R \Sigma}\left(\overline{D},l\right):=W_{\R \Sigma}(D,l)=W_{\R \Sigma}(T(D),l).$$ 
According to Proposition~\ref{Prop-RelationW2W3}, for every $D\in H_2^{-\tau}(\Sigma;\Z)$ such that $\psi(D)=d$, the contribution to $W_{\R X}^{\mathfrak{s}_X,\mathfrak{o}_X}(d,l)$ of elements of $\mathcal{R}(\underline{x}_d)$, which are contained in $\Sigma$, is exactly equal to

$$(-1)^{\epsilon(D)+g(D)+s^{\mathfrak{s}_{X|\Sigma}}(\rho D)}W_{\R \Sigma}\left(\overline{D},l\right).$$

Suppose that $D.S$ is odd. According to Corollary~\ref{coro-degree.case.real}, there are exactly $|D.S|$ such real surfaces $(\Sigma,\tau|_{\Sigma})$ in $\R \mathcal{Q}$.  Therefore, we have
$$W_{\R X}^{\mathfrak{s}_X,\mathfrak{o}_X}(d,l)=\sum\limits_{\overline{D}\in \overline{\psi}^{-1}(d)}(-1)^{\epsilon(D)+g(D)+s^{\mathfrak{s}_{X|\Sigma}}(\rho D)} |D.S| W_{\R \Sigma}\left(\overline{D},l\right).$$

Moreover, the following equalities hold: 
\begin{itemize}
         \item $\epsilon(T(D))=\epsilon(D)+1$ 
 (see Proposition~\ref{prop-function.epsilon1});
 \item $s^{\mathfrak{s}_{X|\Sigma}}(\rho {T(D)}) \equiv s^{\mathfrak{s}_{X|\Sigma}}(\rho D) +1 \mod 2$, since $s^{\mathfrak{s}_{X|\Sigma}}(\rho {T(D)}) \equiv s^{\mathfrak{s}_{X|\Sigma}}(\rho D) +(D.S) \mod 2$ and $D.S$ is odd  (see Proposition~\ref{prop-quasi-quad.and.monodromy}.2);
     \item $g(T(D))=g(D)$ by a straightforward calculation;
     \item $T(D).S=-D.S$ by another straightforward calculation.
\end{itemize}

Thus, the following sum can be simplified as follows: 
\begin{align*} 
      &(-1)^{\epsilon(D)+g(D)+s^{\mathfrak{s}_{X|\Sigma}}(\rho D)}|D.S|W_{\R \Sigma}(D,l) 
      +(-1)^{\epsilon(T(D))+g(T(D))+s^{\mathfrak{s}_{X|\Sigma}}(\rho T(D))} |T(D).S|W_{\R \Sigma}(T(D),l)\\
        &= \left((-1)^{\epsilon(D)+g(D)+s^{\mathfrak{s}_{X|\Sigma}}(\rho D)}    +(-1)^{\epsilon(D)+1+g(D)+s_{\mathfrak{s}_{X|\Sigma}}(\rho D)+1} \right) |D.S|W_{\R \Sigma}(D,l)\\
    &=2 \times (-1)^{\epsilon(D)+g(D)+s^{\mathfrak{s}_{X|\Sigma}}(\rho D)}|D.S|W_{\R \Sigma}(D,l).
\end{align*}

As a consequence, the equality

$$ 
  W_{\R X}^{\mathfrak{s}_X,\mathfrak{o}_X}(d,l)=\sum_{\substack{\{D,T(D)\}\subset H_2^{-\tau}(\Sigma;\Z):\\D\in \psi^{-1}(d)}
      }(-1)^{\epsilon(D)+g(D)+s^{\mathfrak{s}_{X|\Sigma}}(\rho D)} |D.S| W_{\R \Sigma}(D,l)$$
     
     can be written as 
$$W_{\R X}^{\mathfrak{s}_X,\mathfrak{o}_X}(d,l)=\frac{1}{2}\sum_{\substack{D\in H_2^{-\tau}(\Sigma;\Z):\\
          D\in \psi^{-1}(d)}}(-1)^{\epsilon(D)+g(D)+s^{\mathfrak{s}_{X|\Sigma}}(\rho D)} |D.S| W_{\R \Sigma}(D,l).$$

If $D.S$ is even, by symmetry, it follows directly that $W_{\R X}^{\mathfrak{s}_X,\mathfrak{o}_X}(d,l)=0$.

\subsection{Proof of Theorem~\ref{Theorem3-W}}
Firstly, as a consequence of the proof of Theorem~\ref{Theorem2-W}, we obtain directly the following corollary.

\begin{corollary} \label{coro-Theorem2-W}
If $d\in H_2^{-\tau}(X;\Z)$ is such that $D.S$ is even, where $D\in \psi^{-1}(d)$, then  the genus $0$ Welschinger invariants  $W_{\R X}^{\mathfrak{s}_X,\mathfrak{o}_X}(d,l)$ vanish for every $\Spin_3$-structure $\mathfrak{s}_X$ on $\SO(T \R X,\mathfrak{o}_X)$. 
\end{corollary}

Let $\phi(\overline{L},\overline{D})\in \Pic_{D.E-|D.S|L.E}(E)$ denote the divisor as in Lemma~\ref{lem-constant.divisor.L.D}. By Proposition~\ref{prop-corresponding.surface.picE} and Corollary~\ref{coro-degree.case.real}\eqref{c-d.c.r-c}, if $\R E$ is disconnected and, moreover, the divisor corresponding to the real configuration of points $\underline{x}_d$ and the divisor $\phi(\overline{L},\overline{D})$ do not belong to the same component of $\R E$, then there does not exist any real irreducible rational curve of class $d$ passing through $\underline{x}_d$.

A choice of such a configuration $\underline{x}_d$ is described as follows. 
Suppose that the divisor $\phi(\overline{L},\overline{D})\in \mathbb{S}^1_0$ lies in the pointed component of $\R E$. In this case, we choose the configuration $\underline{x}_d$ such that it contains an odd number of real points on the non-pointed component of $\R E$. 
Conversely, if the divisor $\phi(\overline{L},\overline{D})\in \mathbb{S}^1_1$ lies in the non-pointed component of $\R E$, we choose the configuration $\underline{x}_d$ such that it contains at least one real point and that the number of its real points on the non-pointed component of $\mathbb{R} E$ is even (possibly zero). Moreover, in both cases, such real configurations of points are contained in a semi-algebraic open set in $\R (\Sym^{k_d}(X))$, the space of real points of the $\supth{k_d}$ symmetric power of $X$. This assertion follows by a natural extension of the argument in the case $X=\C P^3$ (see \cite[Theorem 10]{kollar2015examples}). Hence, Theorem~\ref{Theorem3-W} follows.

\section{Applications}
\label{applications}
\subsection{Applications of Theorem~\ref{Theorem1-GW}} \label{sect-applications.GW}

\subsubsection{3-dimensional projective space}
A non-singular element $\Sigma\in |-\frac{1}{2}K_{\cpp}|$ is a non-singular quadric surface in $\cpp$, and $\Sigma$ is isomorphic to $\pp$.     We first identify $H_2(\Sigma;\mathbb{Z})$ with  $\Z^2$ by considering the standard basis $(L_1,L_2)$ of $H_2(\Sigma;\mathbb{Z})$ given by
\begin{itemize}
    \item $L_1=[\C P^1\times\{p\}]$,
    \item $L_2=[\{q\}\times \C P^1]$,
\end{itemize}
 where $p,q$ are points in $\C P^1$. As a consequence, a class $a L_1+b L_2\in H_2(\Sigma;\mathbb{Z})$ can be identified with the pair $(a,b)\in \Z^2$. We then identify $H_2(\C P^3;\mathbb{Z})$ with  $\Z$ by considering the basis  $\{H\}$ of $H_2(\C P^3;\mathbb{Z})$, where $H$ is the class of a line in $\C P^3$.

\begin{theorem} \label{Theorem-GW.in.deg.8}
For every positive integer $d$, one has
\begin{align*}
        \GW_{\cpp}(d) &=\frac{1}{2}\sum_{a\in \mathbb{Z}}(d-2a)^2 \GW_{\pp}(a,d-a).
        \end{align*}
\end{theorem}

This formula is in agreement with \cite[Theorem 2]{brugalle2016pencils}.

\begin{proof}
Taking into account the above identifications and the natural inclusion of a quadric surface into projective threespace, the associated surjective map $\psi$ is defined as follows:
    \begin{align*}
        \psi\colon H_2(\Sigma;\mathbb{Z})&\longrightarrow H_2(\C P^3;\mathbb{Z})\\
        (a,b)&\longmapsto a+b.
    \end{align*}
    
This implies that $\ker(\psi)=\Z(1,-1)$ and that $(D.S)^2=(d-2a)^2$ for every $D=(a,d-a)\in \psi^{-1}(d)$. Hence, the statement follows from Theorem~\ref{Theorem1-GW}.
\end{proof}

\subsubsection{3-dimensional projective space blown up at a point} 
   Let $X=\cpp\sharp \overline{\cpp}$. Let us denote the blow-up point by $x$. As has been observed, a non-singular surface $\Sigma \in |-\frac{1}{2}K_X|$ is a non-singular del Pezzo surface of degree $7$ and, moreover, $\Sigma$ is isomorphic to the blow-up of a quadric surface at $x$. This can be written as $\Sigma\cong (\pp \sharp \overline{\cp})$. We first identify $H_2(\Sigma;\mathbb{Z})$ with  $\Z^3$ by considering a basis $(L_1,L_2;E)$ of $H_2(\Sigma;\mathbb{Z})$ given by
\begin{itemize}
    \item $L_1=[\C P^1\times\{p\}]$,
    \item $L_2=[\{q\}\times \C P^1]$,
    \item $E$, the exceptional divisor class, 
\end{itemize}
 where $p$, $q$ are points in $\C P^1$.  As a consequence, a class $a L_1+b L_2-kE\in H_2(\Sigma;\mathbb{Z})$ can be identified with the triple $(a,b;k)\in \Z^3$.     We then identify $H_2(X;\mathbb{Z})$ with  $\Z^2$ by considering a basis $(H;E)$ of $H_2(X;\mathbb{Z})$ given by\looseness=-1
\begin{itemize}
    \item $H$, the class of a line in $\cpp$,
    \item $E$, the class of a line in the exceptional divisor.
\end{itemize}

 As a consequence, a class $d H -kE \in H_2(X;\mathbb{Z})$ can be identified with the point $(d;k)\in \Z^2$.

 \begin{theorem} \label{Theorem-GW.in.deg.7}
    For every positive integer $d$ and every non-negative integer $k$ such that $k\leq d$, one has

$$\GW_{\cpp\sharp \overline{\cpp}}(d;k)
        = \frac{1}{2}\sum_{a\in \mathbb{Z}}(d-2a)^2 \GW_{\pp\sharp \overline{\cp}}(a,d-a;k).$$
    \end{theorem}

 This formula is in agreement with \cite[Theorem 1.2]{ding2020remark}.

 \begin{proof}    
Taking into account the above identifications and the natural inclusion of a blown-up quadric surface into a blown-up projective space, the associated surjective map $\psi$ is given by
    \begin{align*}
        \psi\colon H_2(\Sigma;\mathbb{Z})&\longrightarrow H_2(X;\mathbb{Z})\\
        (a,b;k)&\longmapsto (a+b;k).
    \end{align*}
   
    This implies that     $\ker(\psi)=\Z(1,-1;0)$ and that $(D.S)^2=(d-2a)^2$ for every divisor  $D=(a,d-a;k) $ in $\psi^{-1}(d;k)$.  Hence, the statement follows from Theorem~\ref{Theorem1-GW}.
 \end{proof}

 \subsubsection{Threefold product of the projective line} \label{Subsect-exam.deg.6.GW}
 Let $X=\C P^1\times \C P^1 \times \C P^1$.  As previously observed, a non-singular surface $\Sigma\in |-\frac{1}{2}K_{\ppp}|$---that is, a non-singular surface of tridegree $(1,1,1)$ in the space $\ppp$---is a non-singular del Pezzo surface of degree $6$. 
It can be shown that a non-singular $(1,1,1)$-surface $\Sigma$  in $\ppp$ is isomorphic to the blow-up of a quadric surface at two distinct points, denoted by $x_1$ and $x_2$. This can be written as $\Sigma\cong(\pp \sharp 2 \overline\cp)$.
We first identify $H_2(\Sigma;\mathbb{Z})$ with  $\Z^4$ by considering a basis $(L_1,L_2;E_1,E_2)$ of $H_2(\Sigma;\mathbb{Z})$ given by
\begin{itemize}
    \item $L_1=[\C P^1\times\{p\}]$,
    \item $L_2=[\{q\}\times \C P^1]$,
    \item $E_1$, the exceptional divisor class over $x_1$,
    \item $E_2$, the exceptional divisor class over $x_2$, 
\end{itemize}
 where $p$, $q$ are points in $\C P^1$. 
  As a consequence, a class $a L_1+b L_2-\alpha E_1 - \beta E_2\in H_2(\Sigma;\mathbb{Z})$ can be identified with the quadruple $(a,b;\alpha,\beta)\in \Z^4$.  We then identify $H_2(X;\mathbb{Z})$ with  $\Z^3$ by considering a basis $(M_1,M_2,M_3)$ of $H_2(X;\mathbb{Z})$ given by
\begin{itemize}
    \item $M_1=[\C P^1\times\{p_1\}\times\{p_2\}]$,
    \item $M_2=[\{q_1\}\times \C P^1\times \{q_2\}]$,
    \item $M_3=[\{r_1\}\times \{r_2\}\times \C P^1]$,
\end{itemize}
 where $p_i,q_i,r_i\in \C P^1$ ($i \in \{1,2\}$).  As a consequence,  a class $a_1 M_1 +a_2 M_2+a_3 M_3\in H_2(X;\mathbb{Z})$ can be identified with the point $(a_1,a_2,a_3)\in \Z^3$.

 \begin{theorem}\label{Theorem-GW.in.deg.6}
 For every triple of non-negative integers $(a, b, c)$, one has
 \begin{equation}
      \begin{aligned}
       \GW_{\ppp}(a,b,c)
    =\frac{1}{2}\sum_{\alpha \in \mathbb{Z}}(a+b-c-2\alpha)^2 \GW_{\pp \sharp 2 \overline\cp}(a,b;\alpha, a+b-c-\alpha).
 \end{aligned}
 \label{Eq:GW.in.deg.6}
 \end{equation}
 \end{theorem}

 \begin{proof}
 Taking into account the above identifications, the surjective map $\psi$ induced from the inclusion of $\Sigma$ in $X$ is given by
    \begin{align*}
        \psi\colon H_2(\Sigma;\mathbb{Z})&\longrightarrow H_2(X;\mathbb{Z})\\
       (a,b;\alpha,\beta)&\longmapsto (a,b,a+b-\alpha-\beta).
    \end{align*}

    This implies that    $\ker(\psi)=\Z(0,0;1,-1)$ and  that $(D.S)^2=(a+b-c-2 \alpha)^2$ for every divisor class $D=(a,b;\alpha,a+b-c-\alpha)$ in $\psi^{-1}(a,b,c)$.    Hence, the statement follows from Theorem~\ref{Theorem1-GW}.
\end{proof}

\begin{proposition} \label{Prop-vanishing.GW}
The Gromov--Witten invariant $ \GW_{\ppp}(a,b,c)$ vanishes if $a+b+c>1$ and at least one of the following conditions is satisfied:
    \begin{itemize}
        \item $a\geq b+c$,
        \item $b\geq c+a$,
        \item $c\geq a+b$.
    \end{itemize}
\end{proposition}

\begin{proof}
We first compare genus $0$ Gromov--Witten invariants of  $\pp \sharp 2 \overline\cp$ and  of $\C P^2 \sharp 3 \overline\cp$.  We identify  $H_2(\C P^2 \sharp 3 \overline\cp;\Z)$ with $\Z^4$ by considering a basis $(L;F_1,F_2,F_3)$ of  $H_2(\C P^2 \sharp 3 \overline\cp;\Z)$  given by
\begin{itemize}
    \item $L$, the class of a line in $\C P^2$,
    \item  $F_1$, the exceptional divisor class over $x_1$,
     \item  $F_2$, the exceptional divisor class over $x_2$,
     \item $F_3$, the exceptional divisor class over $x_3$,
\end{itemize}
 where $x_1$, $x_2$, $x_3$ are points in $\C P^2$. We consider the map $ \kappa_*\colon  H_2(\pp \sharp 2 \overline\cp;\Z) \to H_2(\C P^2 \sharp 3 \overline\cp;\Z) $ defined by
 \begin{align*}
      L_1& \longmapsto L-F_1,\\
    L_2 &\longmapsto L-F_2,\\
       E_1 &\longmapsto L-F_1-F_2,\\
     E_2 &\longmapsto F_3.
 \end{align*}
 It follows that $\kappa_*$ sends a class $(a, b; \alpha, \beta)$ to the class $(a + b - \alpha; a - \alpha, b - \alpha, \beta)$. We thus obtain the following equality:
\begin{equation}\label{Eq:GW.P1P1.and.P2}
    \GW_{\pp \sharp 2 \overline\cp}(a,b;\alpha, \beta)=\GW_{\C P^2 \sharp 3 \overline\cp}(a+b-\alpha;a-\alpha, b-\alpha,\beta).
\end{equation}

By Theorem~\ref{Theorem1-GW}, if either $|D.S|=0$ or $\GW_\Sigma(D)=0$ for all $D\in \psi^{-1}(d)$, then $\GW_X(d)$ vanishes. In particular,
\begin{enumerate}
    \item if $|D.S|=0$, which means that $|(a,b;\alpha, a+b-c-\alpha)(0,0;1,-1)|=0$ or, equivalently, 
    $2\alpha={a+b-c}$, then we obtain $D= (a,b;\alpha,\alpha) \in H_2(\pp \sharp 2 \overline\cp;\Z)$; 
    \item if $\GW_\Sigma(D)=\GW_\Sigma(T(D))=0$, then Equation (\ref{Eq:GW.P1P1.and.P2}) implies that
 \begin{equation*}
     \begin{cases}
    \GW_{\pp \sharp 2 \overline\cp}(a,b;\alpha, a+b-c-\alpha)=0, \\
    \GW_{\pp \sharp 2 \overline\cp}(a,b; a+b-c-\alpha,\alpha)=0.
    \end{cases}
 \end{equation*}
Consequently, we obtain
       \begin{equation} \label{Eq:GW.P2.blowup}
        \begin{cases}
    \GW_{\C P^2 \sharp 3 \overline\cp}(a+b-\alpha;a-\alpha, b-\alpha,a+b-c-\alpha)=0, \\
    \GW_{\C P^2 \sharp 3 \overline\cp}(c+\alpha;c-b+\alpha, c-a+\alpha,\alpha)=0.
    \end{cases}
    \end{equation}
 Moreover, it is known that $\GW_{\C P^2 \sharp 3 \overline\cp}(d;a_1,a_2,a_3)=0$ if there exists a pair $a_i,a_j$ such that $a_i+a_j>d$, except in the cases where $(d;a_1,a_2,a_3)\in \{(1;1,1,0),(1;1,0,1),(1;0,1,1)\}$ (see  \cite{gottsche1996quantum}). %Note also the identity
    %$$\GW_{\C P^2 \sharp 3 \overline\cp}((d;a_1,a_2,a_3))=\GW_{\C P^2 \sharp 3 \overline\cp}((2d-a_1-a_2-a_3;d-a_2-a_3,d-a_1-a_3,d-a_1-a_2)).$$
Therefore, Equations (\ref{Eq:GW.P2.blowup}) holds if $\alpha$ satisfies
        \begin{equation*}
     \begin{cases}
   \alpha < \max \{0,b-c,a-c\}\text{ and}  \\
    \alpha > \min \{a,b,a+b-c\}.
    \end{cases}
        \end{equation*}
\end{enumerate}
Combining this with the  symmetry
$$ \GW_{\ppp}(a,b,c)=\GW_{\ppp}(\sigma(a),\sigma(b),\sigma(c)),$$
where $\sigma$ is any permutation of the set $\{a,b,c\}$, we obtain the proposition.
\end{proof}

   It is known, due to the work of  L. G{\"o}ttsche and R. Pandharipande  \cite{gottsche1996quantum}, that the genus $0$ Gromov--Witten invariants of $\C P^2 \sharp 3 \overline\cp$ can be explicitly computed 
by recursive formulas.
        In Table~\ref{Tab:GW.in.deg.6}, we show the first non-vanishing genus $0$ Gromov--Witten invariants $\GW_{\ppp}(a,b,c)$. In this table, we consider the cases where either $a+b+c=1$ or  $a<b+c$, 
    for  $4\geq a\geq b\geq c$ and $a+b+c$ increasing up to $12$.
\begin{table}[ht!]
\centering 
\begin{tabular}{ |c|c||c|c|c|} 
 \hline
$d$ & $\GW_{\ppp}(d)$ & $D$ & $|D.S|$ & $ \GW_{\pp \sharp 2 \overline\cp}(D)$\\
$(a,b,c)$ &  & $(a,b;\alpha, a+b-c-\alpha)$ &  & \\
 \hline
  \hline
     $(1,0,0)^*$ & 1 & (1,0;0,1)&1&1 \\ 
 \hline
    (1,1,1) & 1 & (1,1;0,1)&1&1 \\ 

 \hline
     (2,2,1) & 1 & (2,2;0,3)&3&0 \\ 
      &  & (2,2;1,2)&1&1 \\ 

 \hline
     (2,2,2) & 4 & (2,2;0,2)&2&1 \\ 
 
 \hline
     (3,2,2) & 12 & (3,2;0,3)&3&0 \\ 
      &  & (3,2;1,2)&1&12 \\ 
 \hline
     (3,3,1) & 1 & (3,3;0,5)&5& 0\\ 
       &  & (3,3;1,4)&3& 0\\ 
       &  & (3,3;2,3)&1& 1\\ 
 \hline
     (3,3,2) & 48 & (3,3;0,4)& 4&0\\ 
       &  & (3,3;1,3)&2& 12\\ 
 \hline
     (3,3,3) & 728 & (3,3;0,3)&3& 12\\ 
       &  & (3,3;1,2)&1& 620\\ 

  \hline
     (4,3,2) & 96 & (4,3;0,5)&5& 0\\ 
       &  & (4,3;1,4)&3& 0\\ 
        &  & (4,3;2,3)&1& 96\\ 
 \hline
      (4,3,3) & 2480 & (4,3;0,4)&4& 0\\ 
       &  & (4,3;1,3)&2& 620\\
 \hline
      (4,4,1) & 1 & (4,4;0,7)&7& 0\\ 
       &  & (4,4;1,6)&5& 0\\ 
        &  & (4,4;2,5)&3& 0\\ 
        &  & (4,4;3,4)&1& 1\\ 
 \hline
    (4,4,2) & 384 & (4,4;0,6)&6& 0\\ 
       &  & (4,4;1,5)&4& 0\\ 
        &  & (4,4;2,4)&2& 96\\ 
 \hline
     (4,4,3) & 23712 & (4,4;0,5)&5& 0\\ 
       &  & (4,4;1,4)&3& 620\\ 
        &  & (4,4;2,3)&1& 18132\\ 
 \hline
     (4,4,4) & 359136 & (4,4;0,4)&4& 620\\ 
       &  & (4,4;1,3)&2& 87304\\
 \hline
\end{tabular}
\caption{Genus $0$ Gromov--Witten invariants of $\C P^1 \times \C P^1 \times \C P^1$ using  Formula (\ref{Eq:GW.in.deg.6}) for $a+b+c \leq 12$.} 
\label{Tab:GW.in.deg.6}
\end{table}

\begin{remark}
  The values of $ \GW_{\ppp}(a,b,c)$ in this computation are compatible with those given by X.~Chen and A.~Zinger using Solomon's WDVV-type relations in work in progress.
\end{remark}

\subsection{Applications of Theorems~\ref{Theorem2-W} and~\ref{Theorem3-W}} \label{sect-applications.W}
In this subsection, we use the notation
$$\tau_n \colon  \C P^n \lra \C P^n , \quad \tau_n ([z_0:z_1:\dots :z_n])=[\overline{z_0}:\overline{z_1}:\dots :\overline{z_n}]$$
for the standard conjugation on the projective space $\C P^n$ ($n\geq 1$).

\subsubsection{Real 3-dimensional projective space $\boldsymbol{X=(\C P^3,\tau_3}$)}

 We consider the non-singular real quadric surface $\Sigma$ in $X$ with non-empty real part and the real Lagrangian sphere realizing a $\tau_3|_{\Sigma}$-anti-invariant class. In this case,  $\Sigma$  is isomorphic to $(\pp, \tau_3|_{\Sigma})$, where $$\tau_3|_{\Sigma}(p,q)= (\tau_1(p),\tau_1(q))$$ for all $(p,q)\in \pp$. Hence, the fixed loci of $\tau_3$ and $\tau_3|_{\Sigma}$ are the real projective space $\R P^3$ and the $2$-torus $\R P^1\times \R P^1$, respectively. The group $H_2^{-\tau_3}(\C P^3;\Z)$ is freely generated by $L_{\C P^3}$, the homology class of a line embedded in $\C P^3$, while the group $H_1(\R P^3;\Z_2)$ is generated by $\rho L_{\C P^3}$. We choose a basis for the group $H_2^{-\tau_3 }(\pp;\Z)$ as
$$(L_1,L_2)=\left(\left[\C P^1\times \{p\}\right],\left[\{q\}\times \C P^1\right]\right),\quad \text{where } p,q \in \C P^1,$$
and a basis for the group  $H_1(\R P^1 \times \R P^1;\Z_2)$ as
$$(\rho L_1, \rho L_2)=\left(\left[\R P^1\times \{p_\R\}\right],\left[\{ q_\R \}\times \R P^1\right]\right), \quad \text{where } p_\R , q_\R \in \R P^1.$$
The $\Spin_3$-structure $\mathfrak{s}_{\R P^3}$ over $\R P^3$ is chosen such that the spinor state of a real line in $\C P^3$ is $+1$. In other words, the structure $\mathfrak{s}_{\R P^3}$ is chosen such that the spinor state satisfies $\sp_{\mathfrak{s}_{\R P^3},\mathfrak{o}_{\R P^3}}(f(\C P^1))=+1$, where the map $f\colon  (\C P^1,\tau_1) \to (\C P^3,\tau_3)$ is a real immersion and $f_*[\C P^1]=L_{\C P^3}$.

\begin{lemma} \label{lemma1}
Let $L= (1,0) \in H_2^{-\tau_3}(\pp;\Z)$. Assume that $\R f^*\mathcal{N}_{L/\pp}=E^+$. 
Then, the function $\epsilon\colon  H_2(\pp;\Z)\to \Z_2$ $($see Definition~\ref{def-epsilon.function}\,$)$ satisfies $\epsilon(a,b)=0$ if and only if $a>b$.
\end{lemma}

\begin{proof}
Recall that $S\in \Z(1,-1)=\ker (\psi) $. By a straightforward computation, we have $(D.S)(L.S)=a-b$. By definition, $\epsilon(L)=k_L +1 \mod 2$. Thus, $\epsilon(L)=0 \mod 2$. The lemma then follows from Proposition~\ref{prop-function.epsilon2}.
\end{proof}

We obtain the following theorem.

\begin{theorem} \label{theo-application.degree.8} \label{Theorem-W.in.deg.8}
For every odd positive integer $d$ and every integer $l$ such that $0 \leq l \leq d-1$, one has
\begin{equation} \label{eq-Welschinger.formula.dim.8}
    W_{\R P^3}^{\mathfrak{s}_{\R P^3},\mathfrak{o}_{\R P^3}}(d,l)=\frac{1}{2}\sum_{a\in \Z}(-1)^{a} (d-2a) W_{\R P^1\times \R P^1} ((a,d-a),l).
\end{equation}
\end{theorem} 

Formula (\ref{eq-Welschinger.formula.dim.8}) can be reduced to \cite[Theorem 1]{brugalle2016pencils}.

\begin{proof}
  For every $D=(a,b)\in H_2^{-\tau_3}(\pp;\Z_2)$, it is obvious that
$$k_D=(2,2).(a,b)-1\equiv 1 \mod 2$$
and 
$$g(D)=\tfrac{1}{2}((2,2).(a,b)+(a,b)^2+2)\equiv (a+1)(b+1) \mod 2.$$
We consider a real balanced immersion $f\colon (\C P^1,\tau_1) \to (\C P^3,\tau_3)$ such that 
$$\Im(f)\subset \pp \quad \text{and}\quad f_*[\C P^1]\in H_2^{-\tau_3}(\pp;\Z).$$

It should be noted that the homology class $L_{\C P^3} \in H_2^{-\tau_3}(\C P^3;\Z)$ has two effective classes\footnote{These are homology classes that can be realized by algebraic curves.} in its preimages under the surjective homomorphism $\psi\colon H_2(\pp;\Z)\to H_2(\C P^3;\Z)$. These correspond to the classes $(1,0)$ and $(0,1)$ in $H_2^{-\tau_3}(\pp;\Z)$. Moreover, the restricted $\Pin^-_2$-structure  on $O(T (\R P^1\times \R P^1))$  satisfies  $s^{\mathfrak{s}_{\R P^3|\R P^1\times \R P^1}}(1,0)=0$ and  $s^{\mathfrak{s}_{\R P^3|\R P^1\times \R P^1}}(0,1)=1$. Hence,
$$s^{\mathfrak{s}_{\R P^3|\R P^1\times \R P^1}}(a,b) \equiv b+ab \mod 2.$$
Consequently, the two possibilities for class $D$ can be established as follows: 
\begin{itemize}
    \item If $D=(a,b)$ with $a>b$, then
$$(-1)^{\epsilon(D)+g(D)+s^{\mathfrak{s}_{\R P^3|\R P^1\times \R P^1}}(\rho D)}=(-1)^{0+(a+1)(b+1)+b+ab}.$$
Upon simplification, this sign is equal to $(-1)^{a+1}$.
    \item If $D=(a,b)$ with $a<b$, then
$$(-1)^{\epsilon(D)+g(D)+s^{\mathfrak{s}_{\R P^3|\R P^1\times \R P^1}}(\rho D)}=(-1)^{1+(a+1)(b+1)+b+ab}.$$
Upon simplification, this sign is equal to $(-1)^{a}$.
\end{itemize}
As a result, if $a+b \equiv 1 \mod 2$, then we have
\begin{align*}
     W_{\R P^3}^{\mathfrak{s}_{\R P^3},\mathfrak{o}_{\R P^3}}(d,l)
    = & \frac{1}{2}\left(\sum_{0\leq a<\frac{d}{2}}(-1)^{a} (d-2a) W_{\R P^1\times \R P^1} ((a,d-a),l)\right.\\
    &\hphantom{\frac{1}{2}\Big(}\left. + \sum_{d\geq a>\frac{d}{2}}(-1)^{a+1} (-(d-2a)) W_{\R P^1\times \R P^1} ((a,d-a),l) \right)\\
    = & \sum_{0\leq a<\frac{d}{2}}(-1)^{a} (d-2a) W_{\R P^1\times \R P^1} ((a,d-a),l).
\end{align*}

In addition, the genus $0$ Welschinger invariants $ W_{\R P^1\times \R P^1} ((a,d-a),l)$ vanish when $a \notin [0,d]$; Formula~(\ref{eq-Welschinger.formula.dim.8}) then holds.
\end{proof}

\begin{remark}
Directly from Corollary~\ref{coro-Theorem2-W} and the fact that $$|D.S|=|(a,d-a).(1,-1)|=| d-2a|\equiv d\mod 2,$$
it follows that the invariants $ W_{\R P^3}^{\mathfrak{s}_{\R P^3},\mathfrak{o}_{\R P^3}}(d,l)$ vanish if $d$ is even, for every $\Spin_3$-structure $\mathfrak{s}_{\R P^3}$ over $\R P^3$. This result was obtained by G.~Mikhalkin, see \cite{brugalle2007enumeration}, for symmetry reasons. 
\end{remark}

\begin{proposition} \label{prop-application.deg8}
Assume that $\R E$ is disconnected and the hyperplane class is on its pointed component. If $d$ is even and, moreover, the configuration $\underline{x}_d$ has an odd number of real points on each component of\, $\R E$, then there does not exist any real irreducible rational degree $d$ curve in  $\C P^3$ passing through $\underline{x}_d$. 
\end{proposition}

\begin{proof}
This proposition is a consequence of Theorem~\ref{Theorem3-W}     and the following remarks: 
    \begin{enumerate}
      \item The base locus of the pencil of surfaces in the linear system $|-\frac{1}{2}K_{\C P^3}|$ is an elliptic curve $E$ of degree~$4$ in $\C P^3$. Moreover, the curve $E$ is of homology class 
        $$(2,2)\in  H_2(\C P^1\times \C P^1;\Z).$$
        \item By the intersection form on $\C P^3$ and by the surjectivity of $\psi$, the hyperplane section is of class $$H=(1,1) \in H_2(\C P^1\times \C P^1;\Z) \cong \Pic(\C P^1\times \C P^1),$$ Moreover, the class $H$ can be written as $H=L+T(L)=L_1+L_2$.
        \item  Let $h:=H|_E \in \Pic_4(E)$ be the hyperplane class on $E$. The constant divisor $ \phi(\overline{L},\overline{D})$, which has degree $D.E-|D.S|(L.E)=4(d-a)$,  can be described as
        \begin{align*}
            \phi(\overline{L},\overline{D})&=(a,d-a)|_E-(d-2a) (1,0)|_E\\
            &=(d-a)h.
        \end{align*}
         Hence, the hyperplane class on $E$ and the divisor         $\phi(\overline{L},\overline{D})$ lie in the same connected component of $\R E$, that is, in the pointed component $\mathbb{S}^1_0$ of  $\R E$ by hypothesis.         If the configuration $\underline{x}_d$ has an odd number of real points on each component of $\R E$, then its divisor class $[\underline{x}_d]$ lies in the non-pointed component $\mathbb{S}^1_1$ of $\R E$. By  Corollary~\ref{coro-degree.case.real}, the statement then follows.\qedhere  
               \end{enumerate} 
\end{proof}

This vanishing result was obtained by J. Koll\'ar in \cite{kollar2015examples}.
\subsubsection{Real 3-dimensional projective space blown up at a real point}

Let $(\cpp\sharp\overline{\cpp},\tau)$ denote  the real $3$-dimensional projective space blown up at a real point, denoted by $x$. Here, the real structure $\tau$ is chosen such that its fixed locus is the blow-up at a real point of the real projective space, that is, $\R P^3\sharp{\overline{\R P^3}}$. We consider the immersed non-singular real del Pezzo surfaces $(\Sigma, \tau|_\Sigma)$ of degree $7$ in $\cpp\sharp\overline{\cpp}$,  which are isomorphic to the blow-up at a real point of $\pp$. Moreover, the fixed locus of $\tau|_\Sigma$ is the blow-up at a real point of a $2$-torus $\R P^1\times \R P^1 \sharp {\R P^2}$. We first choose a basis for the group $H_2^{-\tau }(\C P^3\sharp\overline{\C P^3};\Z)$ as
\begin{itemize}
    \item  $L_{\C P^3}, \text{ the homology class of a line embedded in } \C P^3$,
      \item $F$, the class of a line in the exceptional divisor,
\end{itemize}
as well as a basis for the group  $H_1(\R P^3  \sharp \overline{\R P^3};\Z_2)$ as
\begin{itemize}
       \item  $\rho L_{\C P^3},\text{ the real part of the homology class of a line embedded in } \C P^3$,
      \item $\rho F$, the real part of the homology class of a line in the exceptional divisor.
\end{itemize}
We then choose a basis for the group $H_2^{-\tau }(\pp\sharp\overline{\C P^2};\Z)$ as
\begin{itemize}
    \item  $L_1=[\C P^1\times \{p\}]$,
     \item $L_2=[\{q\}\times \C P^1]$,
      \item $\hat{E}$, the exceptional divisor class,
\end{itemize}
where $p,q \in \C P^1$, as well as a basis for the group  $H_1(\R P^1 \times \R P^1\sharp {\R P^2};\Z_2)$ as
\begin{itemize}
    \item  $\rho L_1=[\R P^1\times \{p_\R\}]$,
     \item $\rho L_2=[\{q_\R\}\times \R P^1]$,
      \item $\rho \hat{E}$, the real part of the exceptional divisor class,
\end{itemize}
where $p_\R , q_\R \in \R P^1$. The $\Spin_3$-structure $\mathfrak{s}_{\R P^3\sharp \overline{\R P^3}}$ over $\R P^3\sharp \overline{\R P^3}$ is chosen such that the spinor state satisfies 
$$\sp_{\mathfrak{s}_{\R P^3\sharp \overline{\R P^3}},\mathfrak{o}_{\R P^3\sharp \overline{\R P^3}}}\left(f(\C P^1)\right)=+1,
$$
where $f\colon  (\C P^1,\tau_1) \to (\cpp\sharp\overline{\cpp},\tau)$ is a real immersion and $f_*[\C P^1]=L_{\C P^3}$. 
To simplify the notation, in the following we set  
$$(\R P^3)_1:=\R P^3\sharp \overline{\R P^3}$$ 
and 
$$(\R P^1)^2_1:=\R P^1\times \R P^1 \sharp {\R P^2}.$$
As a direct consequence of Lemma~\ref{lemma1}, the following lemma holds.

\begin{lemma} 
Let  $L= (1,0;0)\in H_2^{-\tau}(\pp\sharp\overline{\C P^2};\Z)$. Assume that $\R f^*\mathcal{N}_{L/\pp\sharp \overline{\C P^2}}=E^+$.
Then, the function
$$\epsilon\colon  H_2\left(\pp\sharp \overline{\C P^2};\Z\right)\lra \Z_2$$
$($see Definition~\ref{def-epsilon.function}\,$)$ satisfies
$\epsilon(a,b;k)=0$ if and only if $a>b$.
\end{lemma}

We are now prepared to state the following theorem.

\begin{theorem} \label{Theorem-W.in.deg.7}
 For every odd positive integer $d$, every non-negative integer $k \leq d$ and every integer $l$ such that $0\leq l \leq \frac{1}{2}(2d-1-k)$, one has
\begin{equation}
    \label{eq-Welschinger.formula.dim.7}
    W_{(\R P^3)_1}^{\mathfrak{s}_{(\R P^3)_1},\mathfrak{o}_{(\R P^3)_1}}((d;k),l)=\frac{1}{2}\sum_{a\in \Z}(-1)^{a+\frac{1}{2}(k+k^2)} (d-2a) W_{(\R P^1)^2_1} ((a,d-a;k),l).
\end{equation}
\end{theorem}
If $k$ is \textit{even}, then Formula (\ref{eq-Welschinger.formula.dim.7}) can be reduced to \cite[Theorem 1.1]{ding2020remark}.

\begin{proof}
A proof analogous to that of Theorem~\ref{theo-application.degree.8} can be employed. We highlight the main points as follows.  For every divisor $D=(a,b;k)\in H_2(\pp\sharp\overline{\C P^2};\mathbb{Z})$, it is obvious that $k_D=(2,2;1).(a,b;k)-1\equiv k+1 \mod 2$. The topological invariant $g(D)$ can be expressed as 
    \begin{align*}
        g(D)&=\tfrac{1}{2}\left((2,2;1).(a,b;k)+(a,b;k)^2+2\right)\\
        &\equiv a+b+ab+ \tfrac{1}{2}(k + k^2 )+1 \mod 2.
    \end{align*}

    Recall that $S\in \Z(1,-1;0)$. We consider a real balanced immersion $f\colon (\C P^1,\tau_1) \to (\C P^3 \sharp \overline{\C P^3},\tau)$ such that $\Im(f)\subset \pp \sharp \overline{\C P^2}$ and $f_*[\C P^1]\in H_2^{-\tau}(\pp\sharp\overline{\C P^2};\Z)$. It should be noted that the homology class $L_{\C P^3}=(1;0)\in H_2^{-\tau}(\C P^3 \sharp \overline{\C P^3};\Z)$ has two effective classes in its preimages under the surjective homomorphism
    $$\psi\colon H_2\left(\pp\sharp\overline{\C P^2};\Z\right)\lra H_2\left(\C P^3 \sharp \overline{\C P^3};\Z\right).$$
These correspond to the classes $(1,0;0)$ and $(0,1;0)$ in $H_2^{-\tau}(\pp\sharp\overline{\C P^2};\Z)$. It is also worth  noting that $\epsilon(L)=k_L+1 \equiv 0\mod 2$. Moreover, the restricted $\Pin^-_2$-structure  on $O(T ((\R P^1)^2_1))$ satisfies $s^{\mathfrak{s}_{(\R P^3)_1|(\R P^1)^2_1}}(1,0;0)=0$ and  $s^{\mathfrak{s}_{(\R P^3)_1|(\R P^1)^2_1}}(0,1;0)=1$. By the properties of quasi-quadratic enhancements, we obtain $$s^{\mathfrak{s}_{(\R P^3)_1|(\R P^1)^2_1}}(0,0;1)=0.$$
Hence, for all $a,b,k \in \Z$, we have
$$s^{\mathfrak{s}_{(\R P^3)_1|(\R P^1)^2_1}}(a,b;k) 
\equiv b+ab\mod 2.$$
Consequently, the two possibilities for the class $D$ can be established as follows: 
\begin{enumerate}
    \item If $D=(a,b;k)$ with $a>b$, then
$$(-1)^{\epsilon(D)+g(D)+s^{\mathfrak{s}_{(\R P^3)_1|(\R P^1)^2_1}}(\rho D)}=(-1)^{0+(a+b+ab+\frac{1}{2}(k+k^2)+1)+(b+ab)}.$$
Upon simplification, this sign is equal to $(-1)^{a+1+\frac{1}{2}(k+k^2)}$.
    \item If $D=(a,b;k)$ with $a<b$, then
$$(-1)^{\epsilon(D)+g(D)+s^{\mathfrak{s}_{(\R P^3)_1|(\R P^1)^2_1}}(\rho D)}=(-1)^{1+(a+b+ab+\frac{1}{2}(k+k^2)+1)+(b+ab)}.$$
Upon simplification, this sign is equal to $(-1)^{a+\frac{1}{2}(k+k^2)}$.
\end{enumerate}

As a result, if $a+b = 1 \mod 2$, then we have 
\begin{align*}
    W_{(\R P^3)_1}^{\mathfrak{s}_{(\R P^3)_1},\mathfrak{o}_{(\R P^3)_1}}((d;k),l)
    = & \frac{1}{2}\left(\sum_{0\leq a<\frac{d}{2}}(-1)^{a+\frac{1}{2}(k+k^2)} (d-2a) W_{(\R P^1)^2_1} ((a,d-a;k),l)\right.\\
    & \left.\hphantom{\frac{1}{2}\Big(}+ \sum_{d\geq a>\frac{d}{2}}(-1)^{a+1+\frac{1}{2}(k+k^2)} (-(d-2a)) W_{(\R P^1)^2_1} ((a,d-a;k),l) \right)\\
    = & \sum_{0\leq a<\frac{d}{2}}(-1)^{a+\frac{1}{2}(k+k^2)} (d-2a) W_{(\R P^1)^2_1} ((a,d-a;k),l).
\end{align*}
In addition, the genus $0$ Welschinger invariants $ W_{(\R P^1)^2_1} ((a,d-a;k),l)$ vanish when $a \notin [0,d]$; Formula (\ref{eq-Welschinger.formula.dim.7}) then holds.
\end{proof}

\begin{remark}  Directly from Corollary~\ref{coro-Theorem2-W} and the fact that $$|D.S|=|(a,d-a;k).(1,-1;0)|=| d-2a|\equiv d \mod 2,$$
it follows that the invariants $W_{(\R P^3)_1}^{\mathfrak{s}_{(\R P^3)_1},\mathfrak{o}_{(\R P^3)_1}}((d;k),l)$ vanish  for every $\Spin_3$-structure $\mathfrak{s}_{(\R P^3)_1}$ over $(\R P^3)_1$ if $d$ is even.
\end{remark}

\begin{proposition} \label{prop-application.deg7}
 Assume that $\R E$ is disconnected and that both the hyperplane class and the exceptional divisor class are on its pointed component. If $d$ is even and, moreover, the configuration $\underline{x}_d$ has an odd number of real points in the non-pointed component of\, $\R E$, then there does not exist any real irreducible rational curve of class $(d;k)$  in  $\C P^3 \sharp \overline{\C P^3}$ passing through $\underline{x}_d$. 
\end{proposition}

\begin{proof}
  This proposition is a consequence of Theorem~\ref{Theorem3-W}
    and the following remarks: 
    \begin{enumerate}
        \item The base locus of the pencil of surfaces in the linear system $|-\frac{1}{2}K_{\C P^3\sharp \overline{\C P^3}}|$ is an elliptic curve $E$ of homology class 
        $$\tfrac{1}{4}K_{\C P^3\sharp \overline{\C P^3}}^2=(4;1)\in H_2\left(\C P^3\sharp \overline{\C P^3};\Z\right).$$
        Moreover, the curve $E$ is also of homology class 
        $$(2,2;1)\in  H_2\left(\C P^1\times \C P^1 \sharp \overline{\C P^2};\Z\right).$$
        \item By the intersection form on $\C P^3 \sharp \overline{\C P^3}$ and the surjectivity of $\psi$, the hyperplane section $\tilde{H}$ is of class 
        $$\tilde{H}=(1,1;0) \in H_2\left(\C P^1\times \C P^1 \sharp \overline{\C P^2};\Z\right) \cong \Pic\left(\C P^1\times \C P^1 \sharp \overline{\C P^2}\right).$$
        Moreover, the class $\tilde{H}$ can be written as
        $\tilde{H}=L+T(L)=L_1+L_2$.
        \item  Let $\tilde{h}:=\tilde{H}|_E \in \Pic_4(E)$ be the hyperplane class on $E$. The constant divisor $ \phi(\overline{L},\overline{D})$, which has degree $D.E+(D.S)(L.E)=4(d-a)-k$, can be described  as 
        \begin{align*}
           \phi\left(\overline{L},\overline{D}\right)&=(a,d-a;k)|_E-(d-2a) (1,0;0)|_E\\
           &=(d-a)\tilde{h}-k\hat{E}|_E.
        \end{align*}
        Hence, the position of the divisor $\phi(\overline{L},\overline{D})$ on $\R E$ depends on the positions of the hyperplane class~$\tilde{h}$ and the exceptional divisor class $\hat{E}|_E$ on $E$. Moreover, if $\underline{x}_d$ has an odd number of real points on the non-pointed component $\mathbb{S}^1_1$ of $\R E$, then $[\underline{x}_d]\in \mathbb{S}^1_1$. By Corollary~\ref{coro-degree.case.real}\eqref{c-d.c.r-c}, the statement then follows.\hfill\qed
    \end{enumerate}
    \renewcommand{\qed}{}
\end{proof}

%%%%%%%%%%%%%%%%%%%%%%%%%%%%%%%%%%
  \subsubsection{Real threefold product of the projective line} \label{Subsect-exam.deg.6.W}
  
   As has been observed in Section~\ref{sect-applications.GW}, non-singular surfaces $\Sigma\in |-\frac{1}{2}K_{\ppp}|$ are non-singular  del Pezzo surfaces of degree $6$. In particular, they are  non-singular tridegree $(1,1,1)$-surfaces  in $\ppp$. Moreover, we have the isomorphism $\Sigma\cong (\pp \sharp 2 \overline\cp)$.
   
   For the sake of notational simplicity, in this paragraph, we set 
   $$\left(\C P^1\right)^3:=\ppp$$
   and  
   $$\left(\C P^1\right)^2_2:=\pp\sharp 2\overline{\C P^2}.$$
   We recall the following identifications (see Section~\ref{Subsect-exam.deg.6.GW}): 
\begin{itemize}
    \item The homology group $H_2((\C P^1)^2_2;\mathbb{Z})=\,<(L_1,L_2;E_1,E_2)>$ is identified with  $\Z^4$, meaning that we identify a homology class of the form $a L_1+b L_2-\alpha E_1 - \beta E_2$  in $H_2((\C P^1)^2_2;\mathbb{Z})$  with $(a,b;\alpha,\beta)$. Note that $(a,b;\alpha,\beta)$ is an effective class if $(a,b)\in \mathbb{N}^2$ and $(\alpha,\beta)\in \mathbb{Z}^2$ are such that if $a=b=0$, then both~$\alpha$ and $\beta$ are negative; otherwise, both $\alpha$ and $\beta$ are non-negative.
    \item The homology group   $H_2((\C P^1)^3;\mathbb{Z})=\,<(M_1,M_2,M_3)>$ is identified  with  $\Z^3$, meaning that we identify a homology class of the form $a_1 M_1 +a_2 M_2+a_3 M_3$ in $ H_2((\C P^1)^3;\mathbb{Z})$ with $(a_1,a_2,a_3)$. Note that $(a_1,a_2,a_3)$ is an effective class if $a_i\in \mathbb{N}$ for all $i\in \{1,2,3\}$ and there exists at least one $a_i$ such that $a_i>0$.
\end{itemize}
Moreover, the inclusion of $(\C P^1)^2_2$ in $(\C P^1)^3$ implies the surjective map
    \begin{align*}
        \psi\colon H_2\left(\left(\C P^1\right)^2_2;\mathbb{Z}\right)&\longrightarrow H_2\left(\left(\C P^1\right)^3;\mathbb{Z}\right)\\
       (a,b;\alpha,\beta)&\longmapsto (a,b,a+b-\alpha-\beta).
    \end{align*}
Since $\psi (0,0;\alpha,-\alpha)=(0,0,0)$, we have $\ker(\psi)=\Z(0,0;1,-1)$. 
      For every $D=(a,b;\alpha,\beta)\in H_2((\C P^1)^2_2;\mathbb{Z})$, we have

    $$k_D=(2,2;1,1).(a,b;\alpha,\beta)-1\equiv  \alpha + \beta +1 \mod 2$$
    and
    \begin{align*}
        g(D)&=\tfrac{1}{2}\left((2,2;1,1).(a,b;\alpha,\beta)+\left(a,b;\alpha,\beta\right)^2+2\right)\\
        &=a+b+ab+ \tfrac{1}{2}\left(\alpha + \alpha^2 \right) + \tfrac{1}{2}\left(\beta+\beta^2\right)+1 \mod 2.
    \end{align*}

 We consider two real structures on $(\C P^1)^3$: the \textit{standard} one, denoted by $\tau_{1,1}$, whose fixed locus is $\rprprp$, and the \textit{twisted} one, denoted  by $\tilde{\tau}_{1,1}$, whose fixed locus is $\mathbb{S}^2\times \R P^1$.

 \begin{remark}
    Since we are considering real  $(1,1,1)$-surfaces on $(\C P^1)^3$ for which  $S$ is $\tau_{1,1}$-anti-invariant and $\tilde{\tau}_{1,1}$-anti-invariant, the two points at which the non-singular quadric is blown up are real points.
\end{remark}   

 Thus, the restrictions of these two real structures to $(\C P^1)_2^2$  yield, as their respective real loci, the surfaces
$\R P^1\times \R P^1 \sharp 2 {\R P^2}$ and $\mathbb{S}^2 \sharp 2 {\R P^2}$.

 \subsubsection*{Case 1: The standard real structure}

For convenience of notation, we set 
 $$\left(\R P^1\right)^3:=\R P^1\times \R P^1 \times \R P^1$$
and  
$$\left(\R P^1\right)^2_2:=\R P^1\times \R P^1 \sharp 2 {\R P^2}.$$
We may choose a basis for the group $H_2^{-\tau_{1,1} }((\C P^1)^3;\Z)$ as the basis for the group $H_2((\C P^1)^3;\Z)$, that is,
    $$(M_1,M_2,M_3)=\left(\left[\C P^1\times \{p_1\}\times \{p_2\}\right],\left[\{q_1\}\times \C P^1\times \{q_2\}\right],\left[\{r_1\}\times \{r_2\}\times \C P^1\right]\right),$$
    where $p_i,q_i,r_i \in \C P^1$,
and choose a basis for the group  $H_1((\R P^1)^3;\Z_2)$ as
 $$(\rho M_1,\rho M_2,\rho M_3)=\left(\left[\R P^1\times \{p_{1\R}\}\times \{p_{2\R}\}\right],\left[\{q_{1\R}\}\times \R P^1\times \{q_{2\R}\}\right],\left[\{r_{1\R}\}\times \{r_{2\R}\}\times \R P^1\right]\right),$$
    where $p_{i,\R},q_{i,\R},r_{i,\R} \in \R P^1$ ($i\in \{1,2\}$).
Suppose that both distinct blow-up points $x_1\neq x_2$  are in $\R P^1\times \R P^1$. We may choose a basis for the group $H_2^{-\tau_{1,1}}((\C P^1)_2^2;\Z)$ as the basis for the group $H_2((\C P^1)_2^2;\Z)$, that is,
\begin{itemize}
    \item  $L_1=[\C P^1\times \{p\}]$,
     \item $L_2=[\{q\}\times \C P^1]$,
      \item $E_1$, the exceptional divisor class over $x_1$,
       \item $E_2$, the exceptional divisor class over $x_2$,
\end{itemize}
where $p,q \in \C P^1$,
and choose a basis for the group  $H_1((\R P^1)_2^2;\Z_2)$ as
\begin{itemize}
    \item  $\rho L_1=[\R P^1\times \{p_\R\}]$,
     \item $\rho L_2=[\{q_\R\}\times \R P^1]$,
      \item $\rho E_1$,  the real part of the exceptional divisor class over $x_1$,
      \item $\rho E_2$,  the real part of the exceptional divisor class over $x_2$,
\end{itemize}
where $p_\R , q_\R \in \R P^1.$
The $\Spin_3$-structure $\mathfrak{s}_{(\R P^1)^3}$ over $(\R P^1)^3$ is chosen such that the spinor state satisfies $\sp_{\mathfrak{s}_{(\R P^1)^3},\mathfrak{o}_{(\R P^1)^3}}(f(\C P^1))=+1$, where $f\colon (\C P^1,\tau_1)\to ((\C P^1)^3,\tau_{1,1})$ is a real immersion and
$$f_*\left[\C P^1\right]\in \{(0,0,1),(0,1,0),(0,0,1)\}\in H_2^{-\tau_{1,1}}\left(\left(\C P^1\right)^3;\Z\right).$$ 

\begin{lemma} \label{lemma2}
Let $L= (0,0;0,-1)\in H_2^{-\tau_{1,1}}( (\C P^1)^2_2;\Z)$. Assume that $\R f^*\mathcal{N}_{L/(\C P^1)^2_2}=E^+$. Then, the function $\epsilon\colon  H_2((\C P^1)^2_2;\Z)\to \Z_2$ $($see Definition~\ref{def-epsilon.function}\,$)$ satisfies $\epsilon(a,b;\alpha,\beta)=0$ if and only if $\alpha > \beta$.
\end{lemma}

\begin{proof}
  By definition, it is obvious that $\epsilon(L)=k_L+1 \equiv 0+0+1+1\equiv 0 \mod 2$. Recall that the vanishing class~$S$ is an element of  $\Z(0,0;1,-1)=\ker(H_2((\C P^1)_2^2;\Z)\to H_2((\C P^1)^3;\Z))$. By a straightforward computation, we have $(D.S)(L.S) = \alpha - \beta.$
Thus, the lemma follows from Proposition~\ref{prop-function.epsilon2}.
\end{proof}

The formula induced from Theorem~\ref{Theorem2-W} in this case is given in the following theorem.

 \begin{theorem} \label{Theorem-W.in.deg.6.1}
For every triple of non-negative integers $(a,b,c)$ such that $a+b+c$ is odd and every integer $l$ such that $0\leq l \leq \tfrac{1}{2}( a+b+c-1)$, one has
\begin{equation} \label{Eq:W.in.deg.6.1}
    W_{(\R P^1)^3}^{\mathfrak{s}_{(\R P^1)^3},\mathfrak{o}_{(\R P^1)^3}}((a,b,c),l)
    =\sum_{0\leq\alpha<\frac{|a+b-c|}{2}}(-1)^{ \alpha +\frac{a+b-c-1}{2}} (a+b-c-2\alpha) W_{(\R P^1)^2_2} ((a,b;\alpha,a+b-c-\alpha),l). 
\end{equation}
\end{theorem}

 \begin{proof}
Let $D=(a,b;\alpha,a+b-c-\alpha)\in H_2^{-\tau_{1,1}}((\C P^1)^2_2;\Z)$. 
It is obvious that $k_D\equiv a+b+c+1 \mod 2$. The topological invariant $g(D)$ can be expressed as
        \begin{align*}
        g(D)&=a+b+ab+
    \alpha^2+ \alpha(a+b-c)+\tfrac{1}{2}(a+b-c)(a+b-c-1)+1\\
        & \equiv    a+b+1+ab+\tfrac{1}{2}(a+b-c-1)\mod 2 \quad \text{if } a+b+c=1 \mod 2. 
    \end{align*}

We consider a real  balanced immersion $f\colon (\C P^1,\tau_1) \to ((\C P^1)^3,\tau_{1,1})$ such that $\Im(f)\subset (\C P^1)_2^2$ and $f_*[\C P^1]\in H_2^{-\tau_{1,1}}((\C P^1)^2_2;\Z)$. It should be noted that the homology class $(0,0,1)\in H_2^{-\tau_{1,1}}((\C P^1)^3;\Z)$ has two effective classes in its preimages under the surjective homomorphism $\psi\colon H_2((\C P^1)^2_2;\Z)\to H_2((\C P^1)^3;\Z)$. These correspond to the classes $(0,0;0,-1)$ and $(0,0;-1,0)$ in $H_2^{-\tau_{1,1}}( (\C P^1)^2_2;\Z)$.

Moreover, the restricted $\Pin^-_2$-structure  on $O(T (\R P^1)^2_2))$ satisfies  $s^{\mathfrak{s}_{(\R P^1)^3|(\R P^1)^2_2}}(0,0;0,-1)=0$ and $s^{\mathfrak{s}_{(\R P^1)^3|(\R P^1)^2_2}}(0,0;-1,0)=1$. 
The spinor states of a  curve of homology class  $(1,0,0)$, $(0,1,0)$ or $(0,0,1)$ in $(\C P^1)^3$ are identified, meaning that
$$\sp_{\mathfrak{s}_{(\R P^1)^3},\mathfrak{o}_{(\R P^1)^3}}(0,0,1)=\sp_{\mathfrak{s}_{(\R P^1)^3},\mathfrak{o}_{(\R P^1)^3}}(1,0,0)=\sp_{\mathfrak{s}_{(\R P^1)^3},\mathfrak{o}_{(\R P^1)^3}}(0,1,0).$$ It follows that

$$s^{\mathfrak{s}_{(\R P^1)^3|(\R P^1)^2_2}}(1,0;1,0)\equiv s^{\mathfrak{s}_{(\R P^1)^3|(\R P^1)^2_2}}(0,1;1,0)\equiv 0 \mod 2.$$
Consequently, we have
$$s^{\mathfrak{s}_{(\R P^1)^3|(\R P^1)^2_2}}(1,0;0,0)\equiv s^{\mathfrak{s}_{(\R P^1)^3|(\R P^1)^2_2}}(1,0;1,0) +s^{\mathfrak{s}_{(\R P^1)^3|(\R P^1)^2_2}}(0,0;-1,0) \equiv 1 \mod 2$$
and
$$s^{\mathfrak{s}_{(\R P^1)^3|(\R P^1)^2_2}}(0,1;0,0) \equiv s^{\mathfrak{s}_{(\R P^1)^3|(\R P^1)^2_2}}(0,1;1,0) +s^{\mathfrak{s}_{(\R P^1)^3|(\R P^1)^2_2}}(0,0;-1,0)\equiv 1 \mod 2.$$
As a result, we obtain
\begin{align*}
    &s^{\mathfrak{s}_{(\R P^1)^3|(\R P^1)^2_2}}(a,b;\alpha,\beta)\\
    &\equiv  \; a s^{\mathfrak{s}_{(\R P^1)^3|(\R P^1)^2_2}}(1,0;0,0)+bs^{\mathfrak{s}_{(\R P^1)^3|(\R P^1)^2_2}}(0,1;0,0) \\
    &\hphantom{\equiv}\;\; -\alpha s^{\mathfrak{s}_{(\R P^1)^3|(\R P^1)^2_2}}(0,0;-1,0)-\beta  s^{\mathfrak{s}_{(\R P^1)^3|(\R P^1)^2_2}}(0,0;0,-1)+ab \mod 2\\
    & \equiv  \; a+b+ab+\alpha \mod 2.
\end{align*}
We have the following two possibilities for the homology class $D$:
\begin{enumerate}
    \item If $D=(a,b;\alpha,\beta)$ with $\alpha>\beta$ and  $\alpha + \beta= a+b-c \equiv 1 \mod 2$, then
$$(-1)^{\epsilon(D)+g(D)+s^{\mathfrak{s}_{(\R P^1)^3|(\R P^1)^2_2}}(\rho D)}=(-1)^{0+(a+b+1+ab+\frac{1}{2}(a+b-c-1))+(a+b+ab+\alpha)}.$$
Upon simplification, this sign becomes $(-1)^{\alpha+1+\frac{1}{2}(a+b-c-1)}$.
    \item  If $D=(a,b;\alpha,\beta)$ with $\alpha<\beta$ and $\alpha + \beta= a+b-c \equiv 1 \mod 2$, then
$$(-1)^{\epsilon(D)+g(D)+s^{\mathfrak{s}_{(\R P^1)^3|(\R P^1)^2_2}}(\rho D)}=(-1)^{1+(a+b+1+ab+\frac{1}{2}(a+b-c-1))+(a+b+ab+\alpha)}.$$
Upon simplification,  this sign becomes $(-1)^{\alpha+\frac{1}{2}(a+b-c-1)}$.
\end{enumerate}
In conclusion, if $a+b +c= 1 \mod 2$, then we have
\begin{align*}
    & W_{(\R P^1)^3}^{\mathfrak{s}_{(\R P^1)^3},\mathfrak{o}_{(\R P^1)^3}}((a,b,c),l)\\
  &=\frac{1}{2}\left(\sum_{0\leq \alpha<\frac{|a+b-c|}{2}}(-1)^{\alpha +\frac{a+b-c-1}{2}} (a+b-c-2\alpha) W_{(\R P^1)^2_2} ((a,b;\alpha,a+b-c-\alpha),l)\right.\\
  &\hphantom{=\frac{1}{2}\Big(}\left.\; +\sum_{|a+b-c|\geq \alpha>\frac{|a+b-c|}{2}}(-1)^{\alpha +1+\frac{a+b-c-1}{2}} (-(a+b-c-2\alpha)) W_{(\R P^1)^2_2} ((a,b;\alpha,a+b-c-\alpha),l) \right)\\
    &=\sum_{0\leq \alpha<\frac{|a+b-c|}{2}}(-1)^{\alpha +\frac{a+b-c-1}{2}} (a+b-c-2\alpha) W_{(\R P^1)^2_2} ((a,b;\alpha,a+b-c-\alpha),l).
\end{align*}
In addition, the genus $0$ Welschinger invariants $W_{(\R P^1)^2_2} ((a,b;\alpha,a+b-c-\alpha),l)$ vanish when $\alpha \notin [0,|a+b-c|]$;  
Formula (\ref{Eq:W.in.deg.6.1}) then holds.
\end{proof}

We call a \textit{$(x,y,z)$-surface class on an elliptic curve $E$} the restriction of  a surface in $(\C P^1)^3$ of \textit{tridegree} $(x,y,z)$ to $E$.
For example, the $(0,0,1)$-surface class on $E$ is $h_3:=(1,1;1,1)|_E\in \Pic_2(E)$. In other words, we can write
$$h_3=L_1{|_E}+L_2{|_E}-E_1{|_E}-E_2{|_E}.$$
It should be noted that
$$L_1{|_E},L_2{|_E} \in \Pic_2(E), \quad E_1{|_E},E_2{|_E}\in \Pic_1(E),$$ 
and 
$$E_1{|_E}+E_2{|_E}=h_{1,1}-h_3,$$
where $h_{a,b}$ is the $(a,b,0)$-surface class on $E$. 

\begin{remark}\label{remark12}
Directly from Corollary~\ref{coro-Theorem2-W} and the fact that $$|D.S|=|(a,b;\alpha,a+b-c-\alpha).(0,0;1,-1)|=|a+b-c-2\alpha|\equiv a+b+c \mod 2,$$
it follows that the invariants $W_{(\R P^1)^3}^{\mathfrak{s}_{(\R P^1)^3},\mathfrak{o}_{(\R P^1)^3}}((a,b,c),l)$ vanish if $a+b+c$ is even for every $\Spin_3$-structure $\mathfrak{s}_{(\R P^1)^3}$ over $(\R P^1)^3$.
\end{remark}

\begin{proposition} \label{prop-application.deg6.1}
Assume that $\R E$ is disconnected and that the  $(0,0,1)$-surface class,  the  $(1,1,0)$-surface class and the  $(b,a,0)$-surface class $(a,b\in \mathbb{N}^*)$ are on the pointed component of\, $\R E$.
If $a+b+c$ is even and, moreover, the real configuration $\underline{x}_d$ of\, $k_d=a+b+c$ points on $E$ 
    has an odd number of real points on each component of\, $\R E$, then there does not exist any real irreducible rational curve of class $(a,b,c)\in H_2^{-\tau_{1,1}}((\C P^1)^3;\Z)$ passing through~$\underline{x}_d$. 
\end{proposition}

\begin{proof}
This is a consequence of Theorem~\ref{Theorem3-W}  
    and the following remarks.
    \begin{enumerate}
        \item The base locus of the pencil of surfaces in the linear system $|-\frac{1}{2}K_{(\C P^1)^3}|$ is an elliptic curve $E$ of homology class 
        $$\tfrac{1}{4}K_{(\C P^1)^3}^2=(2,2,2)\in H_2^{-{\tau}_{1,1}}\left(\left(\C P^1\right)^3;\Z\right).$$
        Moreover, the curve $E$ is also of homology class 
        $$(2,2;1,1)\in  H_2^{-{\tau}_{1,1}}\left(\left(\C P^1\right)^2_2;\Z\right).$$
      \item By the intersection forms of $(\C P^1)^3$ and of $(\C P^1)^2_2$ as previously described,  the divisor $L+T(L)$ can be written as
        $$L+T(L)=E_1+E_2=(0,0;-1,-1)\in H_2\left(\left(\C P^1\right)^2_2;\Z\right)\cong \Pic\left(\left(\C P^1\right)^2_2\right).$$ 
        Taking its restriction to $E$,  we obtain the divisor
        $$\overline{L}_{\phi}=h_{1,1}-h_3 \in \Pic_2(E),$$
        which is determined by the $(0,0,1)$-surface class  and the  $(1,1,0)$-surface class on $E$.
        \item Moreover, if $0\leq \alpha\leq \frac{1}{2}(a+b-c)$, then the divisor $\phi(\overline{L},\overline{D})$, which has degree $D.E-|D.S|(L.E)= 2(c+\alpha)$, can be described as
        \begin{align*}
            \phi\left(\overline{L},\overline{D}\right)&=(a,b;a+b-c-\alpha,\alpha)|_E-(a+b-c-2\alpha)(0,0;0,-1)|_E\\
           & =(a,b;0,0)|_E-(a+b-c-\alpha) (0,0;-1,-1)|_E\\
           & = h_{b,a}-(a+b-c-\alpha)(h_{1,1}-h_3).
        \end{align*} 
            Consequently, Equation (\ref{eq-solutions.in.PicE}) becomes
\begin{equation}
     (a+b-c-2\alpha)l=\underline{x}_{d}-h_{b,a}+(a+b-c-\alpha)(h_{1,1}-h_3).
\end{equation}
    \end{enumerate}

    By hypothesis and by Corollary~\ref{coro-degree.case.real}\eqref{c-d.c.r-c}, the statement follows.
\end{proof}

\subsubsection*{Case 2: The twisted real structure}
  
    We may choose a basis for the group $H_2^{-\tilde{\tau}_{1,1} }((\C P^1)^3;\Z)$ as
    $$(M_1+M_2,M_3)=\left(\left[\C P^1\times \{p_1\}\times \{p_2\}\right]+\left[\{q_1\}\times \C P^1\times \{q_2\}\right],\left[\{r_1\}\times \{r_2\}\times \C P^1\right]\right),$$
    where $p_i,q_i,r_i \in \C P^1$. Hence, the group $H_2^{-\tilde{\tau}_{1,1} }((\C P^1)^3;\Z)$  can be identified with $\Z^2$. Note that the group $H_1(\mathbb{S}^2\times \R P^1;\Z_2)$ has a generator $\rho M_3:=[\{r_{\R}\}\times \R P^1]$,
    where $r_{\R} \in \mathbb{S}^2$. Also note that the two blow-up points are distinct; they are denoted by $x_1\neq x_2\in \mathbb{S}^2$. We may choose a basis for the group $H_2^{-\tilde{\tau}_{1,1}}((\C P^1)_2^2;\Z)$ as
\begin{itemize}
    \item  $L_1+L_2=[\C P^1\times \{p\}]+[\{q\}\times \C P^1]$,
      \item $E_1$, the exceptional divisor class over $x_1$,
       \item $E_2$, the exceptional divisor class over $x_2$,
\end{itemize}
where $p,q \in \C P^1$. Hence, the group $H_2^{-\tilde{\tau}_{1,1}}((\C P^1)^2_2 ;\Z)$  can be identified with $\Z^3$. Note that the  two generators of the group $H_1(\mathbb{S}^2\sharp 2 {\R P^2};\Z_2) $  are
\begin{itemize}
      \item $\rho E_1$, the real part of the exceptional divisor class over $x_1$, and
      \item $\rho E_2$, the real part of the exceptional divisor class over $x_2$.
\end{itemize}

We can write $H_1(\mathbb{S}^2\sharp 2 {\R P^2};\Z_2) =\, <(0;-1,0),(0;0,-1)> $.  The $\Spin_3$-structure $\mathfrak{s}_{\mathbb{S}^2 \times \R P^1}$ over $\mathbb{S}^2 \times \R P^1$ is chosen such that
$$\sp_{\mathfrak{s}_{\mathbb{S}^2 \times \R P^1},\mathfrak{o}_{\mathbb{S}^2 \times \R P^1}}\left(f\left(\C P^1\right)\right)=+1,$$
where $f\colon (\C P^1,\tau_1)\to ((\C P^1)^3,\tilde{\tau}_{1,1})$ is a real immersion and $f_*[\C P^1]=(0;1)\in H_2^{-\tilde{\tau}_{1,1}}((\C P^1)^3;\Z)$.

Lemma~\ref{lemma2} directly implies  the following lemma.

\begin{lemma}
Let $L= (0;0,-1)\in H_2^{-\tilde{\tau}_{1,1}}( (\C P^1)^2_2;\Z)$. Assume that $\R f^*\mathcal{N}_{L/(\C P^1)^2_2}=E^+$. Then, the function $\epsilon\colon  H_2((\C P^1)^2_2;\Z)\to \Z_2$ $($see Definition~\ref{def-epsilon.function}\,$)$ satisfies $$\epsilon(a;\alpha,\beta)=0 \quad\text{if and only if}\quad \alpha > \beta.$$
\end{lemma}

The formula induced from Theorem~\ref{Theorem2-W} in this case is given in the following theorem.

\begin{theorem} \label{Theorem-W.in.deg.6.2}
For every non-negative integer $a$, every odd positive integer $c$ and every integer $l$ such that $0\leq l \leq \frac{1}{2}(2a+c-1)$, one has 
\begin{equation}\label{Eq:W.in.deg.6.2}
    W_{\mathbb{S}^2 \times \R P^1}^{\mathfrak{s}_{\mathbb{S}^2 \times \R P^1},\mathfrak{o}_{\mathbb{S}^2 \times \R P^1}}((a;c),l)=\sum_{0\leq\alpha<\frac{|2a-c|}{2}}(-1)^{ \alpha +\frac{c+1}{2}} (2a-c-2\alpha) W_{\mathbb{S}^2\sharp 2 {\R P^2}} ((a;\alpha,2a-c-\alpha),l).
\end{equation}
\end{theorem}

\begin{proof}
  A proof analogous to that of Theorem~\ref{Theorem-W.in.deg.6.1}  can be employed. We emphasize some distinguishing points 
  that come from  changing real structures on $(\C P^1)^3$ as follows.

Let $D=(a;\alpha,2a-c-\alpha)\in H_2^{-\tilde{\tau}_{1,1}}((\C P^1)^2_2;\Z)$. 
It is obvious that $k_D\equiv c+1 \mod 2$. The topological invariant $g(D)$ can be expressed as
\begin{align*}
    g(D)&= 2a+a^2+\alpha^2+(2a-c)\alpha +\tfrac{1}{2}(2a-c)(2a-c-1)+1\\
    &\equiv \alpha^2+ c\alpha+1+\tfrac{1}{2}c(c+1) \mod 2\\
    & \equiv \tfrac{1}{2}(c-1)\mod 2 \quad \text{if } c\equiv 1 \mod 2.
\end{align*}
We consider a real balanced immersion $f\colon (\C P^1,\tau_1) \to ((\C P^1)^3,\tilde{\tau}_{1,1})$ such that $\Im(f)\subset (\C P^1)^2_2$ and $f_*[\C P^1]\in H_2^{-\tilde{\tau}_{1,1}}( (\C P^1)^2_2;\Z)$. It is to be noted that the homology class $(0;1)\in H_2^{-\tilde{\tau}_{1,1}}((\C P^1)^3;\Z)$ has two effective classes in its preimages under the surjective homomorphism $\psi\colon H_2((\C P^1)^2_2;\Z)\to H_2((\C P^1)^3;\Z)$. These correspond to the classes $(0;0,-1)$ and $(0;-1,0)$ in $H_2^{-\tilde{\tau}_{1,1}}( (\C P^1)^2_2;\Z)$. 

Moreover, the restricted $\Pin^-_2$-structure  on $O(T \mathbb{S}^2 \times \R P^1)$ satisfies $$s^{\mathfrak{s}_{\mathbb{S}^2 \times \R P^1|\mathbb{S}^2 \times \R P^1}}(0;0,-1)=0 \quad\text{and}\quad  s^{\mathfrak{s}_{\mathbb{S}^2 \times \R P^1|\mathbb{S}^2 \times \R P^1}}(0;-1,0)=1.$$ Since $H_1(\mathbb{S}^2\sharp 2 {\R P^2};\Z_2)=<(0;-1,0),(0;0,-1)>$, one has
\begin{align*}
    s^{\mathfrak{s}_{\mathbb{S}^2 \times \R P^1|\mathbb{S}^2 \times \R P^1}}(0;-\alpha,-\beta)
    &\equiv  \alpha s^{\mathfrak{s}_{\mathbb{S}^2 \times \R P^1|\mathbb{S}^2 \times \R P^1}}(0;-1,0)+\beta  s^{\mathfrak{s}_{\mathbb{S}^2 \times \R P^1|\mathbb{S}^2 \times \R P^1}}(0;0,-1) \mod 2\\
   & \equiv \alpha  \mod 2.
\end{align*}

As a consequence, the two possibilities for the class $D$ are established as follows:
\begin{enumerate}
    \item If $D=(a;\alpha,\beta)$ with $\alpha>\beta$ and $\alpha +\beta =2a-c \equiv 1 \mod 2$, then we have
$$(-1)^{\epsilon(D)+g(D)+s^{\mathfrak{s}_{\mathbb{S}^2 \times \R P^1|\mathbb{S}^2 \times \R P^1}}(\rho D)}=(-1)^{0+\frac{1}{2}(c-1)+\alpha}.$$
    \item  If $D=(a;\alpha,\beta)$ with $\alpha<\beta$ and $\alpha +\beta =2a-c \equiv 1 \mod 2$, then we have
$$(-1)^{\epsilon(D)+g(D)+s^{\mathfrak{s}_{\mathbb{S}^2 \times \R P^1|\mathbb{S}^2 \times \R P^1}}(\rho D)}=(-1)^{1+\frac{1}{2}(c-1)+\alpha }.$$
\end{enumerate}

In conclusion, if $c= 1 \mod 2$, then we obtain
\begin{align*}
     W_{\mathbb{S}^2 \times \R P^1}^{\mathfrak{s}_{\mathbb{S}^2 \times \R P^1},\mathfrak{o}_{\mathbb{S}^2 \times \R P^1}}((a;c),l)&=\frac{1}{2}\left(\sum_{0\leq \alpha<\frac{|2a-c|}{2}}(-1)^{\alpha +\frac{c+1}{2}} (2a-c-2\alpha) W_{\mathbb{S}^2\sharp 2 {\R P^2}} ((a;\alpha,2a-c-\alpha),l)\right.\\
  &\hphantom{=\frac{1}{2}\Big(}\left.\; + \sum_{|2a-c|\geq \alpha>\frac{|2a-c|}{2}}(-1)^{\alpha +1+\frac{c+1}{2}} (-(2a-c-2\alpha)) W_{\mathbb{S}^2\sharp 2 {\R P^2}} ((a;\alpha,2a-c-\alpha),l) \right)\\
    &=\sum_{0\leq \alpha<\frac{|2a-c|}{2}}(-1)^{\alpha +\frac{c+1}{2}} (2a-c-2\alpha) W_{\mathbb{S}^2\sharp 2 {\R P^2}} ((a;\alpha,2a-c-\alpha),l).
\end{align*}

In addition, the genus $0$ Welschinger invariants $W_{\mathbb{S}^2\sharp 2 {\R P^2}} ((a;\alpha,2a-c-\alpha),l)$ vanish when $\alpha \notin [0,|2a-c|]$;  
Formula (\ref{Eq:W.in.deg.6.2}) then holds.
\end{proof}

It should be noted that, in the present case, there are two specific types of classes:
\begin{enumerate}
    \item the $(0,0,1)$-surface class on $E$, that is, the class $$h_3:=(1;1,1)|_E=(L_1+L_2){|_E}-E_1{|_E}-E_2{|_E}\in \Pic_2(E);$$
\item the $(a,a,0)$-surface class on $E$, that is, the class
$$h_{a,a}:=(a;0,0)|_E=a(L_1+L_2){|_E}\in \Pic_{4a}(E).$$
\end{enumerate}

\begin{remark}\label{remark13}
Directly from Corollary~\ref{coro-Theorem2-W} and the fact that $$|D.S|=|(a;\alpha,2a-c-\alpha).(0;1,-1)|=|2a-c-2\alpha|\equiv  c \mod 2,$$
it follows that the Welschinger invariants $W_{\mathbb{S}^2 \times \R P^1}^{\mathfrak{s}_{\mathbb{S}^2 \times \R P^1},\mathfrak{o}_{\mathbb{S}^2 \times \R P^1}}((a;c),l)$ vanish if $c$ is even for every $\Spin_3$-structure $\mathfrak{s}_{\mathbb{S}^2 \times \R P^1}$ over $\mathbb{S}^2 \times \R P^1$.
\end{remark}

\begin{proposition} \label{prop-application.deg6.2}
Assume that $\R E$ is disconnected and that both the $(0,0,1)$-surface class and the $(a,a,0)$-surface class $(a\in \mathbb{N}^*)$  are on the pointed component of\, $\R E$.     If $c$ is even and, moreover, the real configuration $\underline{x}_d$ of $k_d=2a+c$ points on $E$ has an odd number of real points on each component of\, $\R E$, then there does not exist any real irreducible rational curve of class $(a;c)$ in $H_2^{-\tilde{\tau}_{1,1}}((\C P^1)^3;\Z)$  passing through $\underline{x}_d$. 
\end{proposition}

\begin{proof}
This is a consequence of Theorem~\ref{Theorem3-W}, Proposition~\ref{prop-application.deg6.1}  and the following remarks:
    \begin{enumerate}
        \item The base locus of the pencil of surfaces in the linear system $|-\frac{1}{2}K_{(\C P^1)^3}|$ is an elliptic curve $E$ of homology class 
        $$\tfrac{1}{4}K_{(\C P^1)^3}^2=(2;2)\in H_2^{-\tilde{\tau}_{1,1}}\left(\left(\C P^1\right)^3;\Z\right).$$
        Moreover, the curve $E$ is also of homology class 
        $$(2;1,1)\in  H_2^{-\tilde{\tau}_{1,1}}\left(\left(\C P^1\right)^2_2;\Z\right).$$
        
      \item By the intersection form on $H_4((\C P^1)^3;\mathbb{Z})$, the class $L+T(L)$ can be written as
        $$L+T(L)=E_1+E_2=(0;-1,-1)\in H_2^{-\tilde{\tau}_{1,1}}\left(\left(\C P^1\right)^2_2;\Z\right)\cong \Pic\left(\left(\C P^1\right)^2_2\right).$$ 
        Taking its restriction to $E$,  we obtain the divisor
        $$\overline{L}_{\phi}=h_{1,1}-h_3 \in \Pic_2(E),$$
        which is formed by the $(0,0,1)$-surface class  and the  $(1,1,0)$-surface class on $E$.
        \item Moreover, if $0\leq \alpha\leq \frac{1}{2}(2a-c)$, then the divisor $\phi(\overline{L},\overline{D})$, which has degree $D.E-|D.S|(L.E)= 2(c+\alpha)$,  can be described as
        \begin{align*}
            \phi\left(\overline{L},\overline{D}\right)&=(a;2a-c-\alpha,\alpha)|_E-(2a-c-2\alpha)(0;0,-1)|_E\\
           & =(a;0,0)|_E-(2a-c-\alpha) (0;-1,-1)|_E\\
           & = h_{a,a}-(2a-c-\alpha)(h_{1,1}-h_3)
        \end{align*} 
        Consequently, Equation (\ref{eq-solutions.in.PicE}) becomes
\begin{equation}
     (2a-c-2\alpha)l=\underline{x}_{d}-h_{a,a}+(2a-c-\alpha)(h_{1,1}-h_3).
\end{equation}
    \end{enumerate}
By hypothesis and by Corollary~\ref{coro-degree.case.real}\eqref{c-d.c.r-c}, the statement follows.
\end{proof}

In the following subsection, we present explicit values of Formulas (\ref{eq-Welschinger.formula.dim.7}),   (\ref{Eq:W.in.deg.6.1}) and~(\ref{Eq:W.in.deg.6.2}), which are listed in Tables~\ref{Tab:W.in.deg.7},~\ref{Tab:W.in.deg.6.1} and~\ref{Tab:W.in.deg.6.2}, respectively.

%%%%%%%%%%%%%%%%%%%%%%%%%%%%%%%%%%%%%%
\subsection{Computations and further comments}

As mentioned in the introduction, explicit computations of genus $0$ Welschinger invariants on real del Pezzo surfaces have been carried out by numerous mathematicians using a variety of techniques, while the methods available for computing Welschinger invariants of threefolds are still relatively limited. Here, we list some first values of the genus $0$ Welschinger invariants of certain threefolds, in particular, 
$$W_{(\R P^3)_1}^{\mathfrak{s}_{(\R P^3)_1}, \mathfrak{o}_{(\R P^3)_1}}((d;k),l), \; W_{(\R P^1)^3}^{\mathfrak{s}_{(\R P^1)^3}, \mathfrak{o}_{(\R P^1)^3}}((a,b,c),l) \quad\text{and}\quad W_{\mathbb{S}^2\times \R P^1}^{\mathfrak{s}_{\mathbb{S}^2\times \R P^1},\mathfrak{o}_{\mathbb{S}^2\times \R P^1}}((a;c),l)$$  in the Tables~\ref{Tab:W.in.deg.7},~\ref{Tab:W.in.deg.6.1} and~\ref{Tab:W.in.deg.6.2}, respectively. These computations have been made using a Maple program.

\subsubsection{Real 3-dimensional del Pezzo varieties of degree 7}
We provide in Table~\ref{Tab:W.in.deg.7} the values of the genus $0$ Welschinger invariants $W_{(\R P^3)_1}^{\mathfrak{s}_{(\R P^3)_1},\mathfrak{o}_{(\R P^3)_1}}((d;k),l)$ for only small values of $d$. We  are under the following condition: $$0\leq l \leq d-\tfrac{1}{2}(2d-1-k) \quad\text{and}\quad k\leq d.$$ 
It is important to note that 
$W_{(\R P^3)_1}^{\mathfrak{s}_{(\R P^3)_1},\mathfrak{o}_{(\R P^3)_1}}((d;k),l)$ vanishes if $d$ is even. Table~\ref{Tab:W.in.deg.7} displays the computations of  genus $0$ Welschinger invariants $W_{(\R P^3)_1}^{\mathfrak{s}_{(\R P^3)_1},\mathfrak{o}_{(\R P^3)_1}}((d;k),l)$ for the $\Spin_3$-structure $\mathfrak{s}_{(\R P^3)_1}$ over $(\R P^3)_1$  such that the spinor state satisfies $\sp_{\mathfrak{s}_{(\R P^3)_1},\mathfrak{o}_{(\R P^3)_1}}(f(\C P^1))=+1$, where $$f_*[\C P^1]=(1;0)\in H_2^{-\tau}(\C P^3\sharp \overline{\C P^3};\Z).$$ 
 Also note that, in Table~\ref{Tab:W.in.deg.7}, we have $0 \leq k\leq d\leq 9$.

The following observations may be established through geometric arguments or deduced directly from Theorem~\ref{Theorem-W.in.deg.7}:
\begin{enumerate}
    \item For every $\Spin_3$-structure on $(\R P^3)_1$, every $d\in \mathbb{N}^*$ and every integer $l$ in $\{0,\ldots, \tfrac{1}{2}(2d-1-k)\}$, one has
    $$W_{(\R P^3)_1}^{\mathfrak{s}_{(\R P^3)_1},\mathfrak{o}_{(\R P^3)_1}}((d;0),l)=-W_{(\R P^3)_1}^{\mathfrak{s}_{(\R P^3)_1},\mathfrak{o}_{(\R P^3)_1}}((d;1),l).$$
    \item  For every $\Spin_3$-structure on $(\R P^3)_1$,  every $d\in \mathbb{N}^*$, every $k\in \{\tfrac{1}{2}(d+1),\ldots,d\}$ and  every integer
$l$ in  $\{0,\ldots, \tfrac{1}{2}(3k+1)\}$, one has
    $$W_{(\R P^3)_1}^{\mathfrak{s}_{(\R P^3)_1},\mathfrak{o}_{(\R P^3)_1}}((d;k),l)=0.$$
    \item  For every $\Spin_3$-structure on $(\R P^3)_1$, every $k\in \mathbb{N}$ and every integer $l$ in  $ \{0,\ldots, \tfrac{1}{2}(3k+1)\}$, one has
    $$W_{(\R P^3)_1}^{\mathfrak{s}_{(\R P^3)_1},\mathfrak{o}_{(\R P^3)_1}}((2k+1;k),l)=(-1)^{\frac{1}{2}(k^2-k)}W_{\R P^2}(k+1,l).$$
\end{enumerate}

%%%%%%%%%%%%%%%%%%%%%%%%%%%%%%%%%%%%%%%%%%%table1

\subsubsection{Real threefold product of the projective line with the standard real structure}

%\newpage
%%%%%%%%%%%%%%%%%%%%%%%%%%%%%%%%%%%%%%table2

One has
$$ W_{(\R P^1)^3}^{\mathfrak{s}_{(\R P^1)^3},\mathfrak{o}_{(\R P^1)^3}}((a,b,c),l)= W_{(\R P^1)^3}^{\mathfrak{s}_{(\R P^1)^3},\mathfrak{o}_{(\R P^1)^3}}((\sigma(a),\sigma(b),\sigma(c)),l),$$ where $\sigma$ is a permutation of the set $\{a,b,c\}$.  Without loss of generality, we may assume that $a \geq b \geq c$ when computing $W_{(\R P^1)^3}^{\mathfrak{s}_{(\R P^1)^3},\mathfrak{o}_{(\R P^1)^3}}((a,b,c),l)$. As a direct consequence of Proposition~\ref{Prop-vanishing.GW} and Remark~\ref{remark12}, we obtain the following proposition.
 
\begin{proposition}\label{Prop-vanishing.W.6.1}
Assume that $a\geq b \geq c >1$. Then, for every $\Spin_3$-structure $\mathfrak{s}_{(\R P^1)^3}$ over $(\R P^1)^3$, the Welschinger invariant
$W_{(\R P^1)^3}^{\mathfrak{s}_{(\R P^1)^3},\mathfrak{o}_{(\R P^1)^3}}((a,b,c),l)$ vanishes if at least one of the following conditions holds:
    \begin{enumerate}
        \item $a> b+c$,
        \item $a+b+c$ is even.
    \end{enumerate}
\end{proposition}

Table~\ref{Tab:W.in.deg.6.1} shows computations of  genus $0$ Welschinger invariants $W_{(\R P^1)^3}^{\mathfrak{s}_{(\R P^1)^3},\mathfrak{o}_{(\R P^1)^3}}((a,b,c),l)$ for the $\Spin_3$-structure $\mathfrak{s}_{(\R P^1)^3}$ over $(\R P^1)^3$  such that the spinor state satisfies $\sp_{\mathfrak{s}_{(\R P^1)^3},\mathfrak{o}_{(\R P^1)^3}}(f(\C P^1))=+1$, where
$$f_*\left[\C P^1\right]\in \{(0,0,1),(0,1,0),(0,0,1)\}\in H_2^{-\tau_{1,1}}\left(\left(\C P^1\right)^3;\Z\right).$$ 
         It should be noted that, in this table, we consider only those integers $a$, $b$ and $c$ satisfying $$7 \geq a \geq b \geq c \geq 0 \quad\text{and}\quad a + b + c \leq 15.$$ Furthermore, we only display the non-vanishing genus $0$ Welschinger invariants $W_{(\R P^1)^3}^{\mathfrak{s}_{(\R P^1)^3},\mathfrak{o}_{(\R P^1)^3}}((a,b,c),l)$. This means that  we display only the cases where $a\leq b+c$ and $a+b+c$ is odd. 
       The computations in Table~\ref{Tab:W.in.deg.6.1}  lead us to propose the following conjecture regarding the positivity of the genus $0$ Welschinger invariants.

       \begin{conjecture}
     Let $\mathfrak{s}_{(\R P^1)^3}$ denote the $\Spin_3$-structure over $(\R P^1)^3$ as in Theorem~\ref{Theorem2-W}.
Then, for all triples $(a,b,c)\in \mathbb{N}^*\times  \mathbb{N}\times \mathbb{N}$ and all integer $l \in \{0,\ldots, \frac{1}{2}(a+b+c-1)\}$, one has
     $$W_{(\R P^1)^3}^{\mathfrak{s}_{(\R P^1)^3},\mathfrak{o}_{(\R P^1)^3}}((a,b,c),l)\geq  0.$$ 
     Moreover,  equality holds if either $a+b+c$ is even or $a>b+c$.
 \end{conjecture}

       As a consequence, Proposition~\ref{Prop-vanishing.W.6.1} can be reformulated as follows:
  $$W_{(\R P^1)^3}^{\mathfrak{s}_{(\R P^1)^3},\mathfrak{o}_{(\R P^1)^3}}((a,b,c),l) = 0 \quad\text{if and only if}\quad\text{either } a+b+c \equiv 0 \mod 2 \text{ or } a>b+c.$$

\subsubsection{Real threefold product of the projective line with the twisted real structure}
The following proposition is a direct consequence of Proposition~\ref{Prop-vanishing.GW} and Remark~\ref{remark13}.  

\begin{proposition}\label{Prop-vanishing.W.6.2}
 For every $\Spin_3$-structure $\mathfrak{s}_{\mathbb{S}^2\times\R P^1}$ over $\mathbb{S}^2\times \R P^1$, the genus $0$ Welschinger invariant\\
$ W_{\mathbb{S}^2 \times \R P^1}^{\mathfrak{s}_{\mathbb{S}^2 \times \R P^1},\mathfrak{o}_{\mathbb{S}^2 \times \R P^1}}((a;c),l)$ vanishes if at least one of the following conditions holds:
    \begin{enumerate}
        \item $c>2a$,
        \item $c$ is even.
            \end{enumerate}
\end{proposition}

\begin{remark}
The values of the genus $0$ Welschinger invariant $W_{(\R P^3)_1}^{\mathfrak{s}_{(\R P^3)_1},\mathfrak{o}_{(\R P^3)_1}}((d;0),l)$, which is presented in Table~\ref{Tab:W.in.deg.7}, coincide with those given by E.~Brugall\'e and P.~Georgieva \cite{brugalle2016pencils}.  
\end{remark}

Table~\ref{Tab:W.in.deg.6.2} displays the computations of  genus $0$ Welschinger invariants $ W_{\mathbb{S}^2 \times \R P^1}^{\mathfrak{s}_{\mathbb{S}^2 \times \R P^1},\mathfrak{o}_{\mathbb{S}^2 \times \R P^1}}((a;c),l)$ for the $\Spin_3$-structure $\mathfrak{s}_{\mathbb{S}^2 \times \R P^1}$ over $\mathbb{S}^2\times \R P^1$   such that the spinor state satisfies 
$$\sp_{\mathfrak{s}_{\mathbb{S}^2 \times \R P^1},\mathfrak{o}_{\mathbb{S}^2 \times \R P^1}}\left(f\left(\C P^1\right)\right)=+1,$$
where $f_*[\C P^1]=(0;1)$ in $H_2^{-\tilde{\tau}_{1,1}}((\C P^1)^3;\Z)$.  It is worth noting that, in this table, we consider only those integers $a$ and $c$ satisfying $a\leq 5$ and $2a+c \leq 19$. Furthermore, we display only the genus $0$ Welschinger invariants $ W_{\mathbb{S}^2 \times \R P^1}^{\mathfrak{s}_{\mathbb{S}^2 \times \R P^1},\mathfrak{o}_{\mathbb{S}^2 \times \R P^1}}((a;c),l)$ in the cases where $c<2a$ and $c$ is odd. As previously discussed,  in these cases, the Welschinger invariants are ``almost'' non-zero by Proposition~\ref{Prop-vanishing.W.6.2}.

 \begin{remark}
The values of the genus $0$ Welschinger invariants
$$W_{(\R P^1)^3}^{\mathfrak{s}_{(\R P^1)^3},\mathfrak{o}_{(\R P^1)^3}}((a,b,c),l) \quad\text{and}\quad   W_{\mathbb{S}^2\times \R P^1}^{\mathfrak{s}_{\mathbb{S}^2\times \R P^1},\mathfrak{o}_{\mathbb{S}^2\times \R P^1}}((a;c),l),$$ which are presented in Tables~\ref{Tab:W.in.deg.6.1} and~\ref{Tab:W.in.deg.6.2}, respectively, coincide with those given by X.~Chen and A.~Zinger (see \cite{chen2021wdvv}).
\end{remark}
\newpage

%%%%%%%%%%%%%%%%%%%%%%%%%%%%%%%%%%%table3
\begin{table}[H]
\centering
     \begin{tabular}{|l||c|c||c|c|c|c||}\hline
\diagbox[innerwidth=1.5cm]{$l$}{$(d;k)$}&
{$(1;0)$}&{$(1;1)$}&{$(3;0)$}&{$(3;1)$}&{$(3;2)$}&
{$(3;3)$}\\
\hline
\hline
0&1 & $-1$&$-1$&$1$&$0$&$0$\\
\hline
1&&&$-1$&1&0&0\\
\hline
2&&&$-1$&1&&\\

\hline
\hline
\end{tabular}

\vspace{1cm}

\centering
     \begin{tabular}{|l||c|c|c|c|c|c||}\hline
\diagbox[innerwidth=1.5cm]{$l$}{$(d;k)$}&{$(5;0)$}&{$(5;1)$}&{$(5;2)$}&{$(5;3)$}&{$(5;4)$}&{$(5;5)$}\\
\hline
\hline
0&45&$-45$&$-8$&0&0&0\\
\hline
1&$29$ &$-29$&$-6$&0&0&0\\
\hline
2&$17$ &$-17$&$-4$&0&0&0\\
\hline
3&$9$ &$-9$&$-2$&0&&\\
\hline
4&$5$ &$-5$&&&&\\
\hline
\hline
\end{tabular}

\vspace{1cm}

\centering
     \begin{tabular}{|l||c|c|c|c|c|c|c|c||}\hline
\diagbox[innerwidth=1.5cm]{$l$}{$(d;k)$}&
{$(7;0)$}&{$(7;1)$}&{$(7;2)$}&{$(7;3)$}&{$(7;4)$}&{$(7;5)$}&{$(7;6)$}&{$(7;7)$}\\
\hline
\hline
0&$-14589$& $14589$&$3816$&$-240$&$0$ &$0$&0&0\\
\hline
1& $-6957$ & $6957$ & $1932$ &$-144$&0&$0$&$0$ &0\\
\hline
2& $-3093$ & $3093$ & $912$ &$-80$&0&$0$&$0$ &0\\
\hline
3& $-1269$ & $1269$ & $396$ &$-40$&0&$0$&$0$ &0\\
\hline
4& $-477$  & $477$ & $152$ &$-16$&0&$0$& &\\
\hline
5& $-173$ & $173$ & $44$ &$0$&&& &\\
\hline
6& $-85$ & $85$ & &&&& &\\
\hline
\hline
\end{tabular}

\vspace{1cm}
\centering
     \begin{tabular}{|l||c|c|c|c|c|c|c|c|c|c||}\hline
\diagbox[innerwidth=10.3mm]{$l$}{$\mkern8mu(d;k)\mkern-8mu$}&
{$(9;0)$}&{$(9;1)$}&{$(9;2)$}&{$(9;3)$}&{$(9;4)$}&{$(9;5)$}&{$(9;6)$}&{$(9;7)$}&{$(9;8)$}&{$(9;9)$}\\
\hline
\hline
0&$17756793$& $-17756793$&$-5519664$&$603840$&$18264$ &$0$&0&0&0&0\\
\hline
1& $6717465$ & $-6717465$ & $-2155050$ &$256080$&9096&$0$&$0$ &0&0&0\\
\hline
2& $2407365$ & $-2407365$ & $-797604$ &$102912$&4272&$0$&$0$ &0&0&0\\
\hline
3& $812157$ & $-812157$ & $-278046$ &$38880$&1872&$0$&$0$ &0&0&0\\
\hline
4& $256065$  & $-256065$ & $-90392$ &$13568$&744&$0$&0 &0&0&0\\
\hline
5& $75281$ & $-75281$ & $-27058$ &$4208$&248&0&0 &0&&\\
\hline
6& $21165$ & $-21165$ &$-7500$ &1088&64&0& &&&\\
\hline
7& $6165$ & $-6165$ & $-2086$&256&&& &&&\\
\hline
8& $1993$ & $-1993$ & &&&& &&&\\
\hline
\hline
\end{tabular}
\caption{Genus $0$ Welschinger invariants $W_{(\R P^3)_1}^{\mathfrak{s}_{(\R P^3)_1},\mathfrak{o}_{(\R P^3)_1}}((d;k),l)$  of real $3$-dimensional del Pezzo varieties of degree $7$  using Formula (\ref{eq-Welschinger.formula.dim.7})  with $d\leq 9$.} %
\label{Tab:W.in.deg.7}
\end{table}

\begin{table}[H]
\centering
     \begin{tabular}{|l||c||c||c||c|c||}\hline
\diagbox[innerwidth=2cm]{$l$}{$(a,b,c)$}&
{$(1,0,0)$}&{$(1,1,1)$}&{$(2,2,1)$}&{$(3,2,2)$}&{$(3,3,1)$}\\
\hline
\hline
0&$1$& $1$&$1$&$8$&$1$ \\
\hline
1&& $1$&$1$&$6$&$1$ \\
\hline
2&&&$1$&$4$&$1$ \\
\hline
3&&&&$2$&$1$ \\
\hline
\hline
\end{tabular}

\vspace{1cm}

\centering
     \begin{tabular}{|l||c|c|c||c|c|c|c||}\hline
\diagbox[innerwidth=2cm]{$l$}{$(a,b,c)$}&
{$(3,3,3)$}&{$(4,3,2)$}&{$(4,4,1)$}&
{$(4,4,3)$}&{$(5,3,3)$}&{$(5,4,2)$}&{$(5,5,1)$}\\
\hline
\hline
0&$216$&48&1&$3864$& $1086$&$256$&$1$\\
\hline
1&$126$&32&1&$1980$& $606$&$160$&1\\
\hline
2&$68$&20&1&$960$&$318$&$96$&1\\
\hline
3&$34$&12&1&$444$&$158$&$56$&1\\
\hline
4&$16$&8&1&$200$&$78$&$32$&1\\
\hline
5&&&&$92$&$46$&$16$&1\\
\hline
\hline
\end{tabular}

\vspace{1cm}

\centering
     \begin{tabular}{|l||c|c|c|c|c||}\hline
\diagbox[innerwidth=2cm]{$l$}{$(a,b,c)$}&{$(5,4,4)$}&{$(5,5,3)$}&{$(6,4,3)$}&{$(6,5,2)$}&{$(6,6,1)$}\\
\hline
\hline
0&$174360$ &57608&18424&1280&1\\
\hline
1&$77064$ &27276&9256&768&1\\
\hline
2&$32496$ &12400&4432&448&1\\
\hline
3&$13136$ &5444&2032&256&1\\
\hline
4&5160&2328&904&144&1\\
\hline
5&2040&$988$&408&80&1\\
\hline
6&896 &$448$&224&48&1\\
\hline
\hline
\end{tabular}

\vspace{1cm}

\centering
     \begin{tabular}{|l||c|c|c|c|c|c|c||}\hline
\diagbox[innerwidth=2cm]{$l$}{$(a,b,c)$}&
{$(5,5,5)$}&{$(6,5,4)$}&{$(6,6,3)$}&{$(7,4,4)$}&{$(7,5,3)$}&{$(7,6,2)$}&{$(7,7,1)$}\\
\hline
\hline
0&15253434&5998848&773808&819200&268575&6144&1\\
\hline
1&5840298&2410496&347202&360896&125855&3584&1\\
\hline
2&2148978&931712&151028&152192&56831&2048&1\\
\hline
3&762450&347520&63958&61568&24831&1152&1\\
\hline
4&262842&125952&26488&24064&10559&640&1\\
\hline
5&89418&45056&10794&9280&4415&352&1\\
\hline
6&31122&16512&4412&3712&1887&192&1\\
\hline
7&12690&6784&1982&1536&991&96&1\\
\hline
\hline
\end{tabular}

\caption{Non-vanishing genus $0$ Welschinger invariants  $W_{(\R P^1)^3}^{\mathfrak{s}_{(\R P^1)^3},\mathfrak{o}_{(\R P^1)^3}}((a,b,c),l)$ of real $3$-dimensional del Pezzo varieties of degree $6$, whose real part is $\rprprp$, using Formula (\ref{Eq:W.in.deg.6.1}) with $a+b+c\leq 15$.} 
\label{Tab:W.in.deg.6.1}
\end{table}

 \begin{table}[H]
\centering
     \begin{tabular}{|l||c||c|c||c|c|c||}\hline
\diagbox{$l$}{$(a;c)$}&
{$(1;1)$}&{$(2;1)$}&{$(2;3)$}&{$(3;1)$}&{$(3;3)$}&{$(3;5)$}\\
\hline
\hline
0&$-1$&1&6&$-1$&$-120$&$-576$\\
\hline
1&$-1$&1&4&$-1$&$-62$&$-288$\\
\hline
2&&1&2&$-1$&$-28$&$-128$\\
\hline
3&&&0&$-1$&$-10$&$-48$\\
\hline
4&&&&&$0$&$-16$\\
\hline
5&&&&&&$-16$\\

\hline
\hline
\end{tabular}

\vspace{1cm}

\centering
     \begin{tabular}{|l||c|c|c|c||}\hline
\diagbox{$l$}{$(a;c)$}&
{$(4;1)$}&{$(4;3)$}&{$(4;5)$}&{$(4;7)$}\\
\hline
\hline
0&1&1692&61704&294336\\
\hline
1&1&768&24048&116352\\
\hline
2&1&324&8592&42624\\
\hline
3&1&128&2760&14208\\
\hline
4&1&44&792&4224\\
\hline
5&&0&224&1152\\
\hline
6&&&96&320\\
\hline
7&&&&$-256$\\
\hline
\hline
\end{tabular}

\vspace{1cm}

\centering
     \begin{tabular}{|l||c|c|c|c|c||}\hline
\diagbox{$l$}{$(a;c)$}&
{$(5;1)$}&{$(5;3)$}&{$(5;5)$}&{$(5;7)$}&{$(5;9)$}\\
\hline
\hline
0&$-1$&$-20760$&$-3792888$&$-97970688$&$-493848576$\\
\hline
1&$-1$&$-8764$&$-1271016$&$-31204224$&$-160966656$\\
\hline
2&$-1$&$-3536$&$-399216$&$-9331200$&$-49582080$\\
\hline
3&$-1$&$-1380$&$-117072$&$-2593536$&$-14303232$\\
\hline
4&$-1$&$-520$&$-32056$&$-663936$&$-3821568$\\
\hline
5&$-1$&$-172$&$-8136$&$-157344$&$-938496$\\
\hline
6&&0&$-1872$&$-35776$&$-215040$\\
\hline
7&&&$-1104$&$-7840$&$-47872$\\
\hline
8&&&&$-4096$&$-10496$\\
\hline
9&&&&&$-26880$\\
\hline
\hline
\end{tabular}

\caption{Genus $0$ Welschinger invariants $W_{\mathbb{S}^2\times \R P^1}^{\mathfrak{s}_{\mathbb{S}^2\times \R P^1},\mathfrak{o}_{\mathbb{S}^2\times \R P^1}}((a;c),l)$ of real $3$-dimensional del Pezzo varieties of degree $6$, whose real part is $\ssrp$,  using Formula (\ref{Eq:W.in.deg.6.2}) with $a \leq 5$.} 
\label{Tab:W.in.deg.6.2}
\end{table}

%%%%%%%%%%%%%%%%%%%%%
% References
%%%%%%%%%%%%%%%%%%%%%
 
\newcommand{\etalchar}[1]{$^{#1}$}

\end{document}